\documentclass[12pt,reqno]{amsart}
\usepackage[margin=0.85in]{geometry}
\usepackage{amsmath,amssymb,amsthm,graphicx,amsxtra, setspace}
\usepackage[utf8]{inputenc}
\usepackage{mathrsfs}
\usepackage{hyperref}
\usepackage{upgreek}
\usepackage{mathtools}
\usepackage[mathcal]{euscript}
\usepackage{xcolor}
\allowdisplaybreaks

\usepackage{fouridx}
\usepackage{todonotes}

\usepackage[pagewise]{lineno}

\usepackage{graphicx,eurosym}
\usepackage{hyperref}
\usepackage{mathtools}

\usepackage[cyr]{aeguill}

\colorlet{darkblue}{blue!50!black}

\hypersetup{
	colorlinks,%
	citecolor=blue,%
	filecolor=red,%
	linkcolor=red,%
	urlcolor=blue,%
	pdfnewwindow=true,%
	pdfstartview={FitH}
}

\newtheorem{theorem}{Theorem}[section]
\newtheorem{lemma}[theorem]{Lemma}
\newtheorem{proposition}[theorem]{Proposition}
\newtheorem{assumption}[theorem]{Assumption}
\newtheorem{corollary}[theorem]{Corollary}
\newtheorem{definition}[theorem]{Definition}

\newtheorem{remark}[theorem]{Remark}

\allowdisplaybreaks

\let\originalleft\left
\let\originalright\right
\renewcommand{\left}{\mathopen{}\mathclose\bgroup\originalleft}
\renewcommand{\right}{\aftergroup\egroup\originalright}

\renewcommand{\d}{\/\mathrm{d}\/}

\def\w{\textbf{W}^{\varepsilon}_{{\theta}^{\varepsilon}}}

\def\L{\mathbb{L}}
\def\A{\mathrm{A}}

\def\I{\mathrm{I}}

\def\C{\mathrm{C}}
\def\f{\boldsymbol{f}}

\def\B{\mathrm{B}}
\def\D{\mathrm{D}}
\def\O{\mathcal{O}}

\def\E{\mathbb{E}}

\def\x{\boldsymbol{x}}

\def\z{\boldsymbol{z}}
\def\v{\boldsymbol{v}}
\def\V{\mathbb{v}}
\def\w{\boldsymbol{w}}
\def\W{\mathrm{W}}

\def\N{\mathbb{N}}

\def\V{\mathbb{V}}
\def\wi{\widetilde}

\def\u{\mathrm{U}}
\def\P{\mathrm{P}}
\def\u{\boldsymbol{u}}
\def\H{\mathbb{H}}

\newcommand{\R}{\mathbb{R}}

\renewcommand{\d}{\/\mathrm{d}\/}

\newcommand\dela[1]{}

\newcommand{\Addresses}{{% additional braces for segregating \footnotesize
		\footnote{
			\noindent \textsuperscript{1}Tata Institute of Fundamental Research - Centre For Applicable Mathematics (TIFR-CAM), Bangalore 560065, Karnataka, India.\par\nopagebreak
			
			\noindent \textsuperscript{2}Center for Mathematics and Applications (NovaMath), Universidade Nova de Lisboa, Portugal. \par\nopagebreak
			
			\noindent \textsuperscript{3}Department of Mathematics, Indian Institute of Technology Roorkee-IIT Roorkee,
			Haridwar Highway, Roorkee, Uttarakhand 247667, India.\par\nopagebreak
			\noindent  
			\textit{e-mail:} \texttt{Kush Kinra:  kushkinra@gmail.com.}
			
			\textit{e-mail:} \texttt{Manil T. Mohan: maniltmohan@ma.iitr.ac.in, maniltmohan@gmail.com.}
			
			\noindent \textsuperscript{*}Corresponding author.
			
			\textit{Key words:} Stochastic globally modified Navier-Stokes equations, unbounded Poincar\'e domains, energy equality, cylindrical Wiener process, invariant measures.
			
			Mathematics Subject Classification (2020): Primary 35B41, 35Q35; Secondary 37L55, 37N10, 35R60.

}}}

\begin{document}
	%	\linenumbers
	
	\title[Random attractors for 3D SGMNS equations]{Random dynamics of solutions for three-dimensional stochastic globally modified Navier-Stokes equations on unbounded Poincar\'e domains
		\Addresses}
	\author[K. Kinra and M. T. Mohan]
	{Kush Kinra\textsuperscript{1,2*} and Manil T. Mohan\textsuperscript{3}}

	\maketitle
	
	\begin{abstract}
		In this article, we consider a novel version of three-dimensional (3D) globally modified Navier-Stokes (GMNS) system introduced by \emph{Caraballo et. al., Adv. Nonlinear Stud.,  \textbf{6} (2006),  411--436} (\cite{Caraballo+Real+Kloeden_2006}), which is very significant from the perspective of deterministic as well as stochastic partial differential equations. Our focus is on examining a stochastic version of the suggested 3D GMNS equations that are perturbed  by an infinite-dimensional additive noise. We can consider a rough additive noise (Lebesgue space valued) with this model, which is not appropriate to consider with the system presented in \emph{Caraballo et. al., Adv. Nonlinear Stud.,  \textbf{6} (2006),  411--436} (\cite{Caraballo+Real+Kloeden_2006}).  One of the technical problems associated with the rough noise is overcome by the use of the corresponding Cameron–Martin (or reproducing kernel Hilbert) space. This article aims to accomplish three objectives.   Firstly, we establish the existence and uniqueness of weak solutions (in the analytic sense) of the underlying stochastic system using a Doss-Sussman transformation and a primary ingredient Minty-Browder technique. Secondly, we demonstrate the existence of random attractors for the underlying stochastic system in the natural space of square integrable divergence-free functions. Finally, we show the existence of an invariant measure for the underlying stochastic system for any  viscosity coefficient $\nu>0$ and uniqueness of invariant measure for sufficiently large $\nu$ by using the exponential stability of solutions. A validation of the considered version of 3D GMNS equations has also been discussed in the appendix by establishing that the sequence of weak solutions of 3D GMNS equations converges to a weak solution of 3D Navier-Stokes equations as the modification parameter goes to infinity.
	\end{abstract}

	\section{Introduction} \label{sec1}\setcounter{equation}{0}
		%	\subsection{Literature survey and motivations} 	
	One of the seven Millennium Prize Problems announced by the Clay Mathematics Institute in May 2000  is the existence and smoothness of Navier–Stokes equations (NSE) \cite{Fefferman_2006-Clay}. The problem is to prove or give a counter-example of the following statement:   Given an initial velocity field in three dimensions of space and time, there is a smooth, globally defined vector velocity and scalar pressure field that solve NSE on a torus $\mathbb{T}^3=\mathbb{R}^3\backslash\mathbb{Z}^3$ or whole space $\mathbb{R}^3$ . Despite being defined  $\mathbb{T}^3$ or $\mathbb{R}^3$, the above mentioned Millennium Prize Problem can be applied to any domain. Let $\mathcal{O}$ be an open connected subset of $\R^3$, and consider the following 3D NSE: 
			\begin{equation}\label{3D-NSE}
				\left\{
				\begin{aligned}
					\frac{\partial \u}{\partial t}-\nu \Delta\u+(\u\cdot\nabla)\u+\nabla p&=\boldsymbol{f}, &&\text{ in }\  \mathcal{O}\times(0,\infty), \\ \nabla\cdot\u&=0, && \text{ in } \ \ \mathcal{O}\times[0,\infty), \\ \u&=\boldsymbol{0}, && \text{ on } \ \ \partial\mathcal{O}\times[0,\infty), \\
					\u(0)&=\x, && \ x\in \mathcal{O},
				\end{aligned}
				\right.
			\end{equation}
			where  $\u(x,t) \in \R^3$, $p(x,t)\in\R$ and $\f(x)\in \R^3$ denote the velocity field, pressure and external forcing, respectively, and the positive constant $\nu$ represents the \emph{kinematic viscosity} of the fluid. Apart from the existence of global regular solutions, the uniqueness of Leray-Hopf weak solutions and the existence of global strong solutions are still open problems. Several mathematicians modified the traditional 3D NSE and posed questions about these models' global solvability, see for example  \cite{Constantin_2003,Holst+Lunasin+Tsogtgerel_2010,Rieusset_2002}, etc.

			In 2006, Caraballo et. al. (\cite{Caraballo+Real+Kloeden_2006}) introduced a modified version of 3D NSE by replacing the nonlinear term $(\u\cdot\nabla)\u$ with $F_{N}(\|\u\|_{\H^1_0(\mathcal{O})})\left[(\u\cdot\nabla)\u\right]$, where $\|\u\|_{\H^1_0(\mathcal{O})}:=\|\nabla\u\|_{\mathbb{L}^2(\mathcal{O})}$ and  the function $F_N:(0,+\infty)\to (0,1]$ is defined by
			\begin{align}\label{FN}
				F_{N}(r)=\min \left\{1,\frac{N}{r}\right\}, \;\;\; r\in(0,+\infty),
			\end{align}
			and $N>0$ is a given constant. In particular, Caraballo et. al. (\cite{Caraballo+Real+Kloeden_2006}) considered the following system:
			\begin{equation}\label{GMNSE-CRK}
				\left\{
				\begin{aligned}
					\frac{\partial \u}{\partial t}-\nu \Delta\u+F_{N}(\|\u\|_{\H^1_0(\mathcal{O})})\left[(\u\cdot\nabla)\u\right]+\nabla p&=\boldsymbol{f}, &&\text{ in }\  \mathcal{O}\times(0,\infty), \\ \nabla\cdot\u&=0, && \text{ in } \ \ \mathcal{O}\times[0,\infty), \\ \u&=\boldsymbol{0}, && \text{ on } \ \ \partial\mathcal{O}\times[0,\infty), \\
					\u(0)&=\x, && \ x\in \mathcal{O}.
				\end{aligned}
				\right.
			\end{equation}
			 The system \eqref{GMNSE-CRK} is indeed a globally modified version of 3D NSE since the modifying factor $F_{N}(\|\u\|_{\H^1_0(\mathcal{O})})$ depends on the norm $\|\u\|_{\H^1_0(\mathcal{O})}$, which in turn depends on $\u$ over the whole domain $\mathcal{O}$ and not just at or near the point $x \in \mathcal{O}$ under consideration. Essentially, it prevents large gradients dominating the dynamics and leading to explosions. It violates the basic laws of mechanics, but mathematically, the system \eqref{GMNSE-CRK} is a well-defined system of equations, just like the modified versions of  NSE of Leray and others with different mollifications of the nonlinear term, for more details, see the review paper of Constantin \cite{Constantin_2003}. In addition, we would also like to mention that Flandoli and Maslowski \cite{Flandoli+Maslowski_1995} used a global cut off function involving the $\D(\A^{\frac14})$ norm for the 2D stochastic NSE, where $\A$ is the Stokes operator defined on $\mathcal{O}$ (see Subsection \ref{LO} below). Moreover, for the case of 2D stochastic NSE, an $\mathbb{L}^4$-cut off the nonlinear operator is considered in \cite[Lemma 2.4]{Menaldi+Sritharan_2002} for establishing local monotonicity results (see \cite{Chueshov+Millet_2010,Liu+Rockner_2010,Rockner+Shang+Zhang_2024}, etc. for further generalizations). 
			
			\subsection{Underlying model and motivation} In this article, we consider a novel version of 3D GMNS system introduced by Caraballo et. al. \cite{Caraballo+Real+Kloeden_2006}, which is very significant from the perspective of stochastic partial differential equations (SPDEs), see the discussion after the stochastic system \eqref{SGMNSE}. In particular, we consider the following modified version of 3D NSE by replacing the nonlinear term $(\u\cdot\nabla)\u$ with $F_{N}(\|\u\|_{\L^4(\mathcal{O})})\left[(\u\cdot\nabla)\u\right]$:
			\begin{equation}\label{2}
				\left\{
				\begin{aligned}
					\frac{\partial \u}{\partial t}-\nu \Delta\u+F_{N}(\|\u\|_{\L^4(\mathcal{O})})\left[(\u\cdot\nabla)\u\right]+\nabla p&=\boldsymbol{f}, &&\text{ in }\  \mathcal{O}\times(0,\infty), \\ \nabla\cdot\u&=0, && \text{ in } \ \ \mathcal{O}\times[0,\infty), \\ \u&=\boldsymbol{0}, && \text{ on } \ \ \partial\mathcal{O}\times[0,\infty), \\
					\u(0)&=\x, && \ x\in \mathcal{O}.
				\end{aligned}
				\right.
			\end{equation}
		   In this study, our focus is on the well-posedness and large time behavior of the following stochastic counterpart of the system \eqref{2} which is the main motivation for considering a novel version of 3D GMNS equations:
			\begin{equation}\label{SGMNSE}
				\left\{
				\begin{aligned}
					\d\u+\{-\nu \Delta\u+F_{N}(\|\u\|_{\L^4(\mathcal{O})})\left[(\u\cdot\nabla)\u\right]+\nabla p\}\d t&=\boldsymbol{f}\d t + \d\mathrm{W}, && \text{ in } \ \mathcal{O}\times(0,\infty), \\ \nabla\cdot\u&=0, && \text{ in } \ \mathcal{O}\times[0,\infty), \\
					\u&=\boldsymbol{0}, && \text{ on } \ \partial\mathcal{O}\times[0,\infty), \\
					\u(0)&=\boldsymbol{x}, && \text{ in } \ \mathcal{O},
				\end{aligned}
				\right.
			\end{equation} 
		 where $\{\W (t)\}_{t \in \R}$ is a two-sided cylindrical Wiener process in $\H$ (see Subsection \ref{FnO} below for the function spaces) with its Cameron-Martin space/Reproducing Kernel Hilbert Space (RKHS) $\mathrm{K}$ satisfying Assumption \ref{assump1} below, defined on some probability space $(\Omega,\mathcal{F},\mathbb{P})$. For 2D stochastic NSE on unbounded Poincar\'e domains, the rough noise has been considered by Brz\'ezniak et. al. in the works \cite{BCLLLR} and \cite{BL}. 
		 
		 \subsubsection{Impact of underlying model as a deterministic system} Note that the modifying factor $F_{N}(\|\u\|_{\L^4(\mathcal{O})})$ depends on the $\L^4$-norm of $\u$. Essentially, it prevents large $\L^4(\mathcal{O})$-norm (instead of the $\H^1_0$-norm in contrast to the system \eqref{GMNSE-CRK})  dominating the dynamics and leading to explosions. It appears to be a better modification of  3D NSE in comparison to the system \eqref{GMNSE-CRK} since  $$\|\u\|_{\L^4(\mathcal{O})}\leq 2^{1/2}\|\u\|_{\L^2(\mathcal{O})}^{1/4}\|\u\|_{\H^1_0(\mathcal{O})}^{3/4}\leq C \|\u\|_{\H^1_0(\mathcal{O})},$$ by an application of Ladyzhenskya's (Remark \ref{rem2.5}) and Poincar\'e's (Assumption \ref{assumpO}) inequalities. We also addressed this new version of 3D GMNS equations in our recent work \cite[Appendix A]{My+Hang+Kinra+Mohan+Nguyen_Arxiv}.
		 
		 \subsubsection{Impact of underlying model as a stochastic system} The work of Brz\'ezniak et. al.  \cite{BCLLLR,BL} is the main inspiration of this subsection. It is well-known that the model closer to reality is one with rougher noise. Landau et. al. in their fundamental work \cite[Chapter 17]{Landau+Lifshitz_1987} proposed to study NSE under additional stochastic small fluctuations. Consequently, the authors considered the classical balance laws for mass, energy and momentum forced by a random noise, to describe the fluctuations, in particular, local stresses and temperature, which are not related to the gradient of the corresponding quantities. In \cite[Chapter 12]{Landau+Lifshitz_1968}, the same authors  derived correlations for the random forcing by following the general theory of fluctuations. They impose, among other things, the constraint that the noise is either spatially uncorrelated or a little correlated. The $\L^2(\mathcal{O})$-valued Wiener process with RKHS $\L^2(\mathcal{O})$  corresponds to spatially uncorrelated noise. If the Sobolev space $\H^{s}(\mathcal{O})$ ($s>0$) is the RKHS of a $\L^2(\mathcal{O})$-valued Wiener process, then the smaller $s$ implies  less correlated noise. In other words, spatially  less regular noise  is  lesser correlated. Note that a finite-dimensional Brownian motion is included as a special case in our Wiener process.

			\subsection{Literature related to random attractors}		
			In mechanics and mathematical physics, one of the most important and comprehensive fields is to examine the asymptotic behavior of dynamical systems. In the realm of deterministic infinite-dimensional dynamical systems theory, the notion of attractors has a pivotal role (see \cite{Robinson2,R.Temam}, etc.). Finding validation that an SPDE produces a random dynamical system (RDS) or stochastic flow is a fundamental question in the study of random dynamics of SPDEs. The generation of RDS for It\^o stochastic ordinary differential equations and a broad family of PDEs with stationary random coefficients are well-known in the literature (see \cite{Arnold,PEK}, etc.).  The analysis of infinite-dimensional RDS is also an essential branch in the study of qualitative properties of SPDEs (see the works \cite{BCF,CF,Crauel}, etc.).

			In bounded domains, the existence of a random attractor is heavily based on the fact that the embedding $\V\hookrightarrow \H$ is compact which makes analysis easier (see \cite{CDF,KM3}, etc.). However, the embedding $\V\hookrightarrow \H$ is no longer compact in the case of unbounded domains. Therefore, we are not able to prove the existence of random attractors using the compactness criterion. In the deterministic case, this difficulty (in unbounded domains) was resolved by different methods, cf.  \cite{Abergel,Ghidaglia,Rosa}, etc., for the autonomous case and \cite{CLR1,CLR2}, etc., for the non-autonomous case. For SPDEs, the methods available in the deterministic case have also been generalized by several authors (see for example, \cite{BLL,BLW, BL, Wang}, etc.). Later, this concept has been  utilized to prove the existence of random attractors for several SPDEs like 1D stochastic lattice differential equation \cite{BLL}, stochastic NSE on the 2D unit sphere \cite{BGT}, stochastic $g$-NSE \cite{FY,LL,LXS}, stochastic non-autonomous Kuramoto-Sivashinsky equations \cite{LYZ}, stochastic heat equations in materials with memory on thin domains \cite{SLHZ}, stochastic reaction-diffusion equations \cite{BLW,Slavik}, 3D stochastic Benjamin-Bona-Mahony equations \cite{Wang}, 2D and 3D stochastic convective Brinkman-Forchheimer equations (CBFE) \cite{Kinra+Mohan_2024_DCDS,KM7}, etc., and references therein.  In unbounded domains, the commonly used methods in literature are the following:
			\begin{itemize}
				\item [1.] the energy equality method introduced by Ball \cite{Ball};
				\item [2.] the uniform-tail estimate method introduced by Wang \cite{Wang_1999}.
			\end{itemize}
			Note that unlike Navier-Stokes equations (see \cite{BL,BCLLLR}) and convective Brinkman-Forchheimer equations (see \cite{Kinra+Mohan_2023_DIE,Kinra+Mohan_2024_DCDS}), we are not able to show that solutions of the stochastic GMNS system \eqref{SGMNSE} satisfies the \textit{weak continuity} property with respect to the initial data, that is,
			$$\x_n\xrightharpoonup{w}\x\ \text{ in }\ \H \Rightarrow \u_n(t)\xrightharpoonup{w}\u(t)\ \text{ in }\ \H,$$ where $\u_n(\cdot)$ and $\u(\cdot)$ are unique weak solutions (in the analytic sense) of the problem  \eqref{S-GMNSE} with the initial datum $\x_n$ and $\x$, respectively. 
			 Due to this reason, we are not able to use the energy equality method to obtain the asymptotic compactness of RDS associated with the system \eqref{SGMNSE}. Therefore, we explore the uniform-tail estimate method to prove the asymptotic compactness (Theorem \ref{Main_theorem_1}).
	\begin{remark}
	In the context of attractor when $\O$ is a bounded domain, one can consider $\f\in \H^{-1}(\mathcal{O})$. But for unbounded domains case, we need to restrict the forcing term to be in a regular space, that is, $\f\in \L^{2}(\mathcal{O})$. If one proves the weak continuity of solutions with respect to the initial data of stochastic GMNSE equations, it will be possible to work with the same regularity as in the case of bounded domains via the energy equality method (see the works \cite{BCLLLR,Kinra+Mohan_2024_DCDS}, etc. for Navier-Stokes equations and related models). Unfortunately, we are not able to prove the weak continuity of solutions of stochastic GMNSE equations with respect to the initial data, and hence the method of energy equality is not applied.
\end{remark}

	\subsection{Scope of the article, approaches and assumptions}
Broadly, this article  is divided into three main parts.  Our first aim is to establish the existence and uniqueness of weak solutions (in the analytic sense) of the stochastic globally modified Navier-Stokes (SGMNS) equations \eqref{SGMNSE}. We begin by defining an Ornstein-Uhlenbeck process which takes values in $\H \cap {\L}^{4}(\mathcal{O})$. In order to define a suitable Ornstein-Uhlenbeck process, one of the technical difficulties related to the rough noise is resolved by using the corresponding  RKHS.  In particular, motivated from  \cite{BCLLLR,BL} for 2D NSE on unbounded Poincar\'e domains, we assume that the RKHS $\mathrm{K}$ satisfies the following assumption:
	\begin{assumption}\label{assump1}
		$ \mathrm{K} \subset \H \cap {\L}^{4}(\mathcal{O}) $ is a Hilbert space such that for some $\delta\in (0, 1/2),$
		\begin{align}\label{A1}
			\A^{-\delta} : \mathrm{K} \to \H \cap  {\L}^{4}(\mathcal{O})  \  \text{ is }\ \gamma \text{-radonifying.}
		\end{align}
	\end{assumption}
	\begin{remark}
		1. Let us fix $p\in(1,\infty)$. Let $(X_i, \mathcal{A}_i, \nu_i),\ i=1,2,$ be $\sigma$-finite measure spaces. A bounded linear operator $R:\mathrm{L}^2(X_1)\to\mathrm{L}^p(X_2)$ is $\gamma$-radonifying if and only if  there exists  a measurable function $\kappa:X_1\times X_2\to \R$ such that $\int_{X_2}\big[\int_{X_1}|\kappa(x_1,x_2)|^2\d\nu_1(x_1)\big]^{p/2}\d\nu_2(x_2)<\infty,$ and for all $\nu_2$-almost  all $x_2\in X_2, (R(f))(x_2)=\int_{X_1} \kappa(x_1,x_2)f(x_1)\d\nu_1(x_1), f\in \mathrm{L}^2(X_1)$ (cf. \cite[Theorem 2.3]{BN}). Thus, it can be easily seen that if $\mathcal{O}$ is a bounded domain, then $\A^{-s}:\H\to\widetilde{\L}^p$ is $\gamma$-radonifying if and only if $\int_{\mathcal{O}}\big[\sum_{j} \lambda_j^{-2s}|e_j(x)|^2\big]^{p/2}\d x<\infty,$ where $\{e_j\}_{j\in\N}$ is an orthogonal basis of $\H$ and $\A e_j=\lambda_j e_j, j\in \N.$  On 3D bounded domains, we know that $\lambda_j\sim j^{2/3},$ for large $j$ (growth of eigenvalues, see \cite{FMRT}) and hence $\A^{-s}$ is Hilbert-Schmidt if and only if $s>\frac{3}{4}.$ In other words, with $\mathrm{K}=\D(\A^{s}),$ the embedding $\mathrm{K}\hookrightarrow\H\cap {\L}^{4}(\mathcal{O})$ is $\gamma$-radonifying if and only if $s>\frac{3}{4}.$ Thus, Assumption \ref{assump1} is satisfied for any $\delta>0.$ In fact, the condition \eqref{A1} holds if and only if the operator $\A^{-(s+\delta)}:\H\to \H\cap {\L}^{4}(\mathcal{O})$ is $\gamma$-radonifying.

		2. The requirement of  $\delta<\frac{1}{2}$ in Assumption \ref{assump1} is necessary because we need (see Subsection \ref{O-Up} below) the corresponding Ornstein-Uhlenbeck process to take values in $\H\cap {\L}^{4}(\mathcal{O})$.
	\end{remark}

	 Making use of the Ornstein-Uhlenbeck process, we define a Doss-Sussman transformation \cite{Doss_1977,Sussmann_1978} (see \eqref{D-S_trans} below) and obtain an equivalent pathwise deterministic system (see the system \eqref{csgmnse} below) to the system \eqref{SGMNSE}. Next, we show the existence of a unique weak solution to the equivalent pathwise deterministic system \eqref{csgmnse} via a monotonicity argument and a Minty-Browder technique. The following monotonicity result of the operator $\mathcal{G}_{N}(\cdot)+ \frac{ 7^{7}\cdot N^8}{2^{13}\cdot \nu^{7}} \I:=\nu\A\cdot+\B_{N}(\cdot+\mathfrak{Z})+ \frac{ 7^{7}\cdot N^8}{2^{13}\cdot \nu^{7}} \I$ plays a crucial role: 
	 \begin{align*}
	 	\left< \mathcal{G}_{N}(\v_1)-\mathcal{G}_{N}(\v_2), \v_1-\v_2   \right> + \frac{ 7^{7}\cdot N^8}{2^{13}\cdot \nu^{7}} \|\v_1-\v_2\|_{\H}^2 \geq 0,
	 \end{align*}
	 for any  $\v_1,\v_2\in\V.$ The global solvability of the system \eqref{csgmnse} helps to establish the existence and uniqueness of weak  the solution to the system \eqref{SGMNSE}. 
	 We point out here that the earlier papers \cite{Caraballo+Kloeden_2013,Caraballo+Real+Kloeden_2006,Caraballo+Real+Kloeden_2010} discussed the solvability results of the problem \eqref{GMNSE-CRK} by taking initial data $\u_0\in\V$ and using the Faedo-Galerkin approximation technique, compactness arguments and the density of $\V\hookrightarrow\H$. Taking the initial data in $\H$ does not immediately allow us to pass to the limit in the approximated system satisfied by $\u^n$, especially in the term involving  $F_N(\|\u^n\|_{\V})$. Recently, in the case of bounded domains,  the authors in \cite{Anh+Thanh+Tuyet_2023} used a monotonicity argument and a Minty-Browder technique to overcome this problem. But if we consider the model \eqref{SGMNSE} in bounded domains, compactness arguments still work and one can pass to the limit in the approximated system (see Remark \ref{SGMNSE} below). As we are working on unbounded domains, it appears to us that the method explained in \cite[Remark 3.2, page 196]{Temam} may not be useful and we need to use a Minty-Browder technique to obtain the global solvability results for the model \eqref{SGMNSE} whenever the initial data is in $\H$.
	 
	 \begin{remark}
	 It is remarkable to mention here that in order to employ a Doss-Sussman transformation, one must take into account the smooth infinite-dimensional additive noise (RKHS in $\V$) if one analyzes the stochastic counterpart of the system \eqref{GMNSE-CRK}. This is one of the major advantages in considering the modified version \eqref{SGMNSE}. 
	 \end{remark}

	 Secondly, we are interested in the large-time behavior of the solution to the SGMNS equations \eqref{SGMNSE}. Similar to the case of 2D NSE, it is a challenging problem to study the large time behavior of SGMNS equations on the whole space $\R^3$. Nevertheless, one can try these kinds of analysis on unbounded Poincar\'e domains, motivated by a number of intriguing publications, such as \cite{BCLLLR,BL,CLR1,Gu+Lu+Wang_2019,My+Hang+Kinra+Mohan+Nguyen_Arxiv,Wang+Kinra+Mohan_2023}. By a Poincar\'e  domain, we mean a domain in which the Poincar\'e inequality \eqref{poin} is satisfied.  A typical example of a unbounded Poincar\'e domain in $\mathbb{R}^3$ is $\mathcal{O}=\R^2\times(-L,L)$ with $L>0$, see Temam \cite[p.306]{R.Temam} and Robinson \cite[p.117]{Robinson2}. More precisely, we consider the following assumption on the domain $\mathcal{O}$:
	 \begin{assumption}\label{assumpO}
	 	Let $\mathcal{O}$ be an open, connected and unbounded subset of $\R^3$, the boundary of which is uniformly of class $\mathrm{C}^3$ (see \cite{Heywood}). We assume that, there exists a positive constant $\lambda $ such that the following Poincar\'e inequality  is satisfied:
	 	\begin{align}\label{poin}
	 		\lambda\int_{\mathcal{O}} |\psi(x)|^2 \d x \leq \int_{\mathcal{O}} |\nabla \psi(x)|^2 \d x,  \ \text{ for all } \  \psi \in \H^{1}_0 (\mathcal{O}).
	 	\end{align}
	 \end{assumption}
	 \begin{remark}
	 	When $\mathcal{O}$ is a bounded domain, Poincar\'e inequality is satisfied automatically with $\lambda=\lambda_1$, where $\lambda_1$ is the first eigenvalue of the Stokes operator defined on bounded domains. Hence the results of this work hold true on bounded domains also.
	 \end{remark}
 
 Our final aim is to demonstrate the existence of an invariant measure for the SGMNS equations \eqref{SGMNSE}. Crauel et. al. in \cite{CF} demonstrated that a sufficient condition for the existence of invariant measures is the existence of a compact invariant random set. They implemented this idea to demonstrate the existence of invariant measures in bounded domains for 2D stochastic NSE and reaction-diffusion equations. Moreover, in unbounded domains, this idea was also used to establish the existence of invariant measures for 2D stochastic NSE by Brz\'ezniak et. al. (\cite{BL}) and Kinra et. al. (\cite{KM8}), for 2D and 3D stochastic CBFE by Kinra et. al. (\cite{KM2,Kinra+Mohan_2024_DCDS}). Since we show the existence of a random attractor,  the existence of an invariant measures is guaranteed as the random attractor itself is a compact invariant set. Furthermore, for sufficiently large $\nu>0$ (Theorem \ref{UIM1}),  we utilize the exponential stability of the solution to demonstrate the uniqueness of invariant measures for the system \eqref{SGMNSE}.

	 Let us now state our main results  whose proofs follow from Theorems \ref{SGMNSE-Sol}, \ref{Main_theorem_1}, \ref{thm6.3} and \ref{UIM2}. 
	 \begin{theorem}
	 Under the Assumptions \ref{assump1} and \ref{assumpO}, 
	 \begin{enumerate}
	 	\item [(i)] there exists a unique weak solution (in the analytic sense) to the SGMNS equations \eqref{SGMNSE},
	 	\item [(ii)] there exists a unique random attractor for the SGMNS equations \eqref{SGMNSE},
	 	\item [(iii)] there exists an invariant measure for the SGMNS equations \eqref{SGMNSE},
	 	\item [(iv)] for $\nu > \frac{7 N}{2} \left(\frac{1}{112 \lambda}\right)^{\frac18}$, the invariant measure is unique. 
	 \end{enumerate}
\end{theorem} 
	 
	 Apart from these results, a validation of the considered version of 3D GMNS equations has also been discussed in the appendix by establishing that the sequence of weak solutions of 3D GMNS equations converges to a weak solution of 3D NSE as $N\to\infty.$
	 
	 	\begin{remark}
	 	Note that the system \eqref{2} has also been considered in \cite{BCMV_Arxiv} to study the weak global  attractor for 3D Navier-Stokes equations via the globally modified Navier-Stokes equations \eqref{2}. We will study the weak random attractor for 3D stochastic Navier stokes equations perturbed by infinite-dimensional rough noise via the stochastic globally modified Navier-Stokes equations \eqref{SGMNSE} in  a future project.
	 \end{remark}

		\subsection{Organization of the article} 
		The rest of the paper is organized as follows: In the next section, we provide the necessary function spaces needed to obtain the existence and uniqueness of random attractors for the system \eqref{SGMNSE}. We also define the linear and nonlinear operators, and explain  their properties. Moreover, we provide an abstract formulation of the system \eqref{SGMNSE} in the same section. The metric dynamical system (MDS) and RDS corresponding to SGMNSE equations is constructed in Section \ref{sec5} using a random transformation (known as Doss-Sussman transformation, \cite{Doss_1977,Sussmann_1978}). The existence and uniqueness of a weak solution satisfying the energy equality to the transformed SGMNS equations (see the system \eqref{csgmnse} below) by using monotonicity arguments and a Minty-Browder technique is also established (Theorem \ref{solution}) in the same section. Section \ref{sec6} provides the main result of this article, that is, the existence of a random attractor for 3D SGMNS equations on  Poincar\'e domains. In order to do this, we first present Lemma \ref{RA1}, which provides  us the energy estimates for the solution of the system \eqref{csgmnse}. Based on Lemma \ref{RA1}, we introduce a new class of functions $\mathfrak{K}$, which are defined in Definition \ref{RA2}. Then, we provide the class $\mathfrak{DK}$ of closed and bounded random sets using functions in the class $\mathfrak{K}$. Next, we prove the uniform-tail estimates for the solution to  the system \eqref{csgmnse} in Lemma \ref{LR}.  We achieve our goal  by proving  Theorem \ref{Main_theorem_1}, which affirms that the RDS $\Phi$ generated by SGMNS equations on Poincar\'e domains is $\mathfrak{DK}$-asymptotically compact. Hence, in view of \cite[Theorem 2.8]{BCLLLR}, the existence of a random attractor of $\Phi$ is deduced. In the final section, we show the existence of an invariant measure for the system \eqref{SGMNSE} in Poincar\'e domains (Theorem \ref{thm6.3})  and for sufficiently large $\nu$, we prove that the invariant measure is unique (Theorem \ref{UIM2}). In Appendix \ref{ApA}, we establish that there existence a sequence of weak solutions of the 3D GMNS equations \eqref{2} which converges to a weak solution to the 3D NSE \eqref{3D-NSE} as $N\to\infty$.

	\section{Mathematical Formulation}\label{sec2}\setcounter{equation}{0}
	We start this section with the necessary function spaces needed to obtain the existence of random attractors for SGMNS equations \eqref{SGMNSE}. In order to obtain an abstract formulation for the system \eqref{SGMNSE}, we also define linear and nonlinear operators along with their properties in this section.
	\subsection{Function spaces}\label{FnO}
	 Let $\C_0^{\infty}(\mathcal{O};\R^3)$ denote the space of all infinite times differentiable functions  ($\R^3$-valued) with compact support in $\mathcal{O}\subset\R^3$. We define 
	\begin{align*} 
		\mathcal{V}&:=\{\u\in\C_0^{\infty}(\mathcal{O};\R^3):\nabla\cdot\u=0\}.
		%\\
%		\mathbb{H}&:=\text{the closure of }\ \mathcal{V} \ \text{ in the Lebesgue space } \L^2(\mathcal{O})=\mathrm{L}^2(\mathcal{O};\R^3),\\
%		\mathbb{V}&:=\text{the closure of }\ \mathcal{V} \ \text{ in the Sobolev space } \H^1(\mathcal{O})=\mathrm{H}^1(\mathcal{O};\R^3).
	\end{align*}
	 The spaces $\H$ and $\V$ are defined  as the closure of  $\mathcal{V}$  in the Lebesgue space  $\L^2(\mathcal{O})=\mathrm{L}^2(\mathcal{O};\R^3)$ and  Sobolev space $\H^1_0(\mathcal{O})=\mathrm{H}^1_0(\mathcal{O};\R^3)$, respectively. We characterize the space $\H$ with norm $\|\u\|_{\H}^2:=\int_{\mathcal{O}}|\u(x)|^2\d x$. As we are considering $\mathcal{O}$ as a  Poincar\'e domain (Assumption \ref{assumpO}), we characterize the space  $\V$ with the norm  $\|\u\|_{\V}^2:=\int_{\mathcal{O}}|\nabla\u(x)|^2\d x.$ Let $(\cdot,\cdot)$ and $(\!(\cdot,\cdot)\!)$ denote the inner product in the Hilbert spaces $\H$ and $\V$, respectively, and $\langle \cdot,\cdot\rangle $ denote the induced duality between the spaces $\V$  and its dual $\V'$.  Moreover, we have the continuous embedding $\V\hookrightarrow\H\cong\H^{'}\hookrightarrow\V'$.
	\subsection{Linear operator}\label{LO}
	Let $\mathcal{P}: \L^2(\mathcal{O}) \to\H$ denote the Helmholtz-Hodge orthogonal projection (cf.  \cite{OAL}). Let us define the Stokes operator 
	\begin{equation*}
		\A\u:=-\mathcal{P}\Delta\u,\;\u\in\D(\A).
	\end{equation*}
	The operator $\A:\V\to\V^{\prime}$ is a linear continuous operator which satisfies the following relation:
	\begin{equation*}
		\langle\A\u,\v\rangle=(\!(\u,\v)\!), \ \ \text{ for }\  \u,\v\in\V.
	\end{equation*}
In addition, the operator $\A$ is a non-negative self-adjoint operator in $\H$ and 
\begin{align}\label{2.7a}
	\langle\A\u,\u\rangle =\|\u\|_{\V}^2,\ \textrm{ for all }\ \u\in\V, \ \text{ so that }\ \|\A\u\|_{\V'}\leq \|\u\|_{\V}.
\end{align}

	\begin{remark}
		1.  Since the boundary of $\mathcal{O}$ is uniformly of class $\mathrm{C}^3$, we infer that $\D(\A)=\V\cap\H^2(\mathcal{O})$ and $\|\A\u\|_{\H}$ defines a norm in $\D(\A),$ which is equivalent to the one in $\H^2(\mathcal{O})$ (cf. \cite[Lemma 1]{Heywood}).  
		
		2. Since $\mathcal{O}$ is a Poincar\'e domain, then $\A$ is invertible and its inverse $\A^{-1}$ is bounded. %Moreover, for $\u\in\D(\A)$, we have 
%		\begin{align*}
%			\|\u\|_{\V}^2=(\A\u,\u)\leq\|\A\u\|_{\H}\|\u\|_{\H}\leq\frac{1}{\lambda^{1/2}}\|\A\u\|_{\H}\|\u\|_{\V},
%		\end{align*}
%		so that we get $\lambda^{-1/2}	\|\u\|_{\V} \leq \|\A\u\|_{\H}$ for all $\u\in\D(\A).$
	\end{remark}
	\subsection{Nonlinear operator}
	Next, we define the \emph{trilinear form} $b(\cdot,\cdot,\cdot):\V\times\V\times\V\to\R$ by $$b(\u,\v,\w)=\int_{\mathcal{O}}(\u(x)\cdot\nabla)\v(x)\cdot\w(x)\d x=\sum_{i,j=1}^3\int_{\mathcal{O}}\u_i(x)\frac{\partial \v_j(x)}{\partial x_i}\w_j(x)\d x.$$ If $\u, \v$ are such that the linear map $b(\u, \v, \cdot) $ is continuous on $\V$, the corresponding element is denoted by $\B(\u, \v)\in \V'$. We also represent $\B(\u) = \B(\u, \u)=\mathcal{P}[(\u\cdot\nabla)\u]$. Using an integration by parts, we obtain 
	\begin{equation}\label{b0}
		\left\{
		\begin{aligned}
			b(\u,\v,\v) &= 0,\ \text{ for all }\ \u,\v \in\V,\\
			b(\u,\v,\w) &=  -b(\u,\w,\v),\ \text{ for all }\ \u,\v,\w\in \V.
		\end{aligned}
		\right.\end{equation}
	\begin{remark}
	 We see that the operator $\B:\V\to\V^{\prime}$ is locally Lipschitz continuous. Consider
		\begin{align*}
			|\left<\B(\u)-\B(\v), \w\right>|&\leq |\left<\B(\u-\v,\u), \w\right>|+| \left<\B(\v, \u-\v), \w\right>|
			\nonumber\\ & \leq  |b(\u-\v, \w, \u)|+| b(\v, \w, \u-\v)|
			\nonumber\\ & \leq \|\u-\v\|_{\L^4(\mathcal{O})}\|\u\|_{\L^4(\mathcal{O})}\|\w\|_{\V} + \|\u-\v\|_{\L^4(\mathcal{O})}\|\v\|_{\L^4(\mathcal{O})}\|\w\|_{\V},
		\end{align*}
	which implies 
	\begin{align}\label{B-L}
		\|\B(\u)-\B(\boldsymbol{v})\|_{\V^{\prime}}&\leq L(r)\|\u-\boldsymbol{v}\|_{\V}, \; \text{ for all }\;\; \u,\boldsymbol{v}\in\V\; \text{ with  }\; \|\u\|_{\V}\leq r, \|\boldsymbol{v}\|_{\V}\leq r,
	\end{align}
		where $L(r)$ is a positive constant.
	\end{remark}
	
	For simplicity, we define
	\begin{align*}
		b_{N}(\u,\boldsymbol{v},\w)= F_{N}(\|\boldsymbol{v}\|_{\L^4(\mathcal{O})})\cdot b(\u,\boldsymbol{v},\w), \; \text{ for all }\ \u,\boldsymbol{v},\w\in \V.
	\end{align*}
	The form $b_{N}$ is linear in $\u$ and $\w$ but it is nonlinear in $\boldsymbol{v}$, however, we have the identity
	\begin{align*}
		b_N(\u,\boldsymbol{v},\boldsymbol{v}) = F_{N}(\|\boldsymbol{v}\|_{\L^4(\mathcal{O})})\cdot b(\u,\boldsymbol{v},\v) =0,\ \text{ for all }\ \u,\boldsymbol{v} \in\V.
	\end{align*}
%	\begin{remark}\label{Rem2.1}
%		In view of H\"older's and Ladyzhenskaya inequalities, we obtain
%		\begin{align}\label{HI+SE}
%			|b(\u,\boldsymbol{v},\w)| &\leq \|\u\|_{\L^4(\mathcal{O})} \|\nabla\boldsymbol{v}\|_{\L^2(\mathcal{O})} \|\w\|_{\L^4(\mathcal{O})}  \leq C \|\u\|_{\L^4(\mathcal{O})} \|\boldsymbol{v}\|_{\V} \|\w\|^{\frac14}_{\H} \|\w\|^{\frac34}_{\V},
%		\end{align}
%	for all $\u, \boldsymbol{v}, \w\in \V$.
%	\end{remark}
	Next, define the operator $\B_{N}:\V\times\V\to\V^{\prime}$ by 
	\begin{align*}
		\left\langle \B_{N}(\u,\boldsymbol{v}), \w \right\rangle = b_{N}(\u,\boldsymbol{v},\w), \;\text{ for all } \u, \boldsymbol{v}, \w\in \V.
	\end{align*}
For simplicity of notation, we write $\B_{N}(\u) = \B_{N}(\u,\u)$.
	%Then, using Remark \ref{Rem2.1}, we get
	%\begin{align}
	%	\|\B_N(\u,\boldsymbol{v})\|_{\V^{\prime}} &\leq C \|\u\|^{1/4}_{\H}\|\u\|^{3/4}_{\V}\|\boldsymbol{v}\|^{1/4}_{\H}\|\boldsymbol{v}\|^{3/4}_{\V},\ & \text{ for all } \u, \boldsymbol{v} \in \V,  \\
	%	\|\B_N(\u,\boldsymbol{v})\|_{\V^{\prime}} &\leq C N \|\u\|_{\V},\ & \text{ for all } \u, \boldsymbol{v} \in \V. 
	%\end{align}
%	 Moreover, we also have the following equality 
%	\begin{align}\label{BN-diff}
%		\B_{N}(\u,\u)-\B_{N}(\boldsymbol{v},\boldsymbol{v})
%		& = F_{N}(\|\u\|_{\L^4(\mathcal{O})}) \cdot  \B (\u-\boldsymbol{v},\u )
%		+ \big[ F_{N}(\|\u\|_{\L^4(\mathcal{O})}) - F_{N}(\|\boldsymbol{v}\|_{\L^4(\mathcal{O})})\big]\cdot \B ( \boldsymbol{v}, \u )
%		\nonumber\\ & \quad
%		+ F_{N}(\|\boldsymbol{v}\|_{\L^4(\mathcal{O})}) \cdot \B(\boldsymbol{v},\u-\boldsymbol{v}).
%	\end{align}
	\subsection{Properties of the function $F_N(\cdot)$}
	Let us now discuss some properties of the function $F_{N}$ which will be useful in the sequel. 	
	\begin{lemma}\label{FN-Prop}%[{\cite[Lemma 2.1]{Romito_2009}}]
		For any $\u,\boldsymbol{v}\in\V$ and each $N>0$, 
		\begin{align}
			0 \leq \|\u\|_{\L^4(\mathcal{O})} F_{N}(\|\u\|_{\L^4(\mathcal{O})}) & \leq N, \label{FN1}\\
			|F_{N}(\|\u\|_{\L^4(\mathcal{O})})-F_{N}(\|\boldsymbol{v}\|_{\L^4(\mathcal{O})})|&\leq \frac{1}{N} F_{N}(\|\u\|_{\L^4(\mathcal{O})})F_{N}(\|\boldsymbol{v}\|_{\L^4(\mathcal{O})})\cdot \|\u-\boldsymbol{v}\|_{\L^4(\mathcal{O})}.\label{FN2} 
%			\\
%			|F_{M}(\|\u\|_{\L^4(\mathcal{O})})-F_{N}(\|\u\|_{\L^4(\mathcal{O})})|&\leq \frac{1}{\|\u\|_{\L^4(\mathcal{O})}}|M-N|.\label{FN3}
		\end{align}
	\end{lemma}	
\begin{proof}
The inequality \eqref{FN1} is immediate from the definition of the function $F_N(\cdot)$. In order to prove the inequality \eqref{FN2}, we consider four different cases.
\vskip 2mm
\noindent
\textbf{Case I.} \textit{When $\|\u\|_{\L^4(\mathcal{O})}\leq N$ and $\|\v\|_{\L^4(\mathcal{O})}\leq N$.} It implies $F_N(\|\u\|_{\L^4(\mathcal{O})})=F_N(\|\v\|_{\L^4(\mathcal{O})})=1.$ Hence \eqref{FN2} follows. 
\vskip 2mm
\noindent
\textbf{Case II.} \textit{When $\|\u\|_{\L^4(\mathcal{O})}> N$ and $\|\v\|_{\L^4(\mathcal{O})}> N$.} In this case, $F_N(\|\u\|_{\L^4(\mathcal{O})})=\frac{N}{\|\u\|_{\L^4(\mathcal{O})}}$ and $F_N(\|\v\|_{\L^4(\mathcal{O})})=\frac{N}{\|\v\|_{\L^4(\mathcal{O})}}.$ Let us consider
\begin{align*}
	|F_{N}(\|\u\|_{\L^4(\mathcal{O})})-F_{N}(\|\boldsymbol{v}\|_{\L^4(\mathcal{O})})| & =\left|\frac{N}{\|\u\|_{\L^4(\mathcal{O})}}- \frac{N}{\|\v\|_{\L^4(\mathcal{O})}}\right| = \frac{N\left|\|\u\|_{\L^4(\mathcal{O})}-\|\v\|_{\L^4(\mathcal{O})}\right|}{\|\u\|_{\L^4(\mathcal{O})}\|\v\|_{\L^4(\mathcal{O})}}
	\nonumber\\
	& \leq \frac{1}{N}\frac{N}{\|\u\|_{\L^4(\mathcal{O})}}\frac{N}{\|\v\|_{\L^4(\mathcal{O})}} \|\u-\v\|_{\L^4(\mathcal{O})} \nonumber\\
	& = \frac{1}{N} F_{N}(\|\u\|_{\L^4(\mathcal{O})})F_{N}(\|\boldsymbol{v}\|_{\L^4(\mathcal{O})})\cdot \|\u-\boldsymbol{v}\|_{\L^4(\mathcal{O})},
\end{align*} 
 so that  \eqref{FN2} follows. 
 \vskip 2mm
 \noindent
 \textbf{Case III.} \textit{When $\|\u\|_{\L^4(\mathcal{O})}\leq N$ and $\|\v\|_{\L^4(\mathcal{O})}> N$.} Here, we have $F_N(\|\u\|_{\L^4(\mathcal{O})})=1$ and $F_N(\|\v\|_{\L^4(\mathcal{O})})=\frac{N}{\|\v\|_{\L^4(\mathcal{O})}}.$ Let us consider
 \begin{align*}
 	|F_{N}(\|\u\|_{\L^4(\mathcal{O})})-F_{N}(\|\boldsymbol{v}\|_{\L^4(\mathcal{O})})| & =\left|1- \frac{N}{\|\v\|_{\L^4(\mathcal{O})}}\right|  = \frac{\left|\|\v\|_{\L^4(\mathcal{O})}-N\right|}{\|\v\|_{\L^4(\mathcal{O})}}\nonumber\\
 	& \leq \frac{\left|\|\v\|_{\L^4(\mathcal{O})}-\|\u\|_{\L^4(\mathcal{O})}\right|}{\|\v\|_{\L^4(\mathcal{O})}}
 	\nonumber\\
 	& \leq \frac{1}{N}\cdot1\cdot\frac{N}{\|\v\|_{\L^4(\mathcal{O})}} \cdot \|\u-\v\|_{\L^4(\mathcal{O})} \nonumber\\
 	& = \frac{1}{N} F_{N}(\|\u\|_{\L^4(\mathcal{O})})F_{N}(\|\boldsymbol{v}\|_{\L^4(\mathcal{O})})\cdot \|\u-\boldsymbol{v}\|_{\L^4(\mathcal{O})},
 \end{align*} 
 so that  \eqref{FN2} follows. 
  \vskip 2mm
 \noindent
 \textbf{Case IV.} \textit{When $\|\u\|_{\L^4(\mathcal{O})}> N$ and $\|\v\|_{\L^4(\mathcal{O})}\leq N$.} By interchanging the role of $\u$ and $\v$ in \textbf{Case III}, we obtain \eqref{FN2} for \textbf{Case IV}, which  completes the proof.
  %For the proof of \eqref{FN3}, we refer readers to the proof of \cite[Lemma 6]{Caraballo+Real+Kloeden_2006}. This completes the proof.
\end{proof}
	 \begin{remark}
	 	In view of Lemma \ref{FN-Prop}, we obtain the locally Lipschitz continuity of operator $B_N$, that is, 
	 	\begin{align}\label{BN-L}
	 		\|\B_N(\u)-\B_{N}(\boldsymbol{v})\|_{\V^{\prime}}\leq \wi L(r)\|\u-\boldsymbol{v}\|_{\V}, 
	 	\end{align}
for all $\u,\boldsymbol{v}\in\V$  with  $\|\u\|_{\V}\leq r$ and $\|\boldsymbol{v}\|_{\V}\leq r$,	 where $\wi L(r)$ is a positive constant.
	 \end{remark}
	
%	The following interpolation inequality is used frequently in the upcoming sections. 
%	\begin{lemma}[Interpolation inequality]\label{Interpolation}
%		Assume $1\leq s_1\leq s\leq s_2\leq \infty$, $a\in(0,1)$ such that $\frac{1}{s}=\frac{a}{s_1}+\frac{1-a}{s_2}$ and $\u\in\L^{s_1}(\mathcal{O})\cap\L^{s_2}(\mathcal{O})$, then we have 
%		\begin{align*}
%			\|\u\|_{\L^s(\mathcal{O})}\leq\|\u\|_{\L^{s_1}(\mathcal{O})}^{a}\|\u\|_{\L^{s_2}(\mathcal{O})}^{1-a}. 
%		\end{align*}
%	\end{lemma}
	\begin{remark}\label{rem2.5}
		We have the following well-known inequality  due to Ladyzhenskaya (\cite[Chapter I]{OAL}) which will be used frequently in the  sequel:
		\begin{align}\label{lady}
			\|\v\|_{\L^{4}(\mathcal{O}) } \leq 
				2^{1/2} \|\v\|^{1/4}_{\L^{2}(\mathcal{O}) } \|\nabla \v\|^{3/4}_{\L^{2}(\mathcal{O}) }, \ \ \ \v\in \H^{1}_0 (\mathcal{O}).
		\end{align}
	\end{remark}

	\subsection{Stochastic globally modified Navier-Stokes equations}
	In this subsection, we provide an abstract formulation of the system \eqref{SGMNSE}.
	On taking projection $\mathcal{P}$ onto the first equation in \eqref{SGMNSE}, we obtain 
	\begin{equation}\label{S-GMNSE}
		\left\{
		\begin{aligned}
			\d\u(t)+\{\nu \A\u(t)+F_{N}(\|\u\|_{\L^4(\mathcal{O})}) \cdot \B(\u(t))\}\d t&=\mathcal{P}\f \d t + \d\mathrm{W}(t), \ \ \ t> 0, \\ 
			\u(0)&=\boldsymbol{x},
		\end{aligned}
		\right.
	\end{equation}
 where $\boldsymbol{x}\in \H,\ \f\in \H^{-1}(\mathcal{O})$ and $\{\mathrm{W}(t)\}_{t\in \R}$ is a two-sided cylindrical Wiener process in $\H$ with its RKHS $\mathrm{K}$ which satisfies Assumption \ref{assump1}.
 Now we present the definition of a solution to system \eqref{S-GMNSE} with the initial data $\x\in\H$ at the initial time $s\in \R.$
 \begin{definition}\label{Def_u}
 	Suppose that the Assumptions \ref{assump1} and \ref{assumpO} are satisfied. If $\x\in \H, s\in \R, \f\in \H^{-1}(\mathcal{O})$ and $\{\W(t)\}_{t\in \R}$ is a two-sided Wiener process with its RKHS $\mathrm{K}$. A process $\{\u(t), \ t\geq 0\},$ with trajectories in $\mathrm{C}([s, \infty); \H) \cap \mathrm{L}^{\frac{8}{3}}_{\mathrm{loc}}([s, \infty); \L^4(\mathcal{O}))$ is a solution to system \eqref{S-GMNSE} if and only if  $\u(s) = \x$ and for any $\psi\in \V, t>s,$ $\mathbb{P}$-a.s.,
 	\begin{align*}
 		(\u(t), \psi) = (\u(s), \psi) - \int_{s}^{t} \langle \nu \A\u(\tau)+\B_N(\u(\tau))  - \f , \psi \rangle \d \tau  +  \int_{s}^{t} ( \psi, \d\W(\tau)).
 	\end{align*}
 \end{definition}
 
  The well-posedness of the system \eqref{S-GMNSE} is discussed in Theorem \ref{SGMNSE-Sol} below. 

	\section{RDS generated by SGMNS equations}\label{sec5}\setcounter{equation}{0}
	 In this section, we prove the existence of RDS generated by 3D SGMNS equations using a random transformation (known as Doss-Sussman transformation, \cite{Doss_1977,Sussmann_1978}). The process of generation of RDS has been adopted form the works \cite{BCLLLR,BL}. 
	 \subsection{Metric dynamical system}Let us denote $\mathrm{X} = \H \cap  {\mathbb{L}}^{4}(\mathcal{O}) $. Let $\mathrm{E}$ denote the completion of $\A^{-\delta}(\mathrm{X})$ with respect to the graph norm $\|x\|_{\mathrm{E}}=\|\A^{-\delta} x \|_{\mathrm{X}}, \text{ for } x\in \mathrm{X}$,  where $ \|\cdot\|_{\mathrm{X}} = \|\cdot\|_{\H} +  \|\cdot\|_{{\L}^4(\mathcal{O})}$. Note that $\mathrm{E}$ is a separable Banach spaces (cf. \cite{Brze2}).
	
	For $\xi \in(0, 1/2)$, let us set 
	$$ \|\omega\|_{C^{\xi}_{1/2} (\mathbb{R};\mathrm{E})} = \sup_{t\neq s \in \mathbb{R}} \frac{\|\omega(t) - \omega(s)\|_{\mathrm{E}}}{|t-s|^{\xi}(1+|t|+|s|)^{1/2}}.$$
	Furthermore, we define
	\begin{align*}
		C^{\xi}_{1/2} (\mathbb{R}; \mathrm{E}) &= \left\{ \omega \in C(\mathbb{R}; \mathrm{E}) : \omega(0)=\boldsymbol{0},\ \|\omega\|_{C^{\xi}_{1/2} (\mathbb{R}; \mathrm{E})} < \infty \right\},\\ \Omega(\xi, \mathrm{E})&= \overline{\{ \omega \in C^\infty_0 (\mathbb{R}; \mathrm{E}) : \omega(0) = 0 \}}^{C^{\xi}_{1/2} (\mathbb{R}; \mathrm{E})}.
	\end{align*}
	The space $\Omega(\xi, \mathrm{E})$ is a separable Banach space. We also define
	$$C_{1/2} (\mathbb{R}; \mathrm{E}) = \left\{ \omega \in C(\mathbb{R}; \mathrm{E}) : \omega(0)=0, \|\omega\|_{C_{1/2} (\mathbb{R}; \mathrm{E})} = \sup_{t \in \mathbb{R}} \frac{\|\omega(t) \|_{\mathrm{E}}}{1+|t|^{\frac{1}{2}}} < \infty \right\}.$$

	Let us denote by $\mathcal{F}$, the Borel $\sigma$-algebra on $\Omega(\xi, \mathrm{E}).$ For $\xi\in (0, 1/2)$, there exists a Borel probability measure $\mathbb{P}$ on $\Omega(\xi, \mathrm{E})$ (cf. \cite{Brze}) such that the canonical process $\{w_t, \ t\in \mathbb{R}\}$ defined by 
	\begin{align}\label{Wp}
		w_t(\omega) := \omega(t), \ \ \ \omega \in \Omega(\xi, \mathrm{E}),
	\end{align}
	is an $\mathrm{E}$-valued two-sided Wiener process such that the RKHS of the Gaussian measure $\mathscr{L}(w_1)$ on $\mathrm{E}$ is $\mathrm{K}$. For $t\in \mathbb{R},$ let $\mathcal{F}_t := \sigma \{ w_s : s \leq t \}.$ Then  there exists a unique bounded linear map $\mathrm{W}(t): \mathrm{K} \to \mathrm{L}^2(\Omega(\xi, \mathrm{E}), \mathcal{F}_t  ,  \mathbb{P}).$ Moreover, the family $\{\mathrm{W}(t)\}_{t\in \mathbb{R}}$ is a $\mathrm{K}$-cylindrical Wiener process on a filtered probability space $(\Omega(\xi, \mathrm{E}), \mathcal{F}, \{\mathcal{F}_t\}_{t \in \mathbb{R}} , \mathbb{P})$ (cf. \cite{BP} for more details).
	
	We consider a flow $\vartheta = (\vartheta_t)_{t\in \mathbb{R}}$ on the space $C_{1/2} (\mathbb{R}; \mathrm{E}),$  defined by
	$$ \vartheta_t \omega(\cdot) = \omega(\cdot + t) - \omega(t), \ \ \ \omega\in C_{1/2} (\mathbb{R};\mathrm{E}), \ \ t\in \mathbb{R}.$$ 
	This flow keeps the spaces $C^{\xi}_{1/2} (\mathbb{R};\mathrm{E})$ and $\Omega(\xi, \mathrm{E})$ invariant and preserves $\mathbb{P}.$ 
	
	Summing up, we have the following result:
	\begin{proposition}[{\cite[Proposition 6.13]{BL}}]\label{m-DS1}
		The quadruple $(\Omega(\xi, \mathrm{E}), \mathcal{F}, \mathbb{P}, {\vartheta})$ is an MDS. 
	\end{proposition}

\subsection{Ornstein-Uhlenbeck process}\label{O-Up}
Let us first recall some analytic preliminaries from \cite{BL} which will help us to define an Ornstein-Uhlenbeck process. All the results of this subsection are valid for the space $\mathrm{C}^{\xi}_{1/2} (\mathbb{R}; \mathbb{Y})$ replaced by $\Omega(\xi, \mathbb{Y}).$ 
\begin{proposition}[{\cite[Proposition 2.11]{BL}}]\label{Ap}
	Let  $-\mathbb{A}$ be the generator of an analytic semigroup $\{e^{t\mathbb{A}}\}_{t\geq 0}$ on a separable Banach space $\mathbb{Y}$ such that for some $C>0\ \text{and}\ \gamma>0$
	\begin{align}\label{ASG}
		\| \mathbb{A}^{1+\delta}e^{-t\mathbb{A}}\|_{\mathfrak{L}(\mathbb{Y})} \leq C_{\delta} t^{-1-\delta} e^{-\gamma t}, \ \ t> 0,
	\end{align}
	where $\mathfrak{L}(\mathbb{Y})$ denotes the space of all bounded linear operators from $\mathbb{Y}$ to $\mathbb{Y}$.	For $\xi \in (\delta, 1/2)$ and $\widetilde{\omega} \in  \mathrm{C}^{\xi}_{1/2} (\mathbb{R};\mathbb{Y}),$  define 
	\begin{align}
		\hat{\z}(t) = \hat{\z} (\mathbb{A}; \widetilde{\omega})(t) := \int_{-\infty}^{t} \mathbb{A}^{1+\delta} e^{-(t-r)\mathbb{A}} (\widetilde{\omega}(t) - \widetilde{\omega}(r))\d r, \ \ t\in \mathbb{R}.
	\end{align}
	If $t\in \mathbb{R},$ then $\hat{\z}(t)$ is a well-defined element of $\mathbb{Y}$ and the mapping 
	$$\mathrm{C}^{\xi}_{1/2} (\mathbb{R};\mathbb{Y}) \ni \widetilde{\omega}  \mapsto \hat{\z}(t) \in \mathbb{Y} $$
	is continuous. Moreover, the map $\hat{\z} :  \mathrm{C}^{\xi}_{1/2} (\mathbb{R}; \mathbb{Y}) \to  \mathrm{C}_{1/2} (\mathbb{R}; \mathbb{Y})$  is well defined, linear and bounded. In particular, there exists a constant $C >0$ such that for any $\widetilde{\omega} \in \mathrm{C}^{\xi}_{1/2} (\mathbb{R};\mathbb{Y})$ 
	\begin{align}\label{X_bound_of_z}
		\|\hat{\z}(\widetilde{\omega})(t)\|_{\mathbb{Y}} \leq C(1 + |t|^{1/2})\|\widetilde{\omega}\|_{\mathrm{C}^{\xi}_{1/2} (\mathbb{R}; \mathbb{Y})}, \ \ \ t \in \R.
	\end{align}
	Furthermore, under the same assumption, the following results hold (Corollaries 6.4, 6.6 and 6.8 in \cite{BL}):
	\begin{itemize}
		\item [1.]For all $-\infty<a<b<\infty$ and $t\in \R$, the map 
		\begin{align}\label{O-U_conti}
			\mathrm{C}^{\xi}_{1/2} (\mathbb{R};\mathbb{Y}) \ni \widetilde{\omega} \mapsto (\hat{\z}(\widetilde{\omega})(t), \hat{\z}(\widetilde{\omega})) \in \mathbb{Y} \times \mathrm{L}^{q} (a, b; \mathbb{Y}),
		\end{align}
		where $q\in [1, \infty]$, is continuous.
		\item [2.] For any $\omega \in \mathrm{C}^{\xi}_{1/2} (\mathbb{R};\mathbb{Y}),$
		\begin{align}\label{stationary}
			\hat{\z}(\vartheta_s \omega)(t) = \hat{\z}(\omega)(t+s), \ \ t, s \in \mathbb{R}.
		\end{align}
		\item [3.] For $\zeta \in \mathrm{C}_{1/2}(\mathbb{R};\mathbb{Y}),$ if we put $\uptau_s(\zeta(t))=\zeta(t+s), \ t,s \in \R,$ then, for $t \in \R ,\  \uptau_s \circ \hat{\z} = \hat{\z}\circ\vartheta_s$, that is, 
		\begin{align}\label{IS}
			\uptau_s\big(\hat{\z}(\omega)\big)= \hat{\z}\big(\vartheta_s(\omega)\big), \ \ \ \omega\in \mathrm{C}^{\xi}_{1/2} (\mathbb{R};\mathbb{Y}).
		\end{align}
	\end{itemize} 
\end{proposition}

	Next, we define the Ornstein-Uhlenbeck process under Assumption \ref{assump1}. For $\delta$ as in Assumption \ref{assump1}, $\nu> 0, \ \chi \geq 0, \ \xi \in (\delta, 1/2)$ and $ \omega \in C^{\xi}_{1/2} (\mathbb{R};\mathrm{E})$ (so that $(\nu \A + \chi\I)^{-\delta}\omega \in C^{\xi}_{1/2} (\mathbb{R};\mathrm{X})$), we define $$ \mathfrak{Z}_{\chi}(\omega) := \hat{\z}((\nu \A + \chi\I); (\nu \A + \chi\I)^{-\delta}\omega) \ \in C_{1/2}(\mathbb{R};\mathrm{X}),$$  that is, for any $t\geq 0,$ 
	\begin{align}\label{DOu1}
		\mathfrak{Z}_{\chi}(\omega)(t)&=\int_{-\infty}^{t} (\nu \A + \chi\I)^{1+\delta} e^{-(t-\tau)(\nu \A + \chi\I)} ((\nu \A + \chi\I)^{-\delta}\vartheta_{\tau} \omega)(t-\tau)\d \tau.
	\end{align}
	For $\omega \in C^{\infty}_0 (\mathbb{R};\mathrm{E})$ with $\omega(0)= \boldsymbol{0},$ using  integration by parts, we obtain 
	\begin{align*}
		\frac{\d\mathfrak{Z}_\chi(t)}{\d t} &= -(\nu \A + \chi\I )\int_{-\infty}^{t} (\nu \A + \chi\I)^{1+\delta} e^{-(t-r)(\nu \A + \chi\I)} [(\nu \A + \chi\I)^{-\delta}\omega(t) \\&\qquad\qquad - (\nu \A + \chi\I)^{-\delta}\omega(r)]\d r +  \frac{\d\omega(t)}{\d t}.
	\end{align*}
	Thus $\mathfrak{Z}_{\chi}(\cdot)$ is the solution of the following equation:
	\begin{align}\label{OuE1}
		\frac{\d\mathfrak{Z}_{\chi} (t)}{\d t} + (\nu \A + \chi\I)\mathfrak{Z}_{\chi}(t) = \frac{\d\omega(t)}{\d t}, \ \ t\in \mathbb{R}.
	\end{align}
	Therefore, from the definition of the space $\Omega(\xi, \mathrm{E}),$ we have 
	\begin{corollary}\label{Diff_z1}
		If $\chi_1, \chi_2 \geq 0,$ then the difference $\mathfrak{Z}_{\chi_1} - \mathfrak{Z}_{\chi_2}$ is a solution to 
		\begin{align}\label{Dif_z1}
			\frac{\d(\mathfrak{Z}_{\chi_1} - \mathfrak{Z}_{\chi_2})(t)}{\d t} + \nu\A(\mathfrak{Z}_{\chi_1} - \mathfrak{Z}_{\chi_2})(t) = -(\chi_1 \mathfrak{Z}_{\chi_1} - \chi_2\mathfrak{Z}_{\chi_2})(t), \ \ \ t \in \R.
		\end{align}
	\end{corollary}
	
	According to the  definition \eqref{Wp} of Wiener process $\{w_t, \ t\in \R\},$ one can view the formula \eqref{DOu1} as a definition of a process $\{\mathfrak{Z}_{\chi}(t), \ t\in \R\}$ on the probability space $(\Omega(\xi, \mathrm{E}), \mathcal{F}, \mathbb{P})$. Equation \eqref{OuE1} clearly tells that the process $\mathfrak{Z}_{\chi}(\cdot)$ is an Ornstein-Uhlenbeck process. Furthermore, the following results hold for $\mathfrak{Z}_{\chi}(\cdot)$.
	\begin{proposition}[{\cite[Proposition 6.10]{BL}}]\label{SOUP1}
		The process $\{\mathfrak{Z}_{\chi}(t), \ t\in \mathbb{R}\},$ is a stationary Ornstein-Uhlenbeck process on $(\Omega(\xi, \mathrm{E}), \mathcal{F}, \mathbb{P})$. It is a solution of the equation 
		\begin{align}\label{OUPe1}
			\d\mathfrak{Z}_{\chi}(t) + (\nu \A + \chi \I)\mathfrak{Z}_{\chi}(t) \d t = \d\mathrm{W}(t), \ \ t\in \mathbb{R},
		\end{align}
		that is, for all $t\in \mathbb{R},$ $\mathbb{P}$-a.s.,
		\begin{align}\label{oup1}
			\mathfrak{Z}_\chi (t) = \int_{-\infty}^{t} e^{-(t-\xi)(\nu \A + \chi\I)} \d\mathrm{W}(\xi),
		\end{align}
		where the integral is an It\^o integral on the M-type 2 Banach space $\mathrm{X}$  (cf. \cite{Brze1}). 	In particular, for some $C$ depending on $\mathrm{X}$,
		\begin{align}\label{E-OUP1}
			\mathbb{E}\left[\|\mathfrak{Z}_{\chi} (t)\|^2_{\mathrm{X}} \right]&= \mathbb{E}\left[\left\|\int_{-\infty}^{t} e^{-(t-\xi)(\nu \A + \chi\I)} \d\mathrm{W}(\xi)\right\|^2_{\mathrm{X}}\right] \leq C\int_{-\infty}^{t} \|e^{-(t-\xi)(\nu \A +  \chi\I)}\|^2_{\gamma(\mathrm{K},\mathrm{X})} \d \xi \nonumber\\&=C \int_{0}^{\infty}  e^{-2\chi \xi} \|e^{-\nu \xi \A}\|^2_{\gamma(\mathrm{K},\mathrm{X})} \d \xi.
		\end{align} 
		Moreover, $\mathbb{E}\left[\|\mathfrak{Z}_{\chi} (t)\|^2_{\mathrm{X}}\right]\to 0$ as $\chi \to \infty.$
	\end{proposition}
	
		\begin{remark}
			Note that \eqref{OUPe1} is the projected form of an  equation of the following type: 
			\begin{equation}\label{eqn_z_alpha}
				\left\{
				\begin{aligned}
					\d\mathfrak{Z}_\chi (t)  + \left[(-\nu\Delta +\chi \I) \mathfrak{Z}_\chi (t) + \nabla  {p}_1 \right]\d t &=\d \mathrm{W}(t) ,  \ \ t\in \mathbb{R},\\
					\mathrm{div}\;\mathfrak{Z}_\chi &=0,
					%,\\		\mathfrak{Z}(0)&=\u_0,
				\end{aligned}
				\right.
			\end{equation}
			where $p_1$ is a scalar field associated with projected equation \eqref{OUPe1}.
	\end{remark}

	\subsection{Random dynamical system}
	Remember that Assumption \ref{assump1} is satisfied and that $\delta$ has the property stated there. Let us fix $\nu> 0$, and the parameters $\chi\geq 0$ and  $\xi \in (\delta, 1/2)$.
	
	Using a random transformation (known as Doss-Sussman transformation, \cite{Doss_1977,Sussmann_1978}), we get a random partial differential equation which equivalent to the system \eqref{S-GMNSE}. Let us define 
	\begin{align}\label{D-S_trans}
		\v^{\chi}(t):=\u(t) - \mathfrak{Z}_{\chi}(\omega)(t).
	\end{align}
	 For convenience, we write $\v^{\chi}(t)=\v(t)$ and $\mathfrak{Z}_{\chi}(\omega)(t)=\mathfrak{Z}(t)$. Then $\v(\cdot)$ satisfies the following system:
	 \begin{equation}\label{cgmnse-with-Pressure}
	 	\left\{ 
	 	\begin{aligned}
	 		\frac{\d \v}{\d t} & = -\underbrace{\nabla (p-p_1)}_{:=\nabla\widehat{p}} + \nu \Delta \v- F_{N}(\|\v+\mathfrak{Z}\|_{\L^4(\mathcal{O})}) \big((\v+\mathfrak{Z})\cdot \nabla\big)(\v+\mathfrak{Z})    
	 		  \\ & \quad  + \chi \mathfrak{Z} + \f  &&  \text{in }  \mathcal{O} \times (0,\infty),\\
	 		\text{div}\; \v&=0 \quad &&  \text{in } \mathcal{O} \times [0,\infty),\\
	 		\v &= \boldsymbol{0} &&  \text{on } \partial\mathcal{O}\times [0,\infty),\\
	 		\v(x,0)&=\x - \mathfrak{Z}_{\chi}(\omega)(0)=:\v_0  \quad &&  \text{in } \mathcal{O},
	 	\end{aligned}
	 	\right.
	 \end{equation}
	 where $p_1$ is the scalar field appearing in \eqref{eqn_z_alpha}, and the following projected system:
	\begin{equation}\label{csgmnse}
		\left\{
		\begin{aligned}
			\frac{\d\v}{\d t} &= -\nu \A\v -  \B_N(\v + \mathfrak{Z}) + \chi \mathfrak{Z} + \mathcal{P}\f, \\
			\v(0)&= \v_0 = \boldsymbol{x} - \mathfrak{Z}_{\chi}(0).
		\end{aligned}
		\right.
	\end{equation}
	Since $\mathfrak{Z}_{\chi}(\omega) \in C_{1/2} (\mathbb{R};\mathrm{X}), $ then $\mathfrak{Z}_{\chi}(\omega)(0)$ is a well defined element of $\H$. Let us now provide the definition of weak solution (in the deterministic sense, for each  fixed $\omega$) for \eqref{csgmnse}.
	\begin{definition}\label{defn5.9}
		Assume that $\v_0 \in \H$, $\f\in \H^{-1}(\mathcal{O})$ and $\mathfrak{Z}\in\mathrm{L}^2_{\mathrm{loc}}([0,\infty);\H\cap\L^4(\mathcal{O}))$. A function $\v(\cdot)$ is called a \emph{weak solution} of the system \eqref{csgmnse} on the time interval $[0, \infty)$, if $$\v\in  \mathrm{C}([0,\infty); \H) \cap \mathrm{L}^{2}_{\mathrm{loc}}(0,\infty; \V), \;\;\; \frac{\d\v}{\d t}\in\mathrm{L}^{2}_{\mathrm{loc}}(0,\infty;\V'),$$ and it satisfies 
		\begin{itemize}
			\item [(i)] for any $\psi\in \V,$ 
			\begin{align*}
				\left<\frac{\d\v(t)}{\d t}, \psi\right>&=  - \left\langle \nu \A\v(t)+\B_N(\v(t)+\mathfrak{Z}(t)) - \chi\mathfrak{Z}(t)- \f , \psi \right\rangle,
			\end{align*}
			for a.e. $t\in[0,\infty);$
			\item [(ii)] the initial data:
			$$\v(0)=\v_0 \ \text{ in }\ \H.$$
		\end{itemize}
	\end{definition}

	\begin{theorem}\label{solution}
		Let $\mathcal{O}$ satisfy  Assumption \ref{assumpO}, $\chi\geq0$, $\v_0 \in \H$, $\f\in \H^{-1}(\mathcal{O})$ and $\mathfrak{Z}\in\mathrm{L}^2_{\mathrm{loc}}([0,\infty);\H\cap\L^{4}(\mathcal{O}))$. Then there exists a unique weak solution $\v(\cdot)$ to the system \eqref{csgmnse} in the sense of Definition \ref{defn5.9}. Moreover,  the following energy equality is satisfied: 
		\begin{align}\label{eeq}
			&\|\v(t)\|_{\H}^2+2\nu\int_0^t\|\v(s)\|_{\V}^2\d s + 2\int_0^t\langle\B_N(\v(s)+\mathfrak{Z}(s)),\v(s)\rangle\d s  \nonumber\\&= \|\v_0\|_{\H}^2 +2\int_0^t\langle\f,\v(s)\rangle\d s+2\chi\int_0^t(\mathfrak{Z}(s),\v(s))\d s,
		\end{align}
		for all  $t\in[0,T]$	with $0<T<\infty$.
	\end{theorem}
%\begin{proof}[Proof 1.]
%	
%\end{proof}

	\begin{proof}%[Proof 2.]
		Let us fix $T>0$. Note that it is enough to prove the result on the  interval $[0,T]$. 
		\vskip 2mm
		\noindent
		\textbf{Step I.} \emph{Existence of weak solutions.} Let $\{\boldsymbol{e}_1, \ldots , \boldsymbol{e}_n, \ldots\}$ be a complete orthonormal
		system in $\H$ belonging to $\V$ and let $\H_{n}:=\mathrm{span}\{\boldsymbol{e}_1,\ldots,\boldsymbol{e}_n\}$. Let $\mathrm{P}_n$ denote the orthogonal projection of $\V'$ to $\H_n$, that is, $\mathrm{P}_n\x=\sum_{i=1}^n\langle \x,\boldsymbol{e}_i\rangle \boldsymbol{e}_i$. Since every element $\x\in\H$ induces a functional $\x^*\in\H$  by the formula $\langle \x^*,\wi\x\rangle =(\x,\wi\x)$, $\wi\x\in\V$, then $\mathrm{P}_n\big|_{\H}$, the orthogonal projection of $\H$ onto $\H_n$  is given by $\mathrm{P}_n\x=\sum_{i=1}^n\left(\x,\boldsymbol{e}_i\right)\boldsymbol{e}_i$. Hence in particular, $\mathrm{P}_n$ is the orthogonal projection from $\H$ onto $\H_n$.

		 Let us consider the following system of ODEs on the finite dimensional space $\H_n$:
%		\begin{equation}
%			\left\{
%			\begin{aligned}
%				\frac{\d\v^n}{\d t} &= \P_n\bigg[-\nu \A\v^n - \B_N(\v^n+\mathfrak{Z}) + \chi\mathfrak{Z} + \mathcal{P}\f\bigg], \\
%				\v^n(0)&= \P_n[\x-\mathfrak{Z}(0)]:={\v_0}_n.
%			\end{aligned}
%			\right.
%		\end{equation}
	\begin{equation}\label{finite-dimS}
	\left\{
	\begin{aligned}
		\frac{\d\v^n(t)}{\d t}&=-\nu \A_n\v^n(t)-\B_{n,N}(\v^n(t)+\mathfrak{Z}(t))+\chi\mathfrak{Z}_n(t)+\f_n,\\
		\v^n(0)&={\v_0}_n,
	\end{aligned}
	\right.
\end{equation}
		where  $\A_n\v^n=\P_n\A\v^n$, $\B_{n,N}(\v^n+\mathfrak{Z})=\mathrm{P}_n\B_N(\v^n+\mathfrak{Z})$, $\mathfrak{Z}_n=\P_n\mathfrak{Z}$ and $\f_n=\P_n[\mathcal{P}\f]$.
	Since $\B_{n,N}(\cdot)$ is  locally Lipschitz (see \eqref{BN-L}), the system \eqref{finite-dimS} has a unique local solution $\v^n\in\mathrm{C}([0,T^*];\H_n)$, for some $0<T^*<T$. The following a priori estimates show that the time $T^*$ can be extended to time $T$. Taking the inner product with $\v^n(\cdot)$ to the first equation of \eqref{finite-dimS}, we obtain for a.e. $t\in[0,T]$
		\begin{align}\label{S1}
			\frac{1}{2}\frac{\d}{\d t}\|\v^n(t)\|^2_{\H}
		  &=-\nu\|\v^n(t)\|^2_{\V} +F_{N}(\|\v^n(t)+\mathfrak{Z}(t)\|_{\L^4(\mathcal{O})})   \cdot b(\v^n(t)+\mathfrak{Z}(t),\v^n(t),\mathfrak{Z}(t))\nonumber\\&\quad+ \chi(\mathfrak{Z}(t),\v^n(t))+\langle\f,\v^n(t)\rangle.
		\end{align}
		Next, we estimate the  terms of the right hand side of \eqref{S1} as 
		\begin{align}\label{Sb}
		&	|F_{N}(\|\v^n+\mathfrak{Z}\|_{\L^4(\mathcal{O})})   \cdot b(\v^n+\mathfrak{Z}, \v^n, \mathfrak{Z})| \nonumber\\& 
		\leq  	F_{N}(\|\v^n+\mathfrak{Z}\|_{\L^4(\mathcal{O})})   \cdot 	\|\v^n+\mathfrak{Z}\|_{{\L}^{4}(\mathcal{O})} \|\v^n\|_{\V} \|\mathfrak{Z}\|_{{\L}^{4}(\mathcal{O})}, \nonumber \\ 
			& \leq  	N  \|\v^n\|_{\V} \|\mathfrak{Z}\|_{{\L}^{4}(\mathcal{O})} \leq \frac{\nu}{4} \|\v^n\|_{\V}^2 +    C N^2 \|\mathfrak{Z}\|_{{\L}^{4}(\mathcal{O})}^2 
		\end{align}
	and
			\begin{align}
			\big|\chi(\mathfrak{Z},\v^n)&+\langle\f,\v^n\rangle\big|\leq \frac{\nu}{4}\|\v^n\|^2_{\V} +C\|\f\|^2_{\H^{-1}(\mathcal{O})}+C\|\mathfrak{Z}\|^2_{\H},\label{S2}
		\end{align}
	where we have used H\"older's inequality, \eqref{FN1} and Young's inequality. Combining \eqref{S1}-\eqref{S2}, we deduce
		\begin{align}\label{S3}
			&\frac{\d}{\d t}\|\v^n(t)\|^2_{\H}+\nu\|\v^n(t)\|^2_{\V} 
			 \leq C\|\f\|^2_{\H^{-1}(\mathcal{O})} + C \|\mathfrak{Z}(t)\|^2_{\H}+ C N^2\|\mathfrak{Z}(t)\|^{2}_{\L^{4}(\mathcal{O})} ,
		\end{align}
		which gives for all $t\in[0,T]$
		\begin{align}\label{S4}
		&	\|\v^n(t)\|^2_{\H} + \nu \int_{0}^{t}\|\v^n(s)\|^2_{\V} \d s  
			\nonumber\\
&\leq 				\|\v^n(0)\|^2_{\H}   +C\int_{0}^{t}\big[ \|\f\|^2_{\H^{-1}(\mathcal{O})} + \|\mathfrak{Z}(s)\|^2_{\H}+ N^2\|\mathfrak{Z}(s)\|^{2}_{\L^{4}(\mathcal{O})}\big]\d s.
		\end{align}
		 Hence, using the fact that $\|\v^n(0)\|_{\H}\leq\|\v(0)\|_{\H}$, $\f\in\H^{-1}(\mathcal{O})$ and $\mathfrak{Z}\in \mathrm{L}^2(0,T; \H\cap {\L}^{4}(\mathcal{O}))$, we have from \eqref{S4} that
		\begin{align}\label{S5}
			\{\v^n\}_{n\in\N} \text{ is a bounded sequence in }\mathrm{L}^{\infty}(0,T;\H)\cap\mathrm{L}^{2}(0,T;\V).
		\end{align}
		For any arbitrary element $\boldsymbol{\psi}\in\mathrm{L}^2(0,T;\V)$, using H\"older's inequality and Sobolev's embedding, we have from \eqref{finite-dimS} that
		\begin{align*}
			&\left|\int_{0}^{T}\left\langle\frac{\d\v^n(t)}{\d t},\boldsymbol{\psi}(t)\right\rangle\d t\right|
			 \nonumber\\&\leq \int_{0}^{T}\bigg[\nu\left|(\nabla\v^n(t),\nabla\boldsymbol{\psi}(t))\right|+\left|F_{N}(\|\v^n(t)+\mathfrak{Z}(t)\|_{\L^4(\mathcal{O})})   \cdot b(\v^n(t)+\mathfrak{Z}(t),\boldsymbol{\psi}(t),\v^n(t)+\mathfrak{Z}(t))\right|
			\nonumber\\&\quad +\chi \left|(\mathfrak{Z}(t),\boldsymbol{\psi}(t))\right|+\left|\langle\f,\boldsymbol{\psi}(t)\rangle\right|\bigg]\d t 
			\nonumber\\&\leq C\int_{0}^{T}\bigg[\|\v^n(t)\|_{\V}\|\boldsymbol{\psi}(t)\|_{\V}
			+ F_{N}(\|\v^n(t)+\mathfrak{Z}(t)\|_{\L^4(\mathcal{O})})   \cdot\|\v^n(t)+\mathfrak{Z}(t)\|_{\L^{4}(\mathcal{O})}^{2}\|\boldsymbol{\psi}(t)\|_{\V} 
			\nonumber\\ & \qquad +\|\mathfrak{Z}(t)\|_{\H}\|\boldsymbol{\psi}(t)\|_{\H} +\|\f\|_{\H^{-1}(\mathcal{O})}\|\boldsymbol{\psi}(t)\|_{\V}\bigg]\d t
			\nonumber\\&\leq C\int_{0}^{T}\bigg[\|\v^n(t)\|_{\V}
			+ N\|\v^n(t)+\mathfrak{Z}(t)\|_{\L^{4}(\mathcal{O})} 
			 +\|\mathfrak{Z}(t)\|_{\H} +\|\f\|_{\H^{-1}(\mathcal{O})}\bigg]\|\boldsymbol{\psi}(t)\|_{\V}\d t
			 \nonumber\\&\leq C\int_{0}^{T}\bigg[ (1+N)\|\v^n(t)\|_{\V}
			 + N\|\mathfrak{Z}(t)\|_{\L^{4}(\mathcal{O})} 
			 +\|\mathfrak{Z}(t)\|_{\H} +\|\f\|_{\H^{-1}(\mathcal{O})}\bigg]\|\boldsymbol{\psi}(t)\|_{\V}\d t
		\nonumber\\&\leq C \bigg[(1+N)\|\v^n\|_{\mathrm{L}^{2}(0,T;\V)}
		+ N\|\mathfrak{Z}\|_{\mathrm{L}^{2}(0,T;\L^{4}(\mathcal{O}))} +\|\mathfrak{Z}\|_{\mathrm{L}^{2}(0,T;\H)}
		 + T^{\frac{1}{2}}\|\f\|_{\H^{-1}(\mathcal{O})}\bigg]\|\boldsymbol{\psi}\|_{\mathrm{L}^2(0,T;\V)},
		\end{align*}
		which implies $\frac{\d \v^n}{\d t}\in \mathrm{L}^{2}(0,T;\V')$ and $\mathrm{P}_n\mathcal{G}_{N}(\v^n)\in \mathrm{L}^{2}(0,T;\V')$, where $\mathcal{G}_{N}(\v^n):=\nu\A\v^n+\B_{N}(\v^n+\mathfrak{Z})$. Using \eqref{S5} and the \emph{Banach-Alaoglu theorem}, we infer the existence of an element $\v\in\mathrm{L}^{\infty}(0,T;\H)\cap\mathrm{L}^{2}(0,T;\V)$ with $\frac{\d \v}{\d t}\in \mathrm{L}^{2}(0,T;\V')$ and $\mathcal{G}_{0,N}\in\mathrm{L}^{2}(0,T;\V')$ such that
		\begin{align}
			\v^n\xrightharpoonup{w^*}&\ \v &&\text{ in }\ \ \ \ \	\mathrm{L}^{\infty}(0,T;\H),\label{S7}\\
			\v^n\xrightharpoonup{w}&\ \v   && \text{ in } \ \ \ \ \ \mathrm{L}^{2}(0,T;\V),\label{S8}\\
			\frac{\d \v^n}{\d t}\xrightharpoonup{w}&\frac{\d \v}{\d t}   && \text{ in }  \ \ \ \ \ \mathrm{L}^{2}(0,T;\V'),\label{S8d}\\
			\mathrm{P}_n\mathcal{G}_{N} (\v^n) \xrightharpoonup{w}& \mathcal{G}_{0,N}   && \text{ in }  \ \ \ \ \ \mathrm{L}^{2}(0,T;\V'),\label{S9d}
		\end{align}
		along a subsequence (still denoted by the same symbol). Note that $\chi\mathfrak{Z}_n(t)+\f_n\to \chi\mathfrak{Z}(t)+\mathcal{P}\f $ in $\mathrm{L}^2(0,T;\V^{\prime})$. Therefore, on passing to limit as $n\to\infty$ in \eqref{finite-dimS}, the limit $\v(\cdot)$ satisfies:
		\begin{equation}
			\left\{
			\begin{aligned}
				\frac{\d\v}{\d t}&=-\mathcal{G}_{0,N} +\chi\mathfrak{Z} +\mathcal{P}\f, && \text{  in }\  \mathrm{L}^{2}(0,T;\V'),\\
				\v(0)&=\v_0, && \text{  in }\  \H.
			\end{aligned}
			\right.
		\end{equation}
		Since $\v\in\mathrm{L}^{2}(0,T;\V)$ and $\frac{\d \v}{\d t}\in \mathrm{L}^{2}(0,T;\V')$, we have (cf. \cite[Ch. I\!I\!I, Lemma 1.2]{Temam}) $\v\in \C([0,T];\H)$, the real-valued function $t\mapsto\|\v(t)\|_{\H}^2$ is absolutely continuous and the following equality is satisfied:
		\begin{align}\label{EE1}
			 \frac{\d}{\d t}\|\v(t)\|_{\H}^2 = 2 \left<\frac{\d \v(t)}{\d t},\v(t)  \right>,\  \ \ \ \text{ for a.e. } t\in[0,T].
		\end{align}
	Hence, we have the following equality for any $\eta\geq0$:
	\begin{align}\label{EE2}
		e^{-2\eta t}\|\v(t)\|_{\H}^2 + 2 \int_{0}^{t}e^{-2\eta s}\left\langle  \mathcal{G}_{0,N}(s) - \chi\mathfrak{Z}(s) - \mathcal{P}\f + \eta \v(s), \v(s) \right\rangle \d s =\|\v(0)\|_{\H}^2,
	\end{align}
for all $t\in(0,T]$. In particular, we will choose $\eta=\frac{  7^{7}\cdot N^8}{2^{13}\cdot \nu^{7}}$ in the sequel. Similar to \eqref{EE2},  for  the system \eqref{finite-dimS}, we obtain the following energy equality:
\begin{align}\label{EE3}
	e^{-2\eta t}\|\v^n(t)\|_{\H}^2 + 2 \int_{0}^{t}e^{-2\eta s}\left\langle  \mathcal{G}_{N}(\v^n(s)) - \chi\mathfrak{Z}_n(s) - \f_n + \eta \v^n(s), \v^n(s) \right\rangle \d s =\|\v^n(0)\|_{\H}^2,
\end{align}
for all $t\in(0,T]$.

\vskip 2mm
\noindent
\textbf{Claim I:} \textit{The operator $\mathcal{G}_{N}:\V\to \V^{\prime}$ is demicontinuous.} Let us take a sequence $\u^n\to\u$ in $\V$, that is, $\|\u^n-\u\|_{\V}\to 0$ as $n\to\infty$. For any $\mathfrak{Z}\in\H\cap\L^4(\mathcal{O})$ and $\boldsymbol{\psi}\in\V$, we consider
\begin{align}\label{DM1}
	\langle\mathcal{G}_N(\u^n)-\mathcal{G}_N(\u),\boldsymbol{\psi}\rangle &=\nu \langle \A\u^n-\A\u,\boldsymbol{\psi}\rangle+\langle\B_N(\u^n+\mathfrak{Z})-\B_N(\u+\mathfrak{Z}),\boldsymbol{\psi}\rangle.
\end{align} 
Next, we take $\langle \A\u^n-\A\u,\boldsymbol{\psi}\rangle$ from \eqref{DM1} and estimate as follows:
\begin{align}
	|\langle \A\u^n-\A\u,\boldsymbol{\psi}\rangle|=|(\nabla(\u^n-\u),\nabla\boldsymbol{\psi})|\leq\|\u^n-\u\|_{\V}\|\boldsymbol{\psi}\|_{\V}\to 0, \ \text{ as } \ n\to\infty, 
\end{align}
since $\u^n\to \u$ in $\V$. We estimate the term $\langle\B_N(\u^n+\mathfrak{Z})-\B_N(\u+\mathfrak{Z}),\boldsymbol{\psi}\rangle$ from \eqref{DM1} using \eqref{b0}, H\"older's inequality, \eqref{FN1}, \eqref{FN2}, \eqref{lady} and \eqref{poin} as  
\begin{align}
	&|\langle\B_N(\u^n+\mathfrak{Z})-\B_N(\u+\mathfrak{Z}),\boldsymbol{\psi}\rangle|
	\nonumber\\ & = |F_N(\|\u^n+\mathfrak{Z}\|_{\L^4(\mathcal{O})})\cdot b(\u^n+\mathfrak{Z},\u^n+\mathfrak{Z},\boldsymbol{\psi})-F_N(\|\u+\mathfrak{Z}\|_{\L^4(\mathcal{O})})\cdot b(\u+\mathfrak{Z},\u+\mathfrak{Z},\boldsymbol{\psi})|
	\nonumber\\ & \leq  |\left[F_N(\|\u^n+\mathfrak{Z}\|_{\L^4(\mathcal{O})})-F_N(\|\u+\mathfrak{Z}\|_{\L^4(\mathcal{O})})\right]\cdot b(\u^n+\mathfrak{Z},\boldsymbol{\psi},\u^n+\mathfrak{Z})|
	\nonumber\\ & \quad + |F_N(\|\u+\mathfrak{Z}\|_{\L^4(\mathcal{O})})\cdot \left[b(\u^n-\u,\boldsymbol{\psi}, \u^n-\u) +b(\u^n-\u,\boldsymbol{\psi}, \u+\mathfrak{Z}) + b(\u+\mathfrak{Z},\boldsymbol{\psi}, \u^n-\u)\right]|
	\nonumber\\ & \leq  \frac{1}{N}F_N(\|\u^n+\mathfrak{Z}\|_{\L^4(\mathcal{O})})F_N(\|\u+\mathfrak{Z}\|_{\L^4(\mathcal{O})})\cdot \|\u^n-\u\|_{\L^4(\mathcal{O})} \|\u^n+\mathfrak{Z}\|_{\L^4(\mathcal{O})}^2\|\boldsymbol{\psi}\|_{\V}
	\nonumber\\ & \quad + F_N(\|\u+\mathfrak{Z}\|_{\L^4(\mathcal{O})})\cdot \bigg[\|\u^n-\u\|^2_{\L^4(\mathcal{O})}  +\|\u^n-\u\|_{\L^4(\mathcal{O})}  \|\u+\mathfrak{Z}\|_{\L^4(\mathcal{O})} 
	\nonumber\\ & \qquad + \|\u+\mathfrak{Z}\|_{\L^4(\mathcal{O})}  \|\u^n-\u\|_{\L^4(\mathcal{O})}\bigg]\|\boldsymbol{\psi}\|_{\V}
	\nonumber\\ & \leq  N \|\u^n-\u\|_{\V} \|\u^n+\mathfrak{Z}\|_{\L^4(\mathcal{O})}\|\boldsymbol{\psi}\|_{\V}
	 + \bigg[\|\u^n-\u\|^2_{\V}  + N \|\u^n-\u\|_{\V}    + N \|\u^n-\u\|_{\V} \bigg]\|\boldsymbol{\psi}\|_{\V}
	\nonumber\\& \to 0, \ \text{ as } \ n\to\infty, 
\end{align}
since $\u^n\to\u$ in $\V$ and $\u^n\in\V\subset \L^4(\mathcal{O})$. From the above convergences, it is immediate that $\langle\mathcal{G}_{N}(\u^n)-\mathcal{G}_{N}(\u),\boldsymbol{\psi}\rangle \to 0$, for all $\boldsymbol{\psi}\in \V$. Hence the operator $\mathcal{G}_{N}:\V \to \V'$ is demicontinuous, which implies that the operator $\mathcal{G}_{N}(\cdot)$ is hemicontinuous also. 

\vskip 2mm
\noindent
\textbf{Claim I\!I:} \textit{The operator $\mathcal{G}_{N}(\cdot):=\nu\A\cdot+\B_{N}(\cdot+\mathfrak{Z})$ satisfies
\begin{align}\label{CL1}
	\left< \mathcal{G}_{N}(\v_1)-\mathcal{G}_{N}(\v_2), \v_1-\v_2   \right> + \frac{ 7^{7}\cdot N^8}{2^{13}\cdot \nu^{7}} \|\v_1-\v_2\|_{\H}^2 \geq \frac{\nu}{2} \|\v_1-\v_2\|^2_{\V}\geq 0,
\end{align}
for any  $\v_1,\v_2\in\V.$}

		We know that 
		\begin{align}\label{CL2}
		\left<	\nu\A\v_2 - \nu\A\v_2 , \v_1 -\v_2 \right> = \nu\|\v_1-\v_2\|_{\V}^2.
		\end{align}
	Now, using \eqref{b0}, H\"older's inequality, \eqref{FN1}, \eqref{FN2}, \eqref{lady}  and Young's inequality, we estimate 
	\begin{align*}
	& |\left<	\B_{N}(\v_1+\mathfrak{Z}) - \B_{N}(\v_2+\mathfrak{Z}), \v_1 - \v_2 \right>| 
	\nonumber\\ & \leq   \left|\left[F_N(\|\v_1+\mathfrak{Z}\|_{\L^4(\mathcal{O})})- F_N(\|\v_2+\mathfrak{Z}\|_{\L^4(\mathcal{O})})\right]\cdot b(\v_1+\mathfrak{Z},\v_1-\v_2, \v_2+\mathfrak{Z} )\right| 
	\nonumber\\ & \quad + \left| F_N(\|\v_2+\mathfrak{Z}\|_{\L^4(\mathcal{O})}) \cdot  b(\v_1-\v_2, \v_1-\v_2 ,\v_2+\mathfrak{Z})  \right|
	\nonumber\\ & \leq   \frac{1}{N}F_N(\|\v_1+\mathfrak{Z}\|_{\L^4(\mathcal{O})})F_N(\|\v_2+\mathfrak{Z}\|_{\L^4(\mathcal{O})})\cdot\|\v_1-\v_2\|_{\L^4(\mathcal{O})} \|\v_1+\mathfrak{Z}\|_{\L^4(\mathcal{O})}     \|\v_1-\v_2\|_{\V}\nonumber\\ & \qquad \times  \|\v_2+\mathfrak{Z}\|_{\L^4(\mathcal{O})} 
	+ F_N(\|\v_2+\mathfrak{Z}\|_{\L^4(\mathcal{O})}) \cdot  \|\v_1-\v_2\|_{\L^4(\mathcal{O})} \|\v_1-\v_2\|_{\V}  \|\v_2+\mathfrak{Z}\|_{\L^4(\mathcal{O})}
	\nonumber\\ & \leq   2 N \|\v_1-\v_2\|_{\L^4(\mathcal{O})}     \|\v_1-\v_2\|_{\V} 
	\nonumber\\ & \leq   2^{\frac{3}{2}} N \|\v_1-\v_2\|^{\frac14}_{\H}   \|\v_1-\v_2\|^{\frac74}_{\V} 
	\nonumber\\ & \leq \frac{\nu}{2} \|\v_1-\v_2\|^2_{\V}  + \frac{7^{7}\cdot N^8}{2^{13}\cdot \nu^{7}} \|\v_1-\v_2\|^{2}_{\H},
	\end{align*}
		which implies  
		\begin{align}\label{CL3}
			& \left<	\B_{N}(\v_1+\mathfrak{Z}) - \B_{N}(\v_2+\mathfrak{Z}), \v_1 - \v_2 \right>
			 \geq  - \frac{\nu}{2} \|\v_1-\v_2\|^2_{\V}  - \frac{ 7^{7}\cdot N^8}{2^{13}\cdot \nu^{7}} \|\v_1-\v_2\|^{2}_{\H}.
		\end{align}
	In view of \eqref{CL2} and \eqref{CL3}, we obtain \eqref{CL1}.
	\vskip 2mm
	\noindent
	\textbf{Minty-Browder technique:} Remember that $\v^n(0)=\mathrm{P}_n\v(0)$, and hence the initial value $\v^n(0)$ converges strongly in $\H$, that is, we have
	\begin{align}\label{MB1}
		\lim_{n\to\infty} \|\v^n(0)-\v(0)\|_{\H}=0.
	\end{align} 
For any $\boldsymbol{\psi}\in \mathrm{L}^{\infty}(0,T;\H_{m})$ with $m<n$, using \eqref{CL1}, we obtain for $\eta=\frac{ 7^{7}\cdot N^8}{2^{13}\cdot \nu^{7}}$
\begin{align}\label{MB2}
	\int_{0}^{T} e^{-2\eta t}\{\left<\mathcal{G}_N(\boldsymbol{\psi}(t))-\mathcal{G}_N(\v_n(t)),\boldsymbol{\psi}(t)-\v_n(t)\right>  + \eta (\boldsymbol{\psi}(t)-\v_n(t), \boldsymbol{\psi}(t)-\v_n(t))\}\d t \geq0.
\end{align}
		Making use of \eqref{EE3} in \eqref{MB2}, we obtain
		\begin{align}\label{MB3}
		&	\int_{0}^{T} e^{-2\eta t} \left<\mathcal{G}_N(\boldsymbol{\psi}(t)) +\eta \boldsymbol{\psi} (t),\boldsymbol{\psi}(t)-\v_n(t)\right> \d t
		\nonumber\\ &  \geq \int_{0}^{T} e^{-2\eta t}\left<\mathcal{G}_N(\v_n(t))+\eta \v^n(t),\boldsymbol{\psi}(t)-\v_n(t)\right> \d t 
		\nonumber\\ &  = \int_{0}^{T} e^{-2\eta t}\left<\mathcal{G}_N(\v_n(t))+\eta \v^n(t),\boldsymbol{\psi}(t)\right> \d t  + \frac12\bigg[e^{-2\eta T}\|\v^n(T)\|_{\H}^2- \|\v^n(0)\|_{\H}^2\bigg]
		\nonumber\\ & \qquad -  \int_{0}^{T}e^{-2\eta t}\left\langle   \chi\mathfrak{Z}_n(t) + \f_n , \v^n(t) \right\rangle \d t. 
		\end{align}
	Taking limit infimum on both sides of \eqref{MB3}, we deduce
	\begin{align}\label{MB4}
		&	\int_{0}^{T} e^{-2\eta t} \left<\mathcal{G}_N(\boldsymbol{\psi}(t)) +\eta \boldsymbol{\psi} (t),\boldsymbol{\psi}(t)-\v(t)\right> \d t
		\nonumber\\ &  \geq  \int_{0}^{T} e^{-2\eta t}\left<\mathcal{G}_{0,N}(t)+\eta \v(t),\boldsymbol{\psi}(t)\right> \d t  + \frac12 \liminf_{n\to \infty}\bigg[e^{-2\eta T}\|\v^n(T)\|_{\H}^2- \|\v^n(0)\|_{\H}^2\bigg]
		\nonumber\\ & \qquad -  \int_{0}^{T}e^{-2\eta t}\left\langle   \chi\mathfrak{Z}(t) + \mathcal{P} \f , \v(t) \right\rangle \d t
		\nonumber\\ &  \geq  \int_{0}^{T} e^{-2\eta t}\left<\mathcal{G}_{0,N}(t)+\eta \v(t),\boldsymbol{\psi}(t)\right> \d t  + \frac12 \bigg[e^{-2\eta T}\|\v(T)\|_{\H}^2- \|\v(0)\|_{\H}^2\bigg]
		\nonumber\\ & \qquad -  \int_{0}^{T}e^{-2\eta t}\left\langle   \chi\mathfrak{Z}(t) + \mathcal{P}\f , \v(t) \right\rangle \d t,
	\end{align}
	where we have used the  lower semicontinuity property of the $\H$-norm and the strong convergence of the initial data \eqref{MB1} in the final inequality. Now, using the equality \eqref{EE2} in \eqref{MB4}, we further have 
	\begin{align}\label{MB5}
		&	\int_{0}^{T} e^{-2\eta t} \left<\mathcal{G}_N(\boldsymbol{\psi}(t)) +\eta \boldsymbol{\psi} (t),\boldsymbol{\psi}(t)-\v(t)\right> \d t
		\nonumber\\ &  \geq   \int_{0}^{T} e^{-2\eta t}\left<\mathcal{G}_{0,N}(t)+\eta \v(t),\boldsymbol{\psi}(t)\right> \d t  - \int_{0}^{T} e^{-2\eta t}\left<\mathcal{G}_{0,N}(t)+\eta \v(t),\v(t)\right> \d t
		\nonumber\\ &  \geq   \int_{0}^{T} e^{-2\eta t}\left<\mathcal{G}_{0,N}(t)+\eta \v(t),\boldsymbol{\psi}(t) - \v(t)\right> \d t.  
	\end{align}
	Note that the estimate \eqref{MB5} holds true for any
	$\boldsymbol{\psi}\in\mathrm{L}^{\infty}(0,T;\H_m)$, $m\in\mathbb{N}$, since the  inequality given in \eqref{MB5} is
	independent of both $m$ and $n$. Using a density
	argument, one can show that the inequality \eqref{MB5} remains true for any
	$\boldsymbol{\psi}\in\mathrm{L}^{\infty}(0,T;\H)\cap\mathrm{L}^2(0,T;\V).$ In fact, for any
	$\boldsymbol{\psi}\in\mathrm{L}^{\infty}(0,T;\H)\cap\mathrm{L}^2(0,T;\V),$ there	exists a strongly convergent subsequence	$\boldsymbol{\psi}_m\in\mathrm{L}^{\infty}(0,T;\H)\cap\mathrm{L}^2(0,T;\V),$ that
	satisfies the inequality \eqref{MB5}.
	
	Taking $\boldsymbol{\psi}=\v+ r \w$, $r>0$, where $\w  \in\mathrm{L}^{\infty}(0,T;\H)\cap\mathrm{L}^2(0,T;\V),$ and substituting for $\boldsymbol{\psi}$ in \eqref{MB5}, we get
	\begin{align}\label{MB6}
		&	\int_{0}^{T} e^{-2\eta t} \left<\mathcal{G}_N(\v(t)+ r \w(t)) - \mathcal{G}_{0,N}(t) +\eta r \w(t), r \w(t)\right> \d t
		   \geq   0.
	\end{align}
Dividing the inequality \eqref{MB6} by $r$, using the
hemicontinuity property of the operator $\mathcal{G}_N(\cdot)$ (see \textbf{Claim I}), and then passing $r\to 0$, we find 
	\begin{align}\label{MB7}
		&	\int_{0}^{T} e^{-2\eta t} \left<\mathcal{G}_N(\v(t)) - \mathcal{G}_{0,N}(t),  \w(t)\right> \d t
		\geq   0,
	\end{align}
for any $\w  \in \mathrm{L}^{\infty}(0,T;\H)\cap\mathrm{L}^2(0,T;\V).$ Therefore, from \eqref{MB7}, we deduce that $\mathcal{G}_{N}(\v(\cdot))=\mathcal{G}_{0,N}(\cdot).$ Hence, we obtain the following weak convergence from \eqref{S9d}:
\begin{align}\label{GN-w-con}
	\mathcal{G}_{N} (\v^n) \xrightharpoonup{w} \mathcal{G}_{N} (\v) \ \ \  \text{ in }  \ \ \ \mathrm{L}^{2}(0,T;\V').
\end{align}
Moreover, in view of \eqref{GN-w-con} and the weak convergence $\A\v^n\xrightharpoonup{w}\A \v$ in  $\mathrm{L}^{2}(0,T;\V')$, we also derive the following weak convergence:
\begin{align}\label{BN-w-con}
	\B_{N} (\v^n+\mathfrak{Z}) \xrightharpoonup{w} \B_{N} (\v+\mathfrak{Z}) \ \ \  \text{ in }  \ \ \ \mathrm{L}^{2}(0,T;\V').
\end{align}
Furthermore, $\v(\cdot)$ satisfies the following energy equality for all $t\in[0,T]$ (see \eqref{EE2}):
	\begin{align}\label{EE4}
		&\|\v(t)\|_{\H}^2+2\nu\int_0^t\|\v(s)\|_{\V}^2\d s + 2\int_0^t\langle\B_N(\v(s)+\mathfrak{Z}(s)),\v(s)\rangle\d s  \nonumber\\&= \|\v_0\|_{\H}^2 +2\int_0^t\langle\f,\v(s)\rangle\d s+2\chi\int_0^t(\mathfrak{Z}(s),\v(s))\d s.
	\end{align}
		\vskip 2mm
		\noindent
		\textbf{Step II.}	\emph{Uniqueness:} Define $\mathfrak{F}=\v_1-\v_2$, where $\v_1$ and $\v_2$ are two weak solutions of the system \eqref{csgmnse} in the sense of Definition \ref{defn5.9}. Then $\mathfrak{F}\in\mathrm{C}([0,T];\H)\cap\mathrm{L}^{2}(0,T;\V)$ and satisfies
		\begin{equation}\label{Uni}
			\left\{
			\begin{aligned}
				\frac{\d\mathfrak{F}(t)}{\d t} &= -\nu \A\mathfrak{F} (t)  - \B_N(\v_1(t)+\mathfrak{Z}(t))+\B_N(\v_2(t)+\mathfrak{Z}(t))\\ & = -\left[\mathcal{G}_{N}(\v_1(t))-\mathcal{G}_{N}(\v_2(t))\right], \\
				\mathfrak{F}(0)&= \textbf{0},
			\end{aligned}
			\right.
		\end{equation}
		in the weak sense.	From the above equation, using  the energy equality, we obtain
		\begin{align}\label{U1}
			&	\frac{1}{2}\frac{\d}{\d t}\|\mathfrak{F}(t)\|^2_{\H} = - \left<\mathcal{G}_{N}(\v_1(t))-\mathcal{G}_{N}(\v_2(t)), \v_1(t)-\v_2(t)\right> \leq \frac{ 7^{7}\cdot N^8}{2^{13}\cdot \nu^{7}} \|\mathfrak{F}(t)\|_{\H}^2,
		\end{align}
where we have used \eqref{CL1}. Applying the variation of constants formula and using the fact that $\mathfrak{F}(0)=\textbf{0}$, we obtain $\v_1(t)=\v_2(t)$, for all $t\in[0,T]$ in $\H$, which proves the uniqueness.
	\end{proof}	
\begin{remark}\label{rem3.7}
	It is remarkable to note that the proof of existence of solutions in the case of bounded domains is bit easier.  We know that the embedding $\V\hookrightarrow\H$ is compact for a bounded domain $\mathcal{O}$. In view of convergences \eqref{S8}-\eqref{S8d} and the Aubin–Lions compactness lemma (cf. \cite[Theorem 5, Corollary 4]{Simon_1987}), one can obtain (along a subsequence) 
	\begin{align}\label{C-H}
		\v^n\to \v \ \  \text{ in } \ \ \ \ \ \mathrm{L}^{2}(0,T;\H),
	\end{align}
which helps us to pass limit in the nonlinear term, since the above convergence implies
\begin{align}\label{351}
	\v^n\to \v \ \  \text{ in } \ \ \ \ \  \text{ for a.e. }\ (t,x)\in(0,T)\times\mathcal{O},
\end{align}
along a further subsequence. The continuity of $F_N(\cdot)$ and the convergence \eqref{351}  imply
\begin{align}
	F_{N}(\|\v^n(t)\|_{\L^4(\O)})\to F_{N}(\|\v(t)\|_{\L^4(\O)}) \ \ \ \ \ \text{ for a.e. } \ t\in(0,T). 
\end{align}

\end{remark}

	Next, we show that the weak solution of system \eqref{csgmnse} is continuous with respect to given data (particularly $\x$, $\f$ and $\mathfrak{Z}$).
	
	\begin{theorem}\label{RDS_Conti1}
		For some $T >0$ fixed, assume that $\boldsymbol{x}_n \to \boldsymbol{x}$ in $\H$, $\f_n \to \f \ \text{ in }\ \H^{-1}(\mathcal{O})$ and $\mathfrak{Z}_n \to \mathfrak{Z}\ \text{ in }\ \mathrm{L}^2 (0, T; \H\cap\L^{4}(\mathcal{O}))$.  Let us denote by $\v(t, \mathfrak{Z})\boldsymbol{x},$ the solution of the system \eqref{csgmnse} and by $\v(t, \mathfrak{Z}_n)\boldsymbol{x}_n,$  the solution of the system \eqref{csgmnse} with $\mathfrak{Z}, \f, \boldsymbol{x}$ being replaced by $\mathfrak{Z}_n, \f_n, \boldsymbol{x}_n$. Then \begin{align}\label{5.20}
			\v(\cdot, \mathfrak{Z}_n)\boldsymbol{x}_n \to \v(\cdot, \mathfrak{Z})\boldsymbol{x} \ \text{ in } \ \mathrm{C}([0,T];\H)\cap\mathrm{L}^2 (0, T;\V).
		\end{align}
		In particular, $\v(T, \mathfrak{Z}_n)\boldsymbol{x}_n \to \v(T, \mathfrak{Z})\boldsymbol{x}$ in $\H$.
	\end{theorem}	
	\begin{proof}
		Let us introduce the following notations which help us to simplify the proof: 
		\begin{align*}
			\v_n (\cdot) &= \v(\cdot, \mathfrak{Z}_n)\boldsymbol{x}_n, \ \  \v(\cdot) = \v(\cdot, \mathfrak{Z})\boldsymbol{x},\ \   \mathfrak{X}_n (\cdot)= \v(\cdot, \mathfrak{Z}_n)\boldsymbol{x}_n - \v(\cdot, \mathfrak{Z})\boldsymbol{x},\\     
			 \hat{ \mathfrak{Z}}_n(\cdot) &= \mathfrak{Z}_n(\cdot) - \mathfrak{Z}(\cdot), \ \  \hat{ \f}_n(\cdot) = \mathcal{P}\f_n(\cdot) - \mathcal{P}\f(\cdot).
		\end{align*}
		Then $\mathfrak{X}_n$ satisfies the following system:
		\begin{equation}\label{finite-dimS_1}
			\left\{
			\begin{aligned}
				\frac{\d\mathfrak{X}_n}{\d t} &= -\nu \A\mathfrak{X}_n  - \B_N(\v_n  + \mathfrak{Z}_n ) + \B_N(\v + \mathfrak{Z}) + \chi \hat{ \mathfrak{Z}}_n + \hat{ \f}_n
				\\ & = -\left[\mathcal{G}_{N}(\v_n)-\mathcal{G}_{N}(\v)\right] + \chi \hat{ \mathfrak{Z}}_n + \hat{ \f}_n, \\
				\mathfrak{X}_n(0)&= \boldsymbol{x}_n - \boldsymbol{x}.
			\end{aligned}
			\right.
		\end{equation}
		Multiplying by $\mathfrak{X}_n(\cdot)$ to the first equation in \eqref{finite-dimS_1}, integrating over $\mathcal{O}$ and using \eqref{CL1}, we obtain 
		\begin{align}\label{5.22}
			\frac{1}{2}& \frac{\d}{\d t}\|\mathfrak{X}_n(t) \|^2_{\H}
			\nonumber\\
			&= -\left<\mathcal{G}_{N}(\v_n(t))-\mathcal{G}_{N}(\v(t)), \mathfrak{X}_n(t)\right>  + \chi( \hat{ \mathfrak{Z}}_n(t), \ \mathfrak{X}_n(t))+ \langle\hat{ \f}_n, \mathfrak{X}_n(t)\rangle
			  \nonumber\\
			& \leq - \frac{\nu}{2} \|\mathfrak{X}_n(t)\|^2_{\V}   + \frac{ 7^{7}\cdot N^8}{2^{13}\cdot \nu^{7}}\|\mathfrak{X}_n(t)\|^2_{\H} + \chi(\hat{ \mathfrak{Z}}_n(t), \ \mathfrak{X}_n(t))+ \langle\hat{ \f}_n, \mathfrak{X}_n(t)\rangle,
		\end{align}
		for a.e. $t\in [0,T]$. Using  H\"older's inequality, \eqref{poin} and Young's inequality, we have 
		\begin{align}\label{5.23}
			|\chi(\hat{ \mathfrak{Z}}_n, \ \mathfrak{X}_n)|& \leq \chi \|\mathfrak{X}_n\|_{\H} \|\hat{ \mathfrak{Z}}_n\|_{\H} \leq C \|\mathfrak{X}_n\|_{\V} \|\hat{ \mathfrak{Z}}_n\|_{\V} \leq \frac{\nu}{8}\|\mathfrak{X}_n\|^2_{\V}+C\|\hat{ \mathfrak{Z}}_n\|^2_{\H},\\
			|	\langle\hat{ \f}_n, \ \mathfrak{X}_n\rangle|& \leq  \|\mathfrak{X}_n\|_{\V} \|\hat{ \f}_n\|_{\H^{-1}(\mathcal{O})}\leq \frac{\nu}{8}\|\mathfrak{X}_n\|^2_{\V} + C \|\hat{ \f}_n\|^2_{\H^{-1}(\mathcal{O})}.
		\end{align} 
	 Combining \eqref{5.22} and \eqref{5.23}, we deduce 
		\begin{align*}
			&\frac{\d}{\d t}\|\mathfrak{X}_n(t) \|^2_{\H} + \frac{\nu}{2} \|\mathfrak{X}_n(t)\|^2_{\V}    \leq C N^8 \|\mathfrak{X}_n(t)\|^2_{\H}  +  C \|\hat{ \mathfrak{Z}}_n(t)\|^2_{\H} + C \|\hat{\f}_n\|^2_{\H^{-1}(\mathcal{O})}, 
		\end{align*}
		for a.e. $t\in[0,T]$.	In view of Gronwall's inequality, we obtain
		\begin{align}\label{Energy_esti_n_2_1}
			&\|\mathfrak{X}_n(t)\|^2_{\H} + \frac{\nu}{2} \int_{0}^{t} \|\mathfrak{X}_n(s) \|^2_{\V}\d s    \leq \biggl\{\|\mathfrak{X}_n(0)\|^2_{\H}  + C \int_{0}^{t} \big[\|\hat{ \mathfrak{Z}}_n(s)\|^2_{\H} + \|\hat{\f}_n\|^2_{\H^{-1}(\mathcal{O})}\big] \d s\biggr\}e^{CN^8t}, 
		\end{align}
		for all $ t\in[0, T]$.	Since, $\|\mathfrak{X}_n(0)\|_{\H} = \|\boldsymbol{x}_n- \boldsymbol{x}\|_{\H} \to 0 $ and $\int_{0}^{T} \big[\|\hat{ \mathfrak{Z}}_n(s)\|^2_{\H} + \|\hat{\f}_n\|^2_{\H^{-1}(\mathcal{O})}\big]\ \d s \to 0$ as $n\to \infty$, then \eqref{Energy_esti_n_2_1} asserts that $\|\mathfrak{X}_n(t)\|_{\H} + \frac{\nu}{2} \int_{0}^{t} \|\mathfrak{X}_n(s) \|^2_{\V}\d s \to 0$ as $n\to\infty$ uniformly in $t\in[0, T].$ Since $\v_n(\cdot)$ and $\v(\cdot)$ are continuous, we further have  $$\v(\cdot, \mathfrak{Z}_n)\boldsymbol{x}_n \to \v(\cdot, \mathfrak{Z})\boldsymbol{x}\  \text{ in } \ \C([0, T]; \H)\cap \mathrm{L}^2(0, T; \V).$$
		This completes the proof.
	\end{proof}
	\begin{definition}
		We define a map $\Phi_{\chi} : [0,\infty) \times \Omega \times \H \to \H$ by
		\begin{align}
			(t, \omega, \boldsymbol{x}) \mapsto \v^{\chi}(t)  + \mathfrak{Z}_{\chi}(\omega)(t) \in \H,
		\end{align}
		where $\v^{\chi}(t) = \v(t, \mathfrak{Z}_{\chi}(\omega)(t))(\boldsymbol{x} - \mathfrak{Z}_{\chi}(\omega)(0))$ is a solution to the system \eqref{csgmnse} with the initial condition $\boldsymbol{x} - \mathfrak{Z}_{\chi}(\omega)(0).$
	\end{definition}
	\begin{proposition}\label{alpha_ind}
		If $\chi_1, \chi_2 \geq 0$, then $\Phi_{\chi_1} = \Phi_{\chi_2}.$
	\end{proposition}
	\begin{proof}
		Let us fix $\boldsymbol{x}\in \H.$ We need to prove that $$\v^{\chi_1}(t) + \mathfrak{Z}_{\chi_1}(t) = \v^{\chi_2}(t) + \mathfrak{Z}_{\chi_2}(t), \ \ \ t\geq 0, $$ where $\mathfrak{Z}_{\chi}$ is defined by \eqref{DOu1} and $\v^{\chi}$ is a solution to the system \eqref{csgmnse}. 
		
		Since $\mathfrak{Z}_{\chi_1}, \mathfrak{Z}_{\chi_2} \in \mathrm{C}([0,T];\H\cap \mathbb{L}^4(\mathcal{O}))$ and  $\mathfrak{Z}_{\chi_2}-\mathfrak{Z}_{\chi_1}$  satisfies \eqref{Dif_z1}, it implies from  \cite[Vol. I\!I, Theorem 3.2, p.22]{Lion-Mag} that  
		\begin{align}\label{359}
			\mathfrak{Z}_{\chi_2}-\mathfrak{Z}_{\chi_1}\in \mathrm{C}([0,T];\H)\cap \mathrm{L}^2(0,T;\V) \ \text{ such that } \partial_{t}(\mathfrak{Z}_{\chi_2}-\mathfrak{Z}_{\chi_1}) \in \mathrm{L}^2(0,T,\V^\prime).
		\end{align}

		From \eqref{csgmnse}, we infer that $\v^{\chi_1}(0) - \v^{\chi_2}(0) = - (\mathfrak{Z}_{\chi_1}(0) - \mathfrak{Z}_{\chi_2}(0)) $ and 
		\begin{align*}
			\frac{\d(\v^{\chi_1}(t) - \v^{\chi_2}(t))}{\d t} = &-\nu \A(\v^{\chi_1}(t) - \v^{\chi_2}(t)) + \big[\chi_1 \mathfrak{Z}_{\chi_1}(t) -  \chi_2 \mathfrak{Z}_{\chi_2}(t)\big]\nonumber\\ &- [\B_N(\v^{\chi_1}(t) + \mathfrak{Z}_{\chi_1}(t))-\B_N(\v^{\chi_2}(t) + \mathfrak{Z}_{\chi_2}(t))], 
		\end{align*}
		for a.e. $t\in[0,T]$,	in $\V'$. It implies from Theorem \ref{solution} that 
		\begin{align}\label{360}
			\v^{\chi_1}-\v^{\chi_2}\in \mathrm{C}([0,T];\H)\cap \mathrm{L}^2(0,T;\V) \ \text{ such that } \ \partial_{t}(\v^{\chi_1}-\v^{\chi_2}) \in \mathrm{L}^2(0,T,\V^\prime).
		\end{align}
		 Adding the equation \eqref{Dif_z1} to the above equation, we obtain 
		\begin{align}\label{Enrgy}
			\frac{\d(\u^{\chi_1}(t) - \u^{\chi_2}(t))}{\d t}& = -\nu \A(\u^{\chi_1}(t) - \u^{\chi_2}(t)) - [\B_N(\u^{\chi_1}(t))-\B_N(\u^{\chi_2}(t))], 
		\end{align}
		for a.e. $t\in[0,T]$,	in $\V'$,	where $\u^{\chi_1}(t)=\v^{\chi_1}(t) + \mathfrak{Z}_{\chi_1}(t),   \u^{\chi_2}(t) = \v^{\chi_2}(t) + \mathfrak{Z}_{\chi_2}(t), \  t\geq 0$ and $\u^{\chi_1}(0) - \u^{\chi_2}(0)= \boldsymbol{0}.$ In view of \eqref{359} and \eqref{360}, we infer that 
		\begin{align*}
			\u^{\chi_1}-\u^{\chi_2}\in \mathrm{C}([0,T];\H)\cap \mathrm{L}^2(0,T;\V) \ \text{ such that } \ \partial_{t}(\u^{\chi_1}-\u^{\chi_2}) \in \mathrm{L}^2(0,T,\V^\prime).
		\end{align*}
	Therefore, from \cite[Ch. I\!I\!I, Lemma 1.2]{Temam}, we have 
		\begin{align*}
		\frac{\d}{\d t} \|\u^{\chi_1}(t) - \u^{\chi_2}(t)\|^2_\H= \big\langle -\nu \A(\u^{\chi_1}(t) - \u^{\chi_2}(t)) + \B_N(\u^{\chi_1}(t))-\B_N(\u^{\chi_2}(t)),  \u^{\chi_1}(t) - \u^{\chi_2}(t) \big\rangle.
	\end{align*}
	Applying similar steps as have performed for \eqref{CL1}, we arrive at
		\begin{align*}
			\frac{\d}{\d t} \|\u^{\chi_1}(t) - \u^{\chi_2}(t)\|^2_\H\leq  \frac{ 7^{7}\cdot N^8}{2^{12}\cdot \nu^{7}} \|\u^{\chi_1}(t) - \u^{\chi_2}(t)\|^2_\H,
		\end{align*}
		for a.e. $t\in[0,T]$. 	Since   $\|\u^{\chi_1}(0) - \u^{\chi_2}(0)\|^2_\H = 0,$ by applying variation of constants formula, we deduce  that $\|\u^{\chi_1}(t) - \u^{\chi_2}(t)\|^2_\H = 0$, for all $t\geq 0$, which completes the proof.
	\end{proof}
	It is proved in Proposition \ref{alpha_ind} that the map $\Phi_{\chi}$ does not depend on $\chi$ and hence, from now onward, it will be denoted by $\Phi$. A proof of the following result is similar to that in \cite[Theorem 6.15]{BL} and hence we omit it here.

	\begin{theorem}
		$(\Phi, \vartheta)$ is an RDS.
	\end{theorem}

	Now we are ready to present some fundamental properties of the solution to  system \eqref{S-GMNSE} with the initial data $\x\in\H$ at the initial time $s\in \R.$
The following theorem is a consequence of our previous discussion.
	\begin{theorem}\label{SGMNSE-Sol}
		In the framework of Definition \ref{Def_u}, suppose that $\u(t)=\mathfrak{Z}_{\chi}(t)+\v^{\chi}(t), t\geq s,$ where $\v^{\chi}$ is the unique solution to system \eqref{csgmnse} with initial data $\x - \mathfrak{Z}_{\chi}(s)$ at time $s$. If the process $\{\u(t), \ t\geq s\},$ has trajectories in $\mathrm{C}([s, \infty); \H) \cap  \mathrm{L}^{\frac83}_{\mathrm{loc}}([s, \infty); \L^4(\mathcal{O}))$, then it is a solution to system \eqref{S-GMNSE}. Vice-versa, if a process $\{\u(t), t\geq s\},$ with trajectories in $\mathrm{C}([s, \infty); \H) \cap \mathrm{L}^{\frac83}_{\mathrm{loc}}([s, \infty); \L^4(\mathcal{O}))$ is a solution to system \eqref{S-GMNSE}, then for any $\chi\geq 0,$ a process $\{\v^{\chi}(t), t\geq s\},$ defined by $\v^{\chi}(t) = \u(t)- \mathfrak{Z}_{\chi}(t), t\geq s,$ is a solution to \eqref{csgmnse} on $[s, \infty).$
	\end{theorem}

	\section{Random attractors for SGMNS equations}\label{sec6}\setcounter{equation}{0}
In this section, we prove the  main results of this work. Here, the RDS $\Phi$ is considered over the MDS $(\Omega,  {\mathcal{F}},  {\mathbb{P}},  {\vartheta})$.	The results that we have obtained in the previous sections  provide a unique solution to the system \eqref{S-GMNSE}, which is continuous with respect to the data (particularly $\x$ and $\f$). Furthermore, if we define, for $\x \in \H,\ \omega \in \Omega,$ and $t\geq s,$
	\begin{align}\label{combine_sol}
		\u(t, s;\omega, \x) := \Phi(t-s; \vartheta_s \omega)\x = \v\big(t, s; \omega, \x - \mathfrak{Z}(s)\big) + \mathfrak{Z}(t),
	\end{align}
	then the process $\{\u(t): \ t\geq s\}$ is a solution to the system \eqref{S-GMNSE}, for each $s\in \mathbb{R}$ and each $\x \in \H$.
	
	In next lemma, we give uniform estimates satisfied by the solution of system \eqref{csgmnse}.
	\begin{lemma}\label{RA1}
		Suppose that $\v$ solves the system \eqref{csgmnse} on the time interval $[a, \infty)$ with $\mathfrak{Z} \in  \mathrm{L}^{2}_{\mathrm{loc}}( \R; \H\cap{\L}^{4}(\mathcal{O}))$, $\f\in\H^{-1}(\mathcal{O})$ and $\chi\geq 0.$ Then, for any $t\geq \tau \geq a,$
		\begin{align}\label{Energy_esti1}
		&	\|\v(t)\|^2_{\H} + \frac{\nu}{2} \int_{\tau}^{t} \|\v(s)\|^{2}_{\V} e^{-\nu\lambda(t-s)} \d s 
			\nonumber\\ & \leq 
				\|\v(\tau)\|^2_{\H}\  e^{-\nu\lambda(t-\tau)}  + C\int_{\tau}^{t} \bigg[\|\mathfrak{Z}(s)\|^{2}_{\H}+ N^2\|\mathfrak{Z}(s)\|^{2}_{\L^{4}(\mathcal{O})} + \|\f\|^2_{\H^{-1}(\mathcal{O})}\bigg] e^{-\nu\lambda(t-s)} \d s,
		\end{align}
		for all $t\in [a, \infty).$
	\end{lemma}
	\begin{proof}
		From \eqref{csgmnse}, we obtain 
		\begin{align}\label{Energy_esti3}
			\frac{1}{2}\frac{\d}{\d t} \|\v(t)\|^2_{\H} 
			&=  - \nu \|\v(t)\|^2_{\V}   -\left<\B_N(\v(t)+ \mathfrak{Z}(t)),\v(t)\right> +\chi (\mathfrak{Z}(t),\v(t))+\big\langle \f, \v(t)\big \rangle,
		\end{align}
	for a.e. $t\in[0,T]$. Using \eqref{b0}, H\"older's inequality, \eqref{FN1}, \eqref{lady} and Young's inequality,  we obtain
		\begin{align*}
			|\left<\B_N(\v+ \mathfrak{Z}),\v\right>| &= |F_N(\|\v+ \mathfrak{Z}\|_{\L^4(\mathcal{O})}) \cdot b(\v+ \mathfrak{Z}, \v+ \mathfrak{Z}, \v)| 
			\nonumber\\&  = |F_N(\|\v+ \mathfrak{Z}\|_{\L^4(\mathcal{O})}) \cdot b(\v+ \mathfrak{Z},\v,  \mathfrak{Z})| 
			\nonumber\\& \leq F_N(\|\v+ \mathfrak{Z}\|_{\L^4(\mathcal{O})}) \|\v+ \mathfrak{Z}\|_{\L^4(\mathcal{O})}\|\v\|_{\V}\|  \mathfrak{Z}\|_{\L^4(\mathcal{O})}
			\nonumber\\& \leq N \|\v\|_{\V}\|  \mathfrak{Z}\|_{\L^4(\mathcal{O})}
			  \leq  \frac{\nu}{8} \|\v\|^2_{\V} + CN^2\|\mathfrak{Z}\|^2_{\L^4(\mathcal{O})}.
		\end{align*}		
			Using H\"older's  and Young's inequalities, we also obtain 
		\begin{align*}
			|\big\langle\chi\mathfrak{Z}+\f, \v\big \rangle| &\leq \big(\chi\|\mathfrak{Z}\|_{\V'} +\|\f\|_{\H^{-1}(\mathcal{O})}\big) \|\v\|_{\V}  \leq  \frac{\nu}{8} \|\v\|^2_{\V} +C\|\mathfrak{Z}\|^2_{\H}+C\|\f\|^2_{\H^{-1}(\mathcal{O})}.
		\end{align*}
		Hence, from \eqref{Energy_esti3} and \eqref{poin}, we deduce 
		\begin{align}
			\frac{\d}{\d t} \|\v(t)\|^2_{\H} & \leq  
				-\nu\lambda\|\v(t)\|^2_{\H}-\frac{\nu}{2}\|\v(t)\|^2_{\V} +C\bigg[\|\mathfrak{Z}(t)\|^{2}_{\H}+ N^2\|\mathfrak{Z}(t)\|^{2}_{{\L}^{4}(\mathcal{O})}
				 + \|\f\|^2_{\H^{-1}(\mathcal{O})}\bigg],\nonumber
		\end{align}
		and an application of  Gronwall's inequality yields \eqref{Energy_esti1}. This completes the proof.
	\end{proof}

In next two lemmas, we give some useful estimates satisfied by the solution of equation \eqref{OUPe1}.

\begin{lemma}\label{Bddns4}
	For each $\omega\in \Omega,$ we obtain 
	\begin{align*}
		\limsup_{t\to - \infty} \|\mathfrak{Z}(\omega)(t)\|^2_{\H}\  e^{\nu\lambda t} = 0.
	\end{align*}
\end{lemma}
\begin{proof}
	Let us fix $\omega\in \Omega$.  Because of \eqref{X_bound_of_z}, there exists a $\rho_1= \rho_1(\omega)\geq 0$ such that
	\begin{align}\label{rho}
		\frac{\|\mathfrak{Z}(t)\|_{\H}}{|t|} \leq \rho_1 \ \text{ and } \   \frac{\|\mathfrak{Z}(t)\|_{\L^4(\mathcal{O})}}{|t|} \leq \rho_1  \ \text{ for }\  t\leq t_0.
	\end{align}
	Therefore, we have, for every $\omega\in \Omega,$
	\begin{align*}
		\limsup_{t\to - \infty} \|\mathfrak{Z}(\omega)(t)\|^2_{\H}\  e^{\nu\lambda t}\leq&  \rho_1^2 \limsup_{t\to - \infty}  |t|^2 e^{\nu\lambda t} =0,
	\end{align*}
	which completes the proof.
\end{proof}
\begin{lemma}\label{Bddns5}
	For each $\omega\in \Omega,$ we get 
	\begin{align*}
		\int_{- \infty}^{0} \bigg\{ 1 + \|\mathfrak{Z}(t)\|^2_{\H} + N^2 \|\mathfrak{Z}(t)\|^2_{\L^4(\mathcal{O})} \bigg\}e^{\nu\lambda t} \d t < \infty.
	\end{align*}
\end{lemma}
\begin{proof}
	Note that for $t_0\leq 0$,
	\begin{align*}
		\int_{t_0}^{0} \bigg\{ 1 + \|\mathfrak{Z}(t)\|^2_{\H} +  N^2 \|\mathfrak{Z}(t)\|^2_{\L^4(\mathcal{O})} \bigg\}e^{\nu\lambda t } \d t < \infty.
	\end{align*}
	Therefore, we only need to show that the integral 
	\begin{align*}
		\int_{- \infty}^{t_0} \bigg\{ 1+ \|\mathfrak{Z}(t)\|^2_{\H} + N^2 \|\mathfrak{Z}(t)\|^2_{\L^4(\mathcal{O})}   \bigg\}e^{\nu\lambda t } \d t < \infty.
	\end{align*}
	In view of  \eqref{rho}, we obtain 
	\begin{align*} 
		&	\int_{- \infty}^{t_0} \bigg\{ \|\mathfrak{Z}(t)\|^2_{\H} +  N^2 \|\mathfrak{Z}(t)\|^2_{\L^4(\mathcal{O})} \bigg\}e^{\nu\lambda t } \d t  \leq \int_{- \infty}^{t_0} \big\{ \rho^2_1|t|^2+ N^2\rho^2_1|t|^2 \big\}e^{\nu\lambda t } \d t< \infty,
	\end{align*}
	which  completes the proof.
\end{proof}

\begin{definition}\label{RA2}
	A function $\kappa: \Omega\to (0, \infty)$ belongs to the class $\mathfrak{K}$ if and only if 
	\begin{align}
		\limsup_{t\to \infty} [\kappa(\vartheta_{-t}\omega)]^2 e^{-\nu\lambda t } = 0. 
	\end{align}
\end{definition}
Let us denote the class of all closed and bounded random sets $D$ on $\H$ by $\mathfrak{DK}$ such that the radius function $\Omega\ni \omega \mapsto \kappa(D(\omega)):= \sup\{\|x\|_{\H}:x\in D(\omega)\}$ belongs to the class $\mathfrak{K}.$ It is straight forward that the constant functions belong to $\mathfrak{K}$. 
It is clear by Definition \ref{RA2} that the class $\mathfrak{K}$ is closed with respect to sum, multiplication by a constant and if $\kappa \in \mathfrak{K}, 0\leq \bar{\kappa} \leq \kappa,$ then $\bar{\kappa}\in \mathfrak{K}.$	
\begin{proposition}\label{radius}
	We define functions $\kappa_{i}:\Omega\to (0, \infty)$, $i= 1, \ldots, 4,$ by the following formulae: for $\omega\in\Omega,$
	\begin{align*}
		[\kappa_1(\omega)]^2 &:= \|\mathfrak{Z}(\omega)(0)\|_{\H},\ \ \ &&
		[\kappa_2(\omega)]^2 := \sup_{s\leq 0} \|\mathfrak{Z}(\omega)(s)\|^2_{\H}\  e^{\nu\lambda s }, \\
		[\kappa_3(\omega)]^2 &:= \int_{- \infty}^{0} \|\mathfrak{Z}(\omega)(t)\|^2_{\H}\ e^{\nu\lambda t } \d t, \ \ \ &&
		[\kappa_4(\omega)]^2 := \int_{- \infty}^{0} \|\mathfrak{Z}(\omega)(t)\|^2_{\L^4(\mathcal{O})}\ e^{\nu\lambda t } \d t.
	\end{align*}
	Then all these functions belongs to class $\mathfrak{K}.$
	\begin{proof}Let us recall from \eqref{stationary} that $\mathfrak{Z}(\vartheta_{-t}\omega)(s) = \mathfrak{Z}(\omega)(s-t)$.
		We consider
		\begin{align*}
			\limsup_{t\to  \infty}[\kappa_1(\vartheta_{-t}\omega)]^2 e^{-\nu\lambda t } =& \limsup_{t\to \infty}\|\mathfrak{Z}(\vartheta_{-t}\omega)(0)\|^2_{\H} e^{-\nu\lambda t }	= \limsup_{t\to \infty}\|\mathfrak{Z}(\omega)(-t)\|^2_{\H} e^{-\nu\lambda t }.
		\end{align*}
		Using Lemma \ref{Bddns4}, we have, $\kappa_1 \in \mathfrak{K}.$ It can be easily seen that 
		\begin{align*}
			[\kappa_2(\vartheta_{-t}\omega)]^2 
			= & \sup_{s\leq 0}  \|\mathfrak{Z}(\omega)(s-t)\|^2_{\H}\  e^{\nu\lambda s }
			%	= & \sup_{s\leq 0}  \|\mathfrak{Z}(\omega)(s-t)\|^2_{\H}\  e^{\nu\lambda (s-t)}\ e^{\nu\lambda t}\\
			=  \sup_{\sigma\leq -t}  \|\mathfrak{Z}(\omega)(\sigma)\|^2_{\H}\  e^{\nu\lambda \sigma}\ e^{\nu\lambda t}
		\end{align*}
		and 
		\begin{align*}
			&	\limsup_{t\to \infty} [\kappa_2(\vartheta_{-t}\omega)]^2 e^{-\nu\lambda t } = \limsup_{t\to \infty} \sup_{\sigma\leq -t}  \|\mathfrak{Z}(\omega)(\sigma)\|^2_{\H}\  e^{\nu\lambda \sigma}
			= 0,
		\end{align*}
		where we have used Lemma \ref{Bddns4}. This implies that $\kappa_2\in \mathfrak{K}.$ From Lemma \ref{Bddns5} and absolute continuity of Lebesgue integrals, we obtain 
		\begin{align*}
			\bigg\{[\kappa_3(\vartheta_{-t}\omega)]^2+ [\kappa_4(\vartheta_{-t}\omega)]^2\bigg\} e^{-\nu\lambda t}
			& =\int_{- \infty}^{-t} \bigg\{  \|\mathfrak{Z}(\omega)(\sigma)\|^2_{\H} + \|\mathfrak{Z}(\omega)(\sigma)\|^2_{\L^4(\mathcal{O})}  \bigg\}e^{\nu\lambda \sigma } \d \sigma
			\\ & \to 0 \ \text{	as  }\  t\to \infty.
		\end{align*}
		This implies that $\kappa_3, \kappa_4 \in \mathfrak{K}$, which completes the proof.
	\end{proof}
\end{proposition}

 The following result helps us to obtain the existence of $\mathfrak{DK}$-absorbing set (see Theorem \ref{Main_theorem_1} below) and the uniform-tail estimates (see Lemma \ref{LR} below) satisfied by solutions of the system \eqref{csgmnse}.
	
	\begin{lemma}\label{D-abH}
		Let Hypotheses \ref{assump1} and \ref{assumpO} be satisfied, and $\f\in\H^{-1}(\mathcal{O})$. Then, for every $\omega\in\Omega$ and $\D \in\mathfrak{DK}$, and for all $\xi\in[-1,0]$, there exists $t_{\D}(\omega)\geq1$ and a constant $C_{\mathrm{ubd}}>0$ such that for all $s\leq -t_{\D}(\omega),$
		\begin{align}\label{S1-ubd}
			&	e^{\nu\lambda \xi}	 \|\v(\xi,\omega;s,\x-\mathfrak{Z}(s))\|^2_{\H}   + \frac{\nu}{2} \int_{s}^{\xi}e^{\nu\lambda \tau} \|\v(\tau,\omega;s,\x-\mathfrak{Z}(s))\|^{2}_{\V}  \d \tau
			\nonumber\\ & 
			\leq  2+2\sup_{t\leq 0}\bigg\{ \|\mathfrak{Z}(t)\|^2_{\H}\  e^{\nu\lambda t}\bigg\} 
			+C_{\mathrm{ubd}} \int_{- \infty}^{0} \bigg\{ \|\mathfrak{Z}(\tau)\|^{2}_{\H} + N^2 \|\mathfrak{Z}(\tau)\|^{2}_{\mathbb{L}^{4}(\mathcal{O})} + \|\f\|^2_{\H^{-1}(\mathcal{O})}\bigg\}e^{\nu\lambda \tau}  \d \tau,
		\end{align}
		for any $\x\in \D(\vartheta_s\omega)$.
	\end{lemma}
	\begin{proof}
		Let $\mathrm{D}$ be a random set from the class $\mathfrak{DK}$. Let $\kappa_{\mathrm{D}}(\omega)$ be the radius of $\mathrm{D}(\omega)$, that is, $\kappa_{\mathrm{D}}(\omega):= \sup\{\|x\|_{\H} : x \in \mathrm{D}(\omega)\}$ for $\omega\in \Omega$. Let $\omega\in\Omega$ be fixed. Since $\kappa_{\mathrm{D}}(\omega)\in \mathfrak{K}$, there exists $t_{\mathrm{D}}(\omega)\geq 1$ such that 
		\begin{align}\label{UTE-TRV}
			& [\kappa_{\mathrm{D}}(\vartheta_{-t}\omega)]^2 e^{-\nu\lambda t } \leq 1, 
		\end{align}
		for $t\geq t_{\mathrm{D}}(\omega).$ 	From \eqref{Energy_esti1}, for $t=\xi\in[-1,0]$ and $s\leq - t_{\D}(\omega)$, we obtain
		\begin{align}\label{UTE1}
			& e^{\nu\lambda \xi}	\|\v(\xi,\omega;s,\x-\mathfrak{Z}(s))\|^2_{\H}  + \frac{\nu}{2} \int_{s}^{0}e^{\nu\lambda \tau} \|\v(\tau,\omega;s,\x-\mathfrak{Z}(s))\|^{2}_{\V}  \d \tau
			\nonumber\\ & \leq \|\x-\mathfrak{Z}(s)\|^2_{\H}\  e^{\nu\lambda s}  
			+ C\int_{s}^{\xi} \bigg[\|\mathfrak{Z}(\tau)\|^{2}_{\H}+ N^2\|\mathfrak{Z}(\tau)\|^{2}_{\mathbb{L}^{4}(\mathcal{O})} + \|\f\|^2_{\H^{-1}(\mathcal{O})}\bigg] e^{\nu\lambda \tau} \d \tau
			\nonumber\\ & 
			\leq 2\|\x\|^2_{\H} e^{\nu\lambda  s}   +  2 \|\mathfrak{Z}(s)\|^2_{\H}e^{\nu\lambda  s}  
			+C_{\mathrm{ubd}} \int_{- \infty}^{0} \bigg\{ \|\mathfrak{Z}(\tau)\|^{2}_{\H}+ N^2\|\mathfrak{Z}(\tau)\|^{2}_{\mathbb{L}^{4}(\mathcal{O})} + \|\f\|^2_{\H^{-1}(\mathcal{O})} \bigg\}e^{\nu\lambda  \tau}  \d \tau
			\nonumber\\ &  \leq
			2+ 2 \sup_{s\leq 0}\bigg\{ \|\mathfrak{Z}(s)\|^2_{\H}\  e^{\nu\lambda  s}\bigg\} 
			+C_{\mathrm{ubd}} \int_{- \infty}^{0} \bigg\{ \|\mathfrak{Z}(\tau)\|^{2}_{\H}+ N^2\|\mathfrak{Z}(\tau)\|^{2}_{\mathbb{L}^{4}(\mathcal{O})} + \|\f\|^2_{\H^{-1}(\mathcal{O})} \bigg\}e^{\nu\lambda  \tau}  \d \tau,
		\end{align}
		where we have used \eqref{UTE-TRV}. Note that, in view of Lemmas \ref{Bddns4} and \ref{Bddns5}, the right hand side of \eqref{UTE1} is finite. This completes the proof.
	\end{proof}

	The following result is used to obtain the uniform-tail estimates for solutions of system \eqref{cgmnse-with-Pressure} (see Lemma \ref{LR} below).
	\begin{lemma}\label{D-convege}
		Let Hypotheses \ref{assump1} and \ref{assumpO} be satisfied and let us denote by $\v(t, \mathfrak{Z})\boldsymbol{x},$ the solution of system \eqref{csgmnse}. Assume that $\f\in \H^{-1}(\mathcal{O})$ and $\{\x_{n}\}_{n\in\N}$ be a bounded sequence in $\H$. Then, there exists $\tilde{\v}\in \mathrm{L}^{\infty}(0, T; \H)\cap\mathrm{L}^{2}(0, T; \V)$  such that along a subsequence
		\begin{align}\label{D-convege-4}
			\v(t, \mathfrak{Z})\boldsymbol{x}_n \ \to \tilde{\v} \ \text{	in	} \ \mathrm{L}^{2}(0,T;\L^2_{\mathrm{loc}}(\mathcal{O})),
		\end{align}
		for every $T>0$.
	\end{lemma}
	\begin{proof}
		Let $\v_n(\cdot)= \v(\cdot, \mathfrak{Z})\boldsymbol{x}_n$. Since $\{\x_{n}\}_{n\in\N}$ is a bounded sequence in $\H$, the sequence
		\begin{align}\label{Bounded}
			\{\v_n\}_{n\in \N} \ \text{ is bounded in }\  \mathrm{L}^{\infty}(0, T; \H)\cap\mathrm{L}^{2}(0, T; \V).
		\end{align}
		Hence, there exists a subsequence $\{\v_{n'}\}_{n'\in \N}$ of $\{\v_n\}_{n\in \N}$ and  $$\widetilde{\v}\in \mathrm{L}^{\infty}(0, T; \H)\cap\mathrm{L}^{2}(0, T; \V),$$ such that, as $n' \to \infty$ (by the Banach-Alaoglu theorem)
		\begin{align}\label{lim1}
			\begin{cases}
				\v_{n'} \xrightharpoonup{w^*} \widetilde{\v} \ \text{  in } \ \mathrm{L}^{\infty}(0, T; \H),\\ \v_{n'} \xrightharpoonup{w} \widetilde{\v} \ \text{  in }\ \mathrm{L}^{2}(0, T; \V).
			\end{cases}
		\end{align}		
		From Theorem \ref{solution}, we have $\big\|\frac{\d\v_{n}}{\d t}\big\|_{\mathrm{L}^{2}(0, T; \V')} \leq C,$ for some $C>0$ and all $n\in \N.$ Therefore, by the Cauchy-Schwarz inequality, for all $0\leq t \leq t+a \leq T$, $n\in \N$ and $\boldsymbol{\varphi}\in \V$, we obtain
		\begin{align}\label{335} 
			|(\v_n(t+a)-\v_n(t), \boldsymbol{\varphi})|\leq \int_{t}^{t+a}\bigg|\bigg\langle\frac{\d\v_{n}(s)}{\d t} , \boldsymbol{\varphi} \bigg\rangle\bigg|\d s\leq C(T) a^{\frac{1}{2}} \|\boldsymbol{\varphi}\|_{\V}.
		\end{align}
		Since $\v_n(t+a)-\v_n(t)\in\V$, for a.e. $t\in(0, T)$, choosing $\boldsymbol{\varphi}=\v_n(t+a)-\v_n(t)$ in \eqref{335}, we obtain 
		\begin{align}
			\|\v_n(t+a)-\v_n(t)\|^{2}_{\H}\leq C(T)a^{\frac{1}{2}}\|\v_n(t+a)-\v_n(t)\|_{\V}.
		\end{align}
		Integrating from $0$ to $T-a$, we further find 
		\begin{align}
			\int_{0}^{T-a}	\|\v_n(t+a)-\v_n(t)\|_{\H}^2\d t
			&\leq C(T)a^{\frac{1}{2}}\int_{0}^{T-a}\|\v_n(t+a)-\v_n(t)\|_{\V}\d t\nonumber\\&\leq C(T)a^{\frac{1}{2}} (T-a)^{\frac{1}{2}}\left(\int_{0}^{T-a}\|\v_n(t+a)-\v_n(t)\|_{\V}^2\d t\right)^{1/2}
			\nonumber\\&\leq \widetilde{C}(T)a^{\frac{1}{2}},
		\end{align}
		where we have used H\"older's inequality and \eqref{Bounded}. Also, $\widetilde{C}(T)$ is an another positive constant independent of $n$. Furthermore, we have 
		\begin{align}\label{3p43}
			\lim_{a\to 0}\sup_{n\in\mathbb{N}}	\int_{0}^{T-a}	\|\v_n(t+a)-\v_n(t)\|_{\H}^2\d t=0.
		\end{align}
		
		Let us now consider a truncation function $\varrho\in\C^1(\R^+)$ with $\varrho(s)=1$ for $s\in[0,1]$ and $\varrho(s)=0$ for $s\in[4,\infty)$. For each $k>0$, let us define $$\v_{n,k}(x):=\varrho\left(\frac{|x|^2}{k^2}\right)\v_n(x), \ \text{ for }x\in\O_{2k},$$
		where $\O_k= \{x\in\O: |x|< k\}$. It can easily be seen from \eqref{3p43} that 
		\begin{align}
			\lim_{a\to 0}\sup_{n\in\mathbb{N}}	\int_{0}^{T-a}	\|\v_{n,k}(t+a)-\v_{n,k}(t)\|_{\L^2(\O_{2k})}^2\d t=0,
		\end{align}
		for all $T,k>0$. Moreover, from \eqref{Bounded}, for all $T,k>0$, we infer 
		\begin{align}
			&  \{\v_{n,k}\}_{n\in\N}\ \text{ is bounded in }\ \mathrm{L}^{\infty}(0 , T;\L^2(\O_{2k}))\cap\mathrm{L}^2(0 , T;\H^1_0(\O_{2k})).
		\end{align}
		Since the injection $\H_0^1(\O_{2k})\subset\L^2{(\O_{2k})}$ is compact, we can apply  \cite[Theorem 16.3]{MMR} (see \cite[Theorem 13.3]{Te1} also) to obtain 
		\begin{align}\label{3.47}
			\{\v_{n,k}\}_{n\in\N}\ \text{ is relatively compact in } \ \mathrm{L}^{2}(0, T;\L^2(\O_{2k})),
		\end{align}
		for all $T,k>0$. From \eqref{3.47}, we further infer 
		\begin{align}\label{3.48}
			\{\v_{n}\}_{n\in\N}\ \text{ is relatively compact in } \ \mathrm{L}^{2}(0 , T;\L^2(\O_{k})),
		\end{align}
		for all $T,k>0$. 
		Using \eqref{lim1} and \eqref{3.48}, we can extract a subsequence of $\{\v_n\}_{n\in\N}$ (not relabeling) such that 
		\begin{align*}
			\v_n \to \tilde{\v} \ \text{ in } \mathrm{L}^{2} (0,  T;\L^{2}(\O_k)), 
		\end{align*}
		for all $T,k>0$, which completes the proof.
	\end{proof}

	Next, we show that the solution to the system \eqref{cgmnse-with-Pressure} satisfies the uniform-tail estimates which will help us to establish the $\mathfrak{DK}$-asymptotic compactness of $\Phi$ on unbounded domains. 
	
	Let $\Psi$ be a smooth function such that $0\leq\Psi(\xi)\leq 1$ for $\xi\in\R$ and
	\begin{align}\label{Cut-off}
		\Psi(\xi)=\begin{cases*}
			0, \text{ for } |\xi|\leq 1,\\
			1, \text{ for } |\xi|\geq2.
		\end{cases*}
	\end{align}
	Then, there exists a positive constant $C^{\ast}$ such that $|\Psi'(\xi)|\leq C^{\ast}$ and $|\Psi''(\xi)|\leq C^{\ast}$ for all $\xi\in\R$.

	\begin{lemma}\label{LR}
		Let Hypotheses \ref{assump1} and \ref{assumpO} be satisfied,  $\f\in\L^2(\O)$, $\omega\in\Omega$ and $\D \in\mathfrak{DK}$. Then, for each $\varepsilon>0$ and for each $\xi\in[-1,0]$, there exists $k_0:=k_0(\varepsilon,\xi,\omega,\D)\in\N$  such that  the solution $\v$ of system \eqref{cgmnse-with-Pressure} satisfies
		\begin{align}
			\left\|\v(\xi,\omega;s,\x-\mathfrak{Z}(s))\right\|^2_{\L^2(\O^{c}_{k})}\leq\varepsilon, 
		\end{align}
		for all $\x\in \D(\vartheta_{s}\omega)$, for all $s\leq - t_{\D}(\omega)$ and for all $k\geq k_0$,	where  $\O_{k}=\{x\in\O:\,|x|< k\}$, $\O^{c}_{k}=\O\backslash\O_{k}$
		and $t_{\D}(\omega)$ is the same time given in Lemma \ref{D-abH}.
	\end{lemma}
	
	\begin{proof} 
		Taking the divergence of the first equation of \eqref{cgmnse-with-Pressure},   we obtain
		\begin{align*}
			-\Delta \widehat{p}
			 & = F_{N}(\|\v+\mathfrak{Z}\|_{\L^4(\mathcal{O})})\nabla\cdot \big[ \big((\v+\mathfrak{Z})\cdot\nabla\big)(\v+\mathfrak{Z}) \big] - \nabla\cdot\f
			 \nonumber\\ & = F_{N}(\|\v+\mathfrak{Z}\|_{\L^4(\mathcal{O})}) \nabla\cdot\left[\nabla\cdot \left\{ \big((\v+\mathfrak{Z})\otimes(\v+\mathfrak{Z})\big)  \right\}\right] -\nabla\cdot\f,
		\end{align*}
		in the weak sense, which implies that
		\begin{align}\label{p-value}
			\widehat{p}=(-\Delta)^{-1}\left[F_N(\|\v+\mathfrak{Z}\|_{\L^4(\O)}) \sum_{i,j=1}^{3}\frac{\partial^2}{\partial x_i\partial x_j}\big((v_i+\mathfrak{Z}_i )(v_j+\mathfrak{Z}_j)\big)-\nabla\cdot\f\right].
		\end{align}
		It follows from \eqref{p-value} that
		\begin{align}
			&\|\widehat{p}\|_{\mathrm{L}^2(\mathcal{O})}\nonumber\\&=F_N(\|\v+\mathfrak{Z}\|_{\L^4(\O)})\left\|\left[\sum_{i,j=1}^{3}\frac{\partial^2}{\partial x_i\partial x_j}(-\Delta)^{-1}\big((v_i+\mathfrak{Z}_i )(v_j+\mathfrak{Z}_j)\big)\right]\right\|_{\mathrm{L}^2(\mathcal{O})} +\|(-\Delta)^{-1}(\nabla\f)\|_{\mathrm{L}^2(\mathcal{O})}\nonumber\\  &\leq CF_N(\|\v+\mathfrak{Z}\|_{\L^4(\O)})\left\|\sum_{i,j=1}^{3}(-\Delta)^{-1}\big((v_i+\mathfrak{Z}_i )(v_j+\mathfrak{Z}_j )\big)
			\right\|_{\mathrm{H}^{2}(\mathcal{O})}
			 +C\|\nabla\f\|_{\mathrm{H}^{-2}(\mathcal{O})}\nonumber\\  &\leq CF_N(\|\v+\mathfrak{Z}\|_{\L^4(\O)}) \left\|\Delta\sum_{i,j=1}^{3}(-\Delta)^{-1}\big((v_i+\mathfrak{Z}_i)(v_j+\mathfrak{Z}_j )\big)
			\right\|_{\mathrm{L}^{2}(\mathcal{O})} +C\|\f\|_{\mathbb{H}^{-1}(\mathcal{O})}
			\nonumber\\  &\leq C F_N(\|\v+\mathfrak{Z}\|_{\L^4(\O)}) \|\v+\mathfrak{Z}\|^2_{\L^4(\mathcal{O})}+C\|\f\|_{\mathbb{L}^2(\mathcal{O})}
			\nonumber\\  &\leq C N  \|\v+\mathfrak{Z}\|_{\L^4(\O)}+C\|\f\|_{\mathbb{L}^2(\mathcal{O})}
			\nonumber\\  &\leq C N  (\|\v\|_{\L^4(\O)} +\|\mathfrak{Z}\|_{\L^4(\O)} )+C\|\f\|_{\mathbb{L}^2(\mathcal{O})}
			\nonumber\\  &\leq C N  (\|\v\|_{\V} +\|\mathfrak{Z}\|_{\L^4(\O)} )+C\|\f\|_{\mathbb{L}^2(\mathcal{O})},\label{p-value-N}
		\end{align}
		where we have used the elliptic regularity for Poincar\'e domains with uniformly smooth boundary of class $\mathrm{C}^3$ (cf. Lemma 1, \cite{Heywood}), Ladyzhenskaya's inequality, \eqref{FN1} and \eqref{poin}.

		Taking the inner product of the first equation of \eqref{cgmnse-with-Pressure} with $\Psi^2\left(\frac{|x|^2}{k^{2}}\right)\v$ in $\mathbb{L}^2(\O)$, we have
		\begin{align}\label{ep1}
			& \frac{1}{2} \frac{\d}{\d t}\int_{\O} |\v|^2 \Psi^2\left(\frac{|x|^2}{k^{2}}\right)\d x 
			\nonumber\\& = \underbrace{\nu\int_{\O}(\Delta\v) \Psi^2\left(\frac{|x|^2}{k^{2}}\right) \v \d x}_{I_{1}(k,t)} -\underbrace{F_{N}(\|\v+\mathfrak{Z}\|_{\L^4(\mathcal{O})}) b\left(\v+\mathfrak{Z},\v,\Psi^2\left(\frac{|x|^2}{k^{2}}\right)\v)\right)}_{I_{2}(k,t)} 
			\nonumber\\ & \quad + \underbrace{ F_{N}(\|\v+\mathfrak{Z}\|_{\L^4(\mathcal{O})}) b\left(\v+\mathfrak{Z},\mathfrak{Z},\Psi^2\left(\frac{|x|^2}{k^{2}}\right)\v\right)}_{I_{3}(k,t)}  - \underbrace{ \int_{\O}(\nabla \widehat{p})\Psi^2\left(\frac{|x|^2}{k^{2}}\right)\v\d x}_{I_{4}(k,t)} 
			\nonumber\\ & \quad + \underbrace{\chi\int_{\O}\mathfrak{Z}\Psi^2\left(\frac{|x|^2}{k^{2}}\right)\v\d x}_{I_{5}(k,t)}
			+ \underbrace{\int_{\O}\f\Psi^2\left(\frac{|x|^2}{k^{2}}\right)\v\d x}_{I_{6}(k,t)}.
		\end{align}	
		Let us now estimate each term on the right hand side of \eqref{ep1}. Integration by parts, divergence free condition of $\v$, and H\"older's, Ladyzhenskaya's (see \eqref{lady}) and Young's inequalities help us to obtain
		\begin{align}
			I_{1}(k,t)&= -\nu \int_{\mathcal{O}}\left|\nabla\left(\Psi\left(\frac{|x|^2}{k^2}\right) \v\right)\right|^2  \d x+\nu \int_{\mathcal{O}}\v\nabla\left(\Psi\left(\frac{|x|^2}{k^2}\right)\right)\nabla\left(\Psi\left(\frac{|x|^2}{k^2}\right) \v \right) \d x \nonumber\\&\quad-\nu \int_{\mathcal{O}}\nabla\v \nabla\left(\Psi\left(\frac{|x|^2}{k^2}\right)\right)\Psi\left(\frac{|x|^2}{k^2}\right) \v  \d x\nonumber\\&\leq-\nu \int_{\mathcal{O}}\left|\nabla\left(\Psi\left(\frac{|x|^2}{k^2}\right) \v\right)\right|^2\d x+\frac{\nu}{8} \int_{\mathcal{O}}\left|\nabla\left(\Psi\left(\frac{|x|^2}{k^2}\right) \v\right)\right|^2\d x\nonumber\\&\quad+\frac{\nu\lambda}{8} \int_{\mathcal{O}}\left|\left(\Psi\left(\frac{|x|^2}{k^2}\right) \v\right)\right|^2\d x+\frac{C}{k}\left[\|\v\|^2_{\H}+\|\v\|^2_{\V}\right]\nonumber\\&\leq-\frac{3\nu\lambda}{4} \int_{\mathcal{O}}\left|\left(\Psi\left(\frac{|x|^2}{k^2}\right) \v\right)\right|^2\d x+\frac{C}{k}\|\v\|^2_{\V},\label{ep2}
			\\  
			-I_{2}(k,t)&=F_{N}(\|\v+\mathfrak{Z}\|_{\L^4(\mathcal{O})})\left|\frac{2}{k^2}\int_{\O} \Psi\left(\frac{|x|^2}{k^{2}}\right)\Psi^{\prime}\left(\frac{|x|^2}{k^{2}}\right)x\cdot(\v+\mathfrak{Z}) |\v|^2 \d x\right|\nonumber\\&= F_{N}(\|\v+\mathfrak{Z}\|_{\L^4(\mathcal{O})}) \left|\frac{2}{k^2} \int\limits_{\O\cap\{k\leq|x|\leq \sqrt{2}k\}}\Psi\left(\frac{|x|^2}{k^{2}}\right) \Psi^{\prime}\left(\frac{|x|^2}{k^{2}}\right)x\cdot(\v+\mathfrak{Z}) |\v|^2 \d x\right|
			\nonumber\\&\leq F_{N}(\|\v+\mathfrak{Z}\|_{\L^4(\mathcal{O})}) \frac{2\sqrt{2}}{k} \int\limits_{\O\cap\{k\leq|x|\leq \sqrt{2}k\}} \left|\Psi^{\prime}\left(\frac{|x|^2}{k^{2}}\right)\right| |\v+\mathfrak{Z}||\v|^2 \d x \nonumber\\& \leq \frac{C}{k} F_{N}(\|\v+\mathfrak{Z}\|_{\L^4(\mathcal{O})}) \|\v+\mathfrak{Z}\|_{\L^4(\O)} \|\v\|^2_{\L^{\frac83}(\O)} 
			\nonumber\\ & \leq \frac{CN}{k}  \|\v\|^2_{\V}
			 \leq \frac{C}{k} \bigg[ (1+N^2)\|\v\|_{\V}^2 \bigg],\label{ep3}
			\\
			I_{3}(k,t) & \leq \frac{4}{k^2}F_{N}(\|\v+\mathfrak{Z}\|_{\L^4(\mathcal{O})})\int_{\O} |\v+\mathfrak{Z}||\mathfrak{Z}|\Psi\left(\frac{|x|^2}{k^{2}}\right)\Psi^{\prime}\left(\frac{|x|^2}{k^{2}}\right)|x||\v| \d x 
			\nonumber\\ & \quad + F_{N}(\|\v+\mathfrak{Z}\|_{\L^4(\mathcal{O})})\int_{\O} |\v+\mathfrak{Z}||\nabla\v|\Psi^2\left(\frac{|x|^2}{k^{2}}\right)|\mathfrak{Z}| \d x  
			\nonumber\\  & \leq \frac{C}{k}F_{N}(\|\v+\mathfrak{Z}\|_{\L^4(\mathcal{O})}) \|\v+\mathfrak{Z}\|_{\L^4(\mathcal{O})} \|\mathfrak{Z}\|_{\L^4(\mathcal{O})}\|\v\|_{\H}  
			\nonumber\\ & \quad + F_N(\|\v+\mathfrak{Z}\|_{\L^4(\mathcal{O})})\|\v+\mathfrak{Z}\|_{\L^4(\mathcal{O})}\|\v\|_{\V}\left(\int_{\O}\Psi^8\left(\frac{|x|^2}{k^{2}}\right)|\mathfrak{Z}|^4\d x\right)^{\frac14}
			\nonumber\\  & \leq \frac{CN}{k}  \|\mathfrak{Z}\|_{\L^4(\mathcal{O})}\|\v\|_{\V}  
			+ N \|\v\|_{\V}\left(\int_{\O}\Psi^8\left(\frac{|x|^2}{k^{2}}\right)|\mathfrak{Z}|^4\d x\right)^{\frac14}
			\nonumber\\  & \leq \frac{C}{k} \left[ N^2 \|\mathfrak{Z}\|^2_{\L^4(\mathcal{O})} + \|\v\|^2_{\V} \right]  
			+ N \|\v\|_{\V}\left(\int_{\O}\Psi^8\left(\frac{|x|^2}{k^{2}}\right)|\mathfrak{Z}|^4\d x\right)^{\frac14},\label{ep4}
			\\
			-I_4(k,t)& =2\int_{\mathcal{O}}\widehat{p} \; \Psi\left(\frac{|x|^2}{k^2}\right)\Psi^{\prime}\left(\frac{|x|^2}{k^2}\right)\frac{2}{k^2}(x\cdot\v)\d x\nonumber\\&\leq\frac{C}{k} \int\limits_{\mathcal{O}\cap\{k\leq|x|\leq \sqrt{2}k\}}\left|\widehat{p}\right|\left|\v\right|\d x
			\leq \frac{C}{k}\bigg[\|\widehat{p}\|_{\mathrm{L}^2(\mathcal{O})}\|\v\|_{\H}\bigg] 
			\nonumber\\& \leq \frac{C}{k}\bigg[\|\widehat{p}\|_{\mathrm{L}^2(\mathcal{O})}\|\v\|_{\V}\bigg]
			\leq \frac{C}{k}\bigg[\|\widehat{p}\|^2_{\mathrm{L}^2(\mathcal{O})}+\|\v\|^2_{\V}\bigg]
			\nonumber\\ & \leq \frac{C}{k}\bigg[ (1+N^2) \|\v\|^2_{\V} + N^2 \|\mathfrak{Z}\|^2_{\L^4(\O)} +\|\f\|^2_{\mathbb{L}^2(\mathcal{O})}\bigg],\label{ep8}
			\\
			I_{5}(k,t)&\leq \frac{\nu\lambda}{8}\int_{\O}\Psi^2\left(\frac{|x|^2}{k^{2}}\right)|\v|^2\d x + C\int_{\O}\Psi^2\left(\frac{|x|^2}{k^{2}}\right)|\f|^2\d x, \label{ep9}
			\\
			I_{6}(k,t)&\leq \frac{\nu\lambda}{8}\int_{\O}\Psi^2\left(\frac{|x|^2}{k^{2}}\right)|\v|^2\d x + C\int_{\O}\Psi^2\left(\frac{|x|^2}{k^{2}}\right)|\mathfrak{Z}|^2\d x.\label{ep10}
		\end{align}
		Combining \eqref{ep1}-\eqref{ep10}, we arrive at
		\begin{align*}
			&	\frac{1}{2} \frac{\d}{\d t}\int_{\O}\Psi^2\left(\frac{|x|^2}{k^{2}}\right)|\v|^2\d x  +  \frac{\nu\lambda}{2} \int_{\O} \Psi^2\left(\frac{|x|^2}{k^{2}}\right)|\v|^2   \d x
			\nonumber\\ & \leq   \frac{C}{k}\bigg[ (1+N^2) \|\v\|^2_{\V} + N^2 \|\mathfrak{Z}\|^2_{\L^4(\O)} +C\|\f\|^2_{\mathbb{L}^2(\mathcal{O})}\bigg]  + N \|\v\|_{\V}\left(\int_{\O}\Psi^8\left(\frac{|x|^2}{k^{2}}\right)|\mathfrak{Z}|^4\d x\right)^{\frac14} 
			\nonumber\\ & \quad + C\int_{\O}\Psi^2\left(\frac{|x|^2}{k^{2}}\right)|\f|^2\d x+ C\int_{\O}\Psi^2\left(\frac{|x|^2}{k^{2}}\right)|\mathfrak{Z}|^2\d x,
		\end{align*}
		which implies 
		\begin{align}\label{ep11}
			&	 \frac{\d}{\d t}\int_{\O}\Psi^2\left(\frac{|x|^2}{k^{2}}\right)|\v|^2\d x + \nu\lambda   \int_{\O} \Psi^2\left(\frac{|x|^2}{k^{2}}\right)|\v|^2   \d x  
			\nonumber\\ & \leq   C\int_{\O\cap\{|x|\geq k\}}|\f|^2\d x + C\int_{\O\cap\{|x|\geq k\}}|\mathfrak{Z}|^2\d x + 2N \|\v\|_{\V} \left(\int_{\O\cap\{|x|\geq k\}}|\mathfrak{Z}|^4\d x\right)^{\frac14}
			\nonumber\\ & \quad + \frac{C}{k}\bigg[ (1+N^2) \|\v\|^2_{\V} + N^2 \|\mathfrak{Z}\|^2_{\L^4(\O)} + \|\f\|^2_{\mathbb{L}^2(\mathcal{O})}\bigg] .
		\end{align}
		An application of variation of constants formula to \eqref{ep8} gives for all $\omega\in\Omega$, for all $\xi\in[-1,0]$ and for all $s\leq -  t_{\D}(\omega)$ 
		\begin{align}\label{ep12}
			&e^{-\nu\lambda }	\int_{\O}\Psi^2\left(\frac{|x|^2}{k^{2}}\right)|\v(\xi,\omega;s,\x-\mathfrak{Z}(s))|^2\d x 
			\nonumber\\&\leq e^{\nu\lambda s} \int_{\O}\Psi^2\left(\frac{|x|^2}{k^{2}}\right)|\x(x)-\mathfrak{Z}(x,s)|^2\d x + C\int_{s}^{0} e^{\nu\lambda  \tau}\int_{\O\cap\{|x|\geq k\}}|\f(x)|^2\d x \d \tau 
			\nonumber\\ & \quad + C\int_{s}^{0} e^{\nu\lambda \tau}\int_{\O\cap\{|x|\geq k\}}|\mathfrak{Z}(x,\tau)|^2\d x \d \tau + \frac{C}{k}\int_{s}^{0} e^{\nu\lambda  \tau}\bigg[ (1+N^2)\|\v(\tau,\omega;s,\x-\mathfrak{Z}(s))\|^2_{\V} \nonumber\\& \qquad + N^2 \|\mathfrak{Z}(\tau)\|^2_{\mathbb{L}^{4}(\O)} + \|\f\|^2_{\L^2(\O)} 
			\bigg] \d \tau
			\nonumber\\ & \qquad + 2N \int_{s}^{0} e^{\nu\lambda  \tau} \|\v(\tau,\omega;s,\x-\mathfrak{Z}(s))\|_{\V} \left(\int_{\O\cap\{|x|\geq k\}}|\mathfrak{Z}(x,\tau)|^4\d x\right)^{\frac14} \d \tau
			\nonumber\\&\leq e^{\nu\lambda s}\int_{\O\cap\{|x|\geq k\}}|\x(x)-\mathfrak{Z}(x,s)|^2\d x + \frac{C(1+N^2)}{k}  + C\int_{s}^{0} e^{\nu\lambda \tau}\int_{\O\cap\{|x|\geq k\}}|\f(x)|^2\d x \d \tau 
			\nonumber\\ & \quad + C\int_{s}^{0} e^{\nu\lambda \tau}\int_{\O\cap\{|x|\geq k\}}|\mathfrak{Z}(x,\tau)|^2\d x \d \tau 
			 + C N \left(\int_{s}^{0} e^{\nu\lambda  \tau}  \left(\int_{\O\cap\{|x|\geq k\}}|\mathfrak{Z}(x,\tau)|^4\d x\right)^{\frac12} \d\tau \right)^{\frac12} 
			\nonumber\\&\to 0\text{ as } k\to\infty,
		\end{align}
		where we have used the finiteness of integral obtained in \eqref{S1-ubd} and Lemmas \ref{Bddns4}-\ref{Bddns5}. Hence, from \eqref{ep12}, we conclude that for any $\varepsilon>0$, for any  $\omega\in\Omega$ and for any $\xi\in[-1,0]$, there exists a $k_0=k_0(\varepsilon,\xi, \omega)\in\N$ such that 
		\begin{align*}%\label{ep10-N}
			&	\int_{\O\cap\{|x|\geq k\}}|\v(\xi,\omega;s,\x-\mathfrak{Z}(s))|^2 \d x \leq \varepsilon,
		\end{align*}
		for all $k\geq k_0$ and $s\leq -t_{\D}(\omega)$. This completes the proof.
	\end{proof}

	\begin{theorem}\label{Main_theorem_1}
		Suppose that  Assumptions \ref{assump1} (for RKHS) and \ref{assumpO} (for domain $\mathcal{O}$) are satisfied, and $\f\in\L^2(\O)$. Consider the MDS, $\Im = (\Omega, \mathcal{F}, \mathbb{P}, \vartheta)$ from Proposition \ref{m-DS1}, and the RDS $\Phi$ on $\H$ over $\Im$ generated by the SGMNS equations \eqref{S-GMNSE} with additive noise satisfying Assumption \ref{assump1}. Then, there exists a unique random $\mathfrak{DK}$-attractor for continuous RDS $\Phi$ in $\H$.
	\end{theorem}
	\begin{proof}
		Because of \cite[Theorem 2.8]{BCLLLR}, it is only needed to prove that there exists a $\mathfrak{DK}$-absorbing set $\textbf{B}\in \mathfrak{DK}$ and the RDS $\Phi$ is $\mathfrak{DK}$-asymptotically compact. 	
		\vskip 0.2 cm 
		\noindent
		\textbf{Existence of $\mathfrak{DK}$-absorbing set $\textbf{B}\in \mathfrak{DK}$:}	Let $\mathrm{D}$ be a random set from the class $\mathfrak{DK}$. Let $\kappa_{\mathrm{D}}(\omega)$ be the radius of $\mathrm{D}(\omega)$, that is, $\kappa_{\mathrm{D}}(\omega):= \sup\{\|x\|_{\H} : x \in \mathrm{D}(\omega)\}$ for $\omega\in \Omega$.
		
		Let $\omega\in \Omega$ be fixed. For a given $s\leq 0$ and $\boldsymbol{x}\in \H$, let $\v$ be the solution of \eqref{csgmnse} on the time interval $[s, \infty)$ with the initial condition $\v(s)= \boldsymbol{x}-\mathfrak{Z}(s).$ Using \eqref{Energy_esti1} for $t=0 \text{ and } \tau=s\leq0$, we obtain 
		\begin{align}\label{Energy_esti5}
			\|\v(0)\|^2_{\H} &\leq  2 \|\boldsymbol{x}\|^2_{\H}\  e^{\nu\lambda s}  + 2 \|\mathfrak{Z}(s)\|^2_{\H}\  e^{\nu\lambda s }  + C\int_{s}^{0} \bigg\{ \|\mathfrak{Z}(t)\|^{2}_{\H}+ N^2\|\mathfrak{Z}(t)\|^{2}_{\L^{4}(\mathcal{O})} + \|\f\|^2_{\H^{-1}(\mathcal{O})}\bigg\}e^{\nu\lambda t} \d t.
		\end{align}
		For $\omega\in \Omega,$ let us set
		\begin{align}
			[\kappa_{11}(\omega)]^2 &= 2 +  2\sup_{s\leq 0}\bigg\{ \|\mathfrak{Z}(s)\|^2_{\H}\  e^{\nu\lambda s}\bigg\} 
			+C \int_{- \infty}^{0} \bigg\{\|\mathfrak{Z}(s)\|^2_{\H} + N^2 \|\mathfrak{Z}(s)\|^2_{\L^4(\mathcal{O})}  +  \|\f\|^2_{\H^{-1}(\mathcal{O})}\bigg\}e^{\nu\lambda t} \d t,\\	
			\kappa_{12}(\omega) &=   \|\mathfrak{Z}(\omega)(0)\|_{\H}.
		\end{align}
		In view of Lemma \ref{Bddns5} and Proposition \ref{radius}, we deduce that both $\kappa_{11},\kappa_{12}\in  \mathfrak{K}$ and also that $\kappa_{11}+\kappa_{12}=:\kappa_{13} \in \mathfrak{K}$ as well. Therefore the random set $\textbf{B}$ defined by $$\textbf{B}(\omega) := \{\u\in\H: \|\u\|_{\H}\leq \kappa_{13}(\omega)\}$$ is such that $\textbf{B}\in\mathfrak{DK}.$ 
		
		Let us now show that $\textbf{B}$ absorbs $\mathrm{D}$. Let $\omega\in\Omega$ be fixed. Since $\kappa_{\mathrm{D}}(\omega)\in \mathfrak{K}$, there exists $t_{\mathrm{D}}(\omega)\geq 0$ such that 
		\begin{align*}
			[\kappa_{\mathrm{D}}(\vartheta_{-t}\omega)]^2 e^{-\nu\lambda t } &\leq 1, \  \text{ for }\  t\geq t_{\mathrm{D}}(\omega).
		\end{align*}
		Thus, for $\omega\in\Omega$, if $\boldsymbol{x}\in \mathrm{D}(\vartheta_{-t}\omega)$ and $s\leq- t_{\mathrm{D}}(\omega),$ then by \eqref{Energy_esti5}, we obtain  
		$$\|\v(0,\omega; s, \boldsymbol{x}-\mathfrak{Z}(s))\|_{\H}\leq \kappa_{11}(\omega), \  \text{ for } \ \omega\in \Omega.$$
		 Thus, we conclude that, for $\omega\in\Omega$  
		 $$\|\u(0,\omega; s, \boldsymbol{x})\|_{\H} \leq \|\v(0,\omega; s, \boldsymbol{x}-\mathfrak{Z}(s))\|_{\H} + \|\mathfrak{Z}(\omega)(0)\|_{\H}\leq \kappa_{13}(\omega).$$
		 The above inequality implies that for $\omega\in \Omega$, $\u(0,\omega; s, \boldsymbol{x}) \in \textbf{B}(\omega)$, for all $s\leq -t_{\mathrm{D}}(\omega).$ This proves  $\textbf{B}$ absorbs $\mathrm{D}$.	
		\vskip 0.2 cm 
		\noindent
		\textbf{The RDS $\Phi$ is $\mathfrak{DK}$-asymptotically compact.} 		
		In order to establish the $\mathfrak{DK}$-asymptotically compactness of $\Psi$, we use uniform-tail estimates obtained in Lemma \ref{LR}. Let us assume that $\mathrm{D} \in \mathfrak{DK}$ and ${\textbf{B}}\in \mathfrak{DK}$ be such that $ {\textbf{B}}$ absorbs $\mathrm{D}$. Let us fix $\omega\in \Omega$ and take a sequence of positive numbers $\{t_m\}^{\infty}_{m=1}$ such that $t_1\leq t_2 \leq t_3 \leq \cdots$ and $t_m \to \infty$. We take an $\H$-valued sequence $\{\boldsymbol{x}_m\}^{\infty}_{m=1}$ such that $\boldsymbol{x}_m \in \mathrm{D}(\vartheta_{-t_m}\omega),$ for all $m\in \mathbb{N}.$ Our aim is to show that the sequence $\Phi(t_m;\vartheta_{-t_m}\omega)\x_{m}$ or $\v(0,\omega;-t_m,\x_{m}-\mathfrak{Z}(-t_m))$ of the solution to the system \eqref{csgmnse} has a convergent subsequence in $\H$.

		Lemma \ref{D-abH} implies that there exists $t_{\D}(\omega) \geq 1$ such that for all $s\leq - t_{\D}(\omega)$ and $\x\in\D(\vartheta_{s}\omega)$,
		\begin{align}\label{Uac1}
			\|\v(-1,\omega;s,\x-\mathfrak{Z}(s))\|_{\H}\leq C,
		\end{align}
		where $C$ is a positive constant independent of $s$ and  $\x.$
		Since $t_m\to \infty$, there exists $M_2=M_2(\omega,\D)\in\N$ such that $t_m\geq t_{\D}(\omega),$ for all $m\geq M_2$. Since $\x_{m} \in \D(\vartheta_{-t_m}\omega)$, \eqref{Uac1} implies that for all $m\geq M_2$, {the sequence 
			\begin{align}\label{Uac2}
				\{\v(-1,\omega;-t_m,\x_{m}-\mathfrak{Z}(-t_m))\}_{m\geq M_2} \subset\H
			\end{align}
			is bounded in $\H$.}	We infer from \eqref{Uac2} and Lemma \ref{D-convege} that there exists $\xi\in(-1,0)$, $\hat{\v}\in\H$ and a subsequence (not relabeling) such that for every $k\in\N$ 
		\begin{align}\label{Uac3}
			\v(\xi,\omega;-t_m,\x_{m} - \mathfrak{Z}(-t_m))=\v(\xi,\omega;-1,\v(-1,\omega;-t_m,\x_{m}-\mathfrak{Z}(-t_m))-\mathfrak{Z}(-1))\to \hat{\v}   \text{ in }   \L^2(\O_k).
		\end{align}
		as $m\to\infty$. Therefore, we infer from {the proof of } Lemma \ref{RDS_Conti1} that there exists a positive constant $C_{\mathrm{Lip}}$ such that
		\begin{align}\label{Uac7}
			&\|\v(0,\omega;\xi,\v(\xi,\omega;-t_m,\x_{m}-\mathfrak{Z}(-t_m))-\mathfrak{Z}(\xi))-\v(0,\omega;\xi,\hat{\v}-\mathfrak{Z}(\xi))\|^2_{\H}
			\nonumber\\ & \leq C_{\mathrm{Lip}}\|\v(\xi,\omega;-t_m,\x_{m}-\mathfrak{Z}(-t_m))-\hat{\v}\|^2_{\H}.
		\end{align}
		Let us now choose $\eta>0$ and fix it. Since $\hat{\v}\in\H$, there exists $K_1=K_1(\eta,\omega,\D)>0$ such that for all $k\geq K_1$, 
		\begin{align}\label{Uac4}
			\int_{\O\cap\{|x|\geq k\}}|\hat{\v}|^2\d x<\frac{\eta^2}{6C_{\mathrm{Lip}}},
		\end{align}
		where $C_{\mathrm{Lip}}>0$ is a constant defined in \eqref{Uac7}. Also, we know from Lemma \ref{LR} that there exist $M_3=M_3(\xi,\eta,\omega,\D)\in\N$ and $K_2=K_2(\xi,\eta,\omega,\D)\geq K_1$ such that for all $m\geq M_3$ and $k\geq K_2$,
		\begin{align}\label{Uac5}
			\int_{\O\cap\{|x|\geq k\}}|\v(\xi,\omega;-t_m,\x_{m}-\mathfrak{Z}(-t_m))|^2\d x<\frac{\eta^2}{6C_{\mathrm{Lip}}}.
		\end{align}
		From \eqref{Uac3}, we have that there exists $M_4=M_4(\xi,\eta,\omega,\D)>M_3$ such that for all $m\geq M_4$,
		\begin{align}\label{Uac6}
			\int_{\O\cap\{|x|< K_2\}}|\v(\xi,\omega;-t_m,\x_{m}-\mathfrak{Z}(-t_m))-\hat{\v}|^2\d x<\frac{\eta^2}{3C_{\mathrm{Lip}}}.
		\end{align}
		Finally, we infer from \eqref{Uac7} that
		\begin{align}\label{Uac8}
			&\|\v(0,\omega;\xi,\v(\xi,\omega;-t_m,\x_{m}-\mathfrak{Z}(-t_m))-\mathfrak{Z}(\xi))-\v(0,\omega;\xi,\hat{\v}-\mathfrak{Z}(\xi))\|^2_{\H}\nonumber\\&\leq  C_{\mathrm{Lip}}\bigg[\int_{\O\cap\{|x|<K_2\}}|\v(\xi,\omega;-t_m,\x_{m}-\mathfrak{Z}(-t_m))-\hat{\v}|^2\d x 
			\nonumber\\ & \qquad + \int_{\O\cap\{|x|\geq K_2\}} |\v(\xi,\omega;-t_m,\x_{m}-\mathfrak{Z}(-t_m))-\hat{\v}|^2\d x\bigg]\nonumber\\&\leq C_{\mathrm{Lip}}\bigg[\int_{\O\cap\{|x|<K_2\}}|\v(\xi,\omega;-t_m,\x_{m}-\mathfrak{Z}(-t_m))-\hat{\v}|^2\d x \nonumber\\ & \qquad +2\int_{\O\cap\{|x|\geq K_2\}}(|\v(\xi,\omega;-t_m,\x_{m}-\mathfrak{Z}(-t_m))|^2+|\hat{\v}|^2)\d x\bigg]
			\nonumber\\ & < \eta^2,
		\end{align}
		for every $m\geq M_3$, where we have used \eqref{Uac4}-\eqref{Uac6}. Since $\eta>0$ is arbitrary, we conclude the proof.
	\end{proof}

	\section{Invariant Measures}\label{sec7}\setcounter{equation}{0}
In this section, we first show the existence of invariant measures for SGMNS equations in $\H$. It is established in \cite{CF} that the existence of a compact invariant random set is a sufficient condition for the existence of invariant measures, that is, if a RDS $\Phi$ has compact invariant random set, then there exists an invariant measure for $\Phi$ (\cite[Corollary 4.4]{CF}). Since, the random attractor itself is a compact invariant random set, the existence of invariant measures for the SGMNS equations \eqref{S-GMNSE} is a direct consequence of \cite[Corollary 4.4]{CF} and Theorem \ref{Main_theorem_1}. %The existence of random attractors for 2D stochastic NSE in unbounded Poincar\'e domains has been established in \cite{BMO,BL}, etc.  Recently, the existence and uniqueness  of invariant measures for 2D stochastic NSE perturbed by a linear multiplicative Gaussian noise defined on the whole space has  been obtained in \cite{KM8}.  
Furthermore, we also prove the uniqueness of invariant measures for the system \eqref{SGMNSE} under suitable condition on the coefficients (see Theorem \ref{UIM1} below).

\subsection{Existence of invariant measures}
Let us define the transition operator $\{\mathrm{T}_t\}_{t\geq 0}$ by 
\begin{align}\label{71}
	\mathrm{T}_t f(\x)=\int_{\Omega}f(\Phi(\omega,t,\x))\d\mathbb{P}(\omega)=\E\left[f(\Phi(t,\x))\right], 
\end{align}
  for all $f\in\mathcal{B}_b(\H)$, where $\mathcal{B}_b(\H)$ is the space of all bounded and Borel measurable functions on $\H$ and $\Phi$ is the RDS corresponding to the system \eqref{S-GMNSE}, which is defined by \eqref{combine_sol}. The continuity of $\Phi$ (cf. Lemma \ref{RDS_Conti1}), \cite[Proposition 3.8]{BL} provides the following result: 
\begin{lemma}\label{Feller}
	The family $\{\mathrm{T}_t\}_{t\geq 0}$ is Feller, that is, $\mathrm{T}_tf\in\C_{b}(\H)$ if $f\in\C_b(\H)$, where $\C_b(\H)$ is the space of all bounded and continuous functions on $\H$. Furthermore, for any $f\in\C_b(\H)$, $\mathrm{T}_tf(\x)\to f(\x)$ as $t\downarrow 0$. 
\end{lemma}
Analogously as in the proof of \cite[Theorem 5.6]{CF}, one can prove that $\Phi$ is a Markov RDS, that is, $\mathrm{T}_{t_1+t_2}=\mathrm{T}_{t_1}\mathrm{T}_{t_2}$, for all $t_1,t_2\geq 0$. Since, we know by \cite[Corollary 4.4]{CF} that if a Markov RDS on a Polish space has an invariant compact random set, then there exists a Feller invariant probability measure $\upmu$ for $\Phi$. 
\begin{definition}
	A Borel probability measure $\upmu$ on $\H$  is called an \emph{invariant measure} for a Markov semigroup $\{\mathrm{T}_t\}_{t\geq 0}$ of Feller operators on $\C_b(\H)$ if and only if $$\mathrm{T}_{t}^*\upmu=\upmu, \ t\geq 0,$$ where $(\mathrm{T}_{t}^*\upmu)(\Gamma)=\int_{\H}\mathrm{T}_{t}(\u,\Gamma)\upmu(\d\u),$ for $\Gamma\in\mathcal{B}(\H)$ and  $\mathrm{T}_t(\u,\cdot)$ is the transition probability, $\mathrm{T}_{t}(\u,\Gamma)=\mathrm{T}_{t}(\chi_{\Gamma})(\u),\ \u\in\H$.
\end{definition}

By the definition of random attractors, it is clear  that there exists an invariant compact random set in $\H$. A Feller invariant probability measure for a Markov RDS $\Phi$ on $\H$ is, by definition, an invariant probability measure for the semigroup $\{\mathrm{T}_t\}_{t\geq 0}$ defined by \eqref{71}. Hence, we have the following result on the existence of invariant measures for the system \eqref{S-GMNSE} in $\H$.
\begin{theorem}\label{thm6.3}
Suppose that  Assumptions \ref{assump1} (for RKHS) and \ref{assumpO} (for domain $\mathcal{O}$) are satisfied, and $\f\in\L^2(\O)$. Then, there exists an invariant measure for the system \eqref{S-GMNSE} in $\H$.
\end{theorem}

\subsection{Uniqueness of invariant measures}
For sufficiently large $\nu$, by using the exponential stability of solutions, we show that the invariant measure is unique. 
%In this work, $\W(\cdot)$ is a Wiener process with RKHS $\mathrm{K}$ satisfying Assumption \ref{assump1}. In particular, $\mathrm{K} \subset\H$ and the natural embedding  $i : \mathrm{K}\hookrightarrow \H$ is a Hilbert-Schmidt operator. For a fixed orthonormal basis $\{w_k\}_{k\in\N}$ of $\mathrm{K}$ and a sequence $\{\beta_k\}_{k\in\N}$ of independent Brownian motions defined on some filtered probability space $(\Omega, \mathscr{F}, (\mathscr{F}_t)_{t\in \R}, \mathbb{P})$ such that $\W(\cdot)$ can be written in the following form
%\begin{align}\label{Sum-W}
%	\W(t)=\sum_{k=1}^{\infty}\beta_k(t) w_k,  \ \ \ t\in\mathbb{R}.
%\end{align}
%Moreover, there exists a covariance operator $\J \in \mathfrak{L}(\H)$ associated with $\W(\cdot)$ defined by 
%\begin{align*}
%	\left\langle \J h_1,h_2\right\rangle=\mathbb{E}\left[\left\langle h_1,\W(1)\right\rangle_{\H}\left\langle \W(1),h_2\right\rangle_{\H}\right], \ \ \ h_1,h_2\in \H. 
%\end{align*}
%It is well known from \cite{DZ1} that $\J$ is a non-negative self-adjoint and trace class operator in $\H$. Furthermore, $\J = ii^*$ and $\mathrm{K} = R(\J^{\frac{1}{2}} ),$ where $R(\J^{\frac{1}{2}} )$ is the range of the operator $\J^{\frac{1}{2}}$ (see \cite{BN1}). Note that 
%\begin{align*}
%	\sum_{k=1}^{\infty}\|iw_k\|^2_{\H}= \text{Tr}\left[\J\right]<\infty.
%\end{align*}
\subsubsection{Exponential stability}
Here, we obtain the exponential stability of solutions, which is used to obtain  the uniqueness of invariant measures.
\begin{theorem}\label{UIM1}
	Let $\u_1(\cdot)$ and $\u_2(\cdot)$ be two solutions of the system \eqref{S-GMNSE} with the initial data $\u_1^0,\u_2^0\in\H$, respectively. Then, we have
	\begin{align}\label{62}
		\mathbb{E}\left[\|\u_1(t)-\u_2(t)\|^2_{\H}\right]\leq \|\u_1^0-\u_2^0\|^2_{\H}\ \exp\left\{-\left[\nu\lambda   - \frac{ 7^{7}\cdot N^8}{2^{12}\cdot \nu^{7}}\right]t\right\},
	\end{align}
	for all $t\geq0$, provided $\nu   > \frac{7 N}{2} \left(\frac{1}{112 \lambda}\right)^{\frac18}$.
\end{theorem}
\begin{proof}
	Let $\mathfrak{Z}(\cdot)=\u_1(\cdot)-\u_2(\cdot)$, then $\mathfrak{Z}(\cdot)$ satisfies  the following equality:
	\begin{align}\label{UN1}
		\|\mathfrak{Z}(t)\|^2_{\H}&=\|\mathfrak{Z}(0)\|^2_{\H}-2\nu\int_{0}^{t}\|\mathfrak{Z}(\zeta)\|^2_{\V}\d\zeta  -2\int_{0}^{t}\left \langle\B_N(\u_1(\zeta))-\B_N(\u_2(\zeta)),\mathfrak{Z}(\zeta)\right\rangle\d\zeta,
	\end{align}
	for all $t\geq0$. An argument similar to \eqref{CL1} gives
	\begin{align}\label{UN2}
		\|\mathfrak{Z}(t)\|^2_{\H}&\leq \|\mathfrak{Z}(0)\|^2_{\H}- \left[\nu\lambda   - \frac{ 7^{7}\cdot N^8}{2^{12}\cdot \nu^{7}}\right] \int_{0}^{t}\|\mathfrak{Z}(\zeta)\|^2_{\H}\d\zeta,
	\end{align}
for all $t\geq0$. Now, taking expectation and applying Gronwall's inequality, one can  conclude the proof.
%	\begin{align}\label{63}
%		\mathbb{E}[\|\mathfrak{Z}(t)\|^2_{\H}]\leq\begin{cases}
%			\|\mathfrak{Z}(0)\|^2_{\H}\ \text{exp}\left\{-\left[\nu\lambda   - \frac{ 7^{7}\cdot N^8}{2^{12}\cdot \nu^{7}}\right]t\right\},&\text{	for  } r>3,\\
%			\|\mathfrak{Z}(0)\|^2_{\H}\ \text{exp}[-(\nu\lambda+\nu\lambda)t],&\text{	for  } r=3 \text{	with  }2\beta\nu\geq1,
%		\end{cases}
%	\end{align}
%which completes the proof.
\end{proof}
\begin{theorem}\label{UIM2}
	Let the condition $\nu   > \frac{7 N}{2} \left(\frac{1}{112 \lambda}\right)^{\frac18}$ given in Theorem \ref{UIM1} be satisfied, and $\f\in\L^2(\O)$. Then, there is a unique invariant measure for the system \eqref{S-GMNSE}. Moreover, the invariant measure is ergodic and strongly mixing.
\end{theorem}
	\begin{proof}
			For $\varphi\in \text{Lip}(\H)$ (Lipschitz $\varphi$) and an invariant measure $\upmu$, we have for all $t\geq 0$,
			\begin{align*}
				\left|\mathrm{T}_t\varphi(\u_1^0)-\int_{\H}\varphi(\u_2^0)\upmu(\d \u_2^0)\right|
				  &=	\left|\E[\varphi(\u(t,\u_1^0))]-\int_{\H}\mathrm{T}_t\varphi(\u_2^0)\upmu(\d \u_2^0)\right|\nonumber\\&=\left|\int_{\H}\E\left[\varphi(\u(t,\u_1^0))-\varphi(\u(t,\u_2^0))\right]\upmu(\d \u_2^0)\right|\nonumber\\&\leq L_{\varphi}\int_{\H}\E\left[\left\|\u(t,\u_1^0)-\u(t,\u_2^0)\right\|_{\H}\right]\upmu(\d \u_2^0)\nonumber\\&\leq L_{\varphi}\mathbb{E}\left[e^{-\left(\nu\lambda   - \frac{ 7^{7}\cdot N^8}{2^{12}\cdot \nu^{7}}\right)t}\right]\int_{\H}\left\|\u_1^0-\u_2^0\right\|_{\H}\upmu(\d \u_2^0)
			\nonumber\\ & 	\to 0 \ \text{ as }\  t\to \infty,
			\end{align*}
			since $\int_{\H}\|\u_2^0\|_{\H}\upmu(\d \u_2^0)+\int_{\H}\|\u_1^0\|_{\H}\upmu(\d \u_2^0)<+\infty$. Hence, we conclude
			\begin{align}\label{U2}
				\lim_{t\to\infty}\mathrm{T}_t\varphi(\u_1^0)=\int_{\H}\varphi(\u_2^0)\d\upmu(\u_2^0), \ \upmu\text{-a.s., for all }\ \u_1^0\in\H\ \text{ and }\  \varphi\in\C_b(\H),
			\end{align} 
			by the density of $\text{Lip}(\H)$ in $\C_b (\H)$. Since we have a stronger result that $\mathrm{T}_t\varphi(\u_1^0)$ converges exponentially fast to the equilibrium, this property is known as the \emph{exponential mixing property}. Now suppose that  $\widetilde{\upmu}$ is an another invariant measure, then we have for all $t\geq 0$,
			\begin{align}
			&	\left|\int_{\H}\varphi(\u_1^0)\upmu(\d\u_1^0)-\int_{\H}\varphi(\u_2^0)\widetilde{\upmu}(\d \u_2^0)\right| \nonumber\\&= \left|\int_{\H}\mathrm{T}_t\varphi(\u_1^0)\upmu(\d \u_1^0)-\int_{\H}\mathrm{T}_t\varphi(\u_2^0)\widetilde{\upmu}(\d \u_2^0)\right|\nonumber\\&=\left|\int_{\H}\int_{\H}\left[\mathrm{T}_t\varphi(\u_1^0)-\mathrm{T}_t\varphi(\u_2^0)\right]\upmu(\d \u_1^0)\wi\upmu(\d \u_2^0)\right|\nonumber\\&=\left|\int_{\H}\int_{\H}\E\left[\varphi(\u(t,\u_1^0))-\varphi(\u(t,\u_2^0))\right]\upmu(\d \u_1^0)\wi\upmu(\d \u_2^0)\right|\nonumber\\&\leq L_{\varphi}\int_{\H}\int_{\H}\E\left[\left\|\u(t,\u_1^0)-\u(t,\u_2^0)\right\|_{\H}\right]\upmu(\d \u_1^0)\wi\upmu(\d \u_2^0) \nonumber\\&\leq L_{\varphi}e^{-\left(\nu\lambda   - \frac{ 7^{7}\cdot N^8}{2^{12}\cdot \nu^{7}}\right)t}
				 \int_{\H}\int_{\H}\E\left[\left\|\u_1^0-\u_2^0\right\|_{\H}\right]\upmu(\d \u_1^0)\wi\upmu(\d \u_2^0) \nonumber\\&\leq L_{\varphi}e^{-\left(\nu\lambda   - \frac{ 7^{7}\cdot N^8}{2^{12}\cdot \nu^{7}}\right)t}
				 \bigg(\int_{\H}\|\u_1^0\|_{\H}\upmu(\d \u_1^0)+\int_{\H}\|\u_2^0\|_{\H}\widetilde{\upmu}(\d \u_2^0)\bigg)
				\nonumber\\ & \to 0 \ \text{ as }\ t\to\infty,
			\end{align}
			since $\int_{\H}\|\u_1^0\|_{\H}\upmu(\d \u_1^0)<+\infty$ and $\int_{\H}\|\u_2^0\|_{\H}\widetilde{\upmu}(\d \u_2^0)<+\infty$.  Since $\upmu$ is the unique invariant measure for $\{\mathrm{T}_t\}_{t\geq 0}$, it follows from \cite[Theorem 3.2.6]{DZ} that $\upmu$ is ergodic also. 
	\end{proof}

\begin{appendix}
	\renewcommand{\thesection}{\Alph{section}}
	\numberwithin{equation}{section}
	
	\section{Convergence to weak solutions of 3D Navier-Stokes equations} \label{ApA}\setcounter{equation}{0}
In this section, we show that a subsequence of weak the solution to 3D GMNS equations converges to a weak solution of 3D NSE as $N\to\infty$. Let us consider the following projected NSE and GMNS equations (we denote the solution of NSE and GMNS equations by $\u^*$ and $\u^N$, respectively, in this section):
\begin{equation}\label{D-NSE}
	\left\{
	\begin{aligned}
		\frac{\d\u^*(t)}{\d t} + \nu \A\u^*(t)+\B(\u^*(t))  &=\mathcal{P}\f , \ \ \ t> 0, \\ 
		\u^*(0)&=\boldsymbol{x}^*,
	\end{aligned}
	\right.
\end{equation}
and %(we denote the solution of GMNS equations by $\u^{N}$)
\begin{equation}\label{D-GMNSE}
	\left\{
	\begin{aligned}
		\frac{\d\u^{N}(t)}{\d t} + \nu \A\u^{N}(t)+F_{N}(\|\u^{N}\|_{\L^4(\mathcal{O})}) \cdot\B(\u^{N}(t))  &=\mathcal{P}\f , \ \ \ t> 0, \\ 
		\u^{N}(0)&=\boldsymbol{x}^{N},
	\end{aligned}
	\right.
\end{equation}
respectively.

Next we provide the definition of Leray-Hopf weak solution for the system \eqref{D-NSE}. For the system \eqref{D-GMNSE}, we refer readers to Definition \ref{defn5.9} with $\mathfrak{Z}\equiv\boldsymbol{0}$.
\begin{definition}\label{defnNSE}
	For $T>0$, assume that $\x^* \in \H$ and $\f\in \H^{-1}(\mathcal{O})$. A function $\u^*(\cdot)$ is called a \emph{Leray-Hopf weak solution} of the system \eqref{D-NSE} on the time interval $[0, T]$, if 
	\begin{align}\label{A2}
	\u^*\in  \mathrm{L}^{\infty}(0,T; \H) \cap \mathrm{L}^{2}(0,T; \V), \;\;\; \frac{\d\u^*}{\d t}\in\mathrm{L}^{\frac43}(0,T;\V'),
		\end{align}
	and it satisfies 
	\begin{itemize}
		\item [(i)] for any $\psi\in \V,$ 
		\begin{align}\label{W-NSE}
			\left<\frac{\d\u^*(t)}{\d t}, \psi\right>&=  - \left\langle \nu \A\u^*(t)+\B(\u^*(t)), \psi \right\rangle  +\left\langle \f , \psi \right\rangle,
		\end{align}
		for a.e. $t\in[0,T];$
		\item [(ii)] the initial data:
		$$\u^*(0)=\x^* \ \text{ in }\ \H,$$
		\item [(iii)] the following energy inequality is satisfied: 
		\begin{align}\label{A5}
			\|\u^*(t)\|_{\H}^2+2\nu\int_0^t\|\u^*(s)\|_{\V}^2\d s\leq \|\x^*\|_{\H}^2+2\int_0^t\langle\f,\u^*(s)\rangle\d s,
		\end{align}
		for all $t\in[0,T]$. 
	\end{itemize}
\end{definition}
\begin{remark}
	The regularity given in \eqref{A2} implies that $\u^*\in\C_w([0,T];\H)$, where $\C_w([0,T];\H)$ is the space of functions $\u^*:[0,T]\to\H$ which are weakly continuous; that is, such that for all $\varphi\in\H$, the scalar function $[0,T]:t\mapsto(\u^*(t),\varphi)\in\mathbb{R}$ is continuous on $[0,T].$ The existence of a Leray-Hopf weak solution to the system  \eqref{D-NSE} can be obtained from \cite[Theorem 3.1]{Temam}. 
\end{remark}
Let us now state the main result of this section. The proof of following theorem will be given after an auxiliary lemma. 
\begin{theorem}\label{Con-N}
	Suppose that $\f\in\H^{-1}(\mathcal{O})$. Also, let $\u^{N}(\cdot)$ be a weak solution of the GMNS equations \eqref{D-GMNSE} with the initial value $\x^{N}\in\H$ and  $\x^{N}\xrightharpoonup{} \x^*$ weakly in $\H$ as $N\to\infty$.  Then, there exists a sequence $\{\u^{N_j}(\cdot)\}_{j=1}^{\infty}$ and an element $\u^*(\cdot)\in \mathrm{L}^{\infty}(0,T;\H)\cap \mathrm{L}^{2}(0,T;\V)$ such that  $\{\u^{N_j}(\cdot)\}_{j=1}^{\infty}$ converges to $\u^*$ weak-star in $\mathrm{L}^{\infty}(0,T;\H)$, weakly in $\mathrm{L}^{2}(0,T;\V)$ and strongly to $\mathrm{L}^{2}(0,T;\L^2_{\mathrm{loc}}(\mathcal{O}))$ as $N_j\to\infty$. In addition, $\u^*(\cdot)$ is a  weak solution of the 3D NSE \eqref{D-NSE} with initial data $\x^*$ (in the sense of Definition \ref{defnNSE} satisfying (i) and (ii)) on the interval $[0,T]$, for every $T>0$. %(Modify the statement!) 
	Furthermore, if $\x^N\to\x$ strongly in $\H$, then $u^*$ satisfies the energy equality \eqref{A5}. 
\end{theorem}

The following lemma is a key element to prove the main result of this section.

\begin{lemma}\label{FN-Lemma}
	Let $\u^N$ be a solution of the 3D GMNS equations \eqref{D-GMNSE} with the initial value $\x^{N}\in\H$ and  $\x^{N}\xrightharpoonup{} \x^*$ weakly in $\H$ as $N\to\infty$. Then, for each $p\geq1$,
	\begin{align}
		F_N\left(\|\u^{N}(s)\|_{\L^4(\mathcal{O})}\right) \to 1 \;\;\; \text{  in }\;\;\; \mathrm{L}^{p}(0,T;\R),
	\end{align}
as $N\to\infty$.
\end{lemma}
\begin{proof}
		Since $\x^{N}\xrightharpoonup{} \x^*$ weakly in $\H$ as $N\to\infty$, it implies that $\x^N$ is a bounded sequence in $\H$, that is, there exists a constant $M>0$ such that $\|\x^N\|_{\H}\leq M$, for every $N>0$.
	
	A calculation similar to \eqref{S4} (for $t=T$ and $\mathfrak{Z}\equiv\boldsymbol{0}$) gives 
	\begin{align}\label{UE1-N}
		\sup_{s\in [0, T]}\|\u^N(s)\|_{\H}^2 + \int_{0}^{T}\|\u^N(s)\|_{\V}^2\d s &\leq  	\|\x^{N}\|^2_{\H}   + CT \|\f\|^2_{\H^{-1}(\mathcal{O})} \nonumber \\
	& 	\leq  	M^2   + CT \|\f\|^2_{\H^{-1}(\mathcal{O})} =: K_{T}.
	\end{align}
Let us denote 
\begin{align}
	I_{N}=\{s\in(0,T):\|\u^N(s)\|_{\L^4(\mathcal{O})}\geq N\},
\end{align}
and $|I_N|$ the Lebesgue measure on $I_N$.  Then
\begin{align*}
	|I_N|N^{\frac83} & \leq \int_{0}^{T} \|\u^N(s)\|^{\frac83}_{\L^4(\mathcal{O})}\d s 
	  \leq \int_{0}^{T}\|\u^N(s)\|^{\frac23}_{\H} \|\u^N(s)\|^{2}_{\V}\d s
	\nonumber\\ & \leq \sup_{s\in [0, T]} \|\u^N(s)\|^{\frac23}_{\H} \int_{0}^{T} \|\u^N(s)\|^{2}_{\V}\d s 
	  \leq (K_T)^{\frac43},
\end{align*}
and so
\begin{align*}
	|I_N| \leq \left[\frac{K_T}{N^{2}}\right]^{\frac{4}{3}} \to 0 \;\;\; \text{	as  }\;\;\; N\to\infty. 
\end{align*}
Now, as
\begin{align*}
T - |I_N| = \int_{[0,T]\backslash I_N}F_N\left(\|\u^N(s)\|_{\L^4(\mathcal{O})}\right) \d s \leq \int_{0}^{T} F_N\left(\|\u^N(s)\|_{\L^4(\mathcal{O})}\right) \d s \leq T,
\end{align*}
we infer that 
\begin{align*}
\int_{0}^{T} F_N\left(\|\u^N(s)\|_{\L^4(\mathcal{O})}\right) \d s \to \int_{0}^{T} 1 \d s \;\;\; \text{	as  }\;\;\; N\to\infty. 
\end{align*}
But since $0\leq F_N\left(\|\u^N(s)\|_{\L^4(\mathcal{O})}\right)\leq 1$, so
\begin{align*}
	\int_{0}^{T} \left|1-F_N\left(\|\u^N(s)\|_{\L^4(\mathcal{O})}\right)\right|\d s =  	\int_{0}^{T} \left(1-F_N\left(\|\u^N(s)\|_{\L^4(\mathcal{O})}\right)\right)\d s  \to 0 \;\;\; \text{	as  }\;\;\; N\to\infty. 
\end{align*}
Finally, using the fact that $\left|1-F_N\left(\|\u^N(s)\|_{\L^4(\mathcal{O})}\right)\right|\leq 1$, we deduce 
\begin{align*}
	\int_{0}^{T} \left|1-F_N\left(\|\u^N(s)\|_{\L^4(\mathcal{O})}\right)\right|^p\d s & =  	\int_{0}^{T} \left|1-F_N\left(\|\u^N(s)\|_{\L^4(\mathcal{O})}\right)\right|^{p-1}\left|1-F_N\left(\|\u^N(s)\|_{\L^4(\mathcal{O})}\right)\right|\d s 
	\nonumber\\ & \leq  	\int_{0}^{T} \left|1-F_N\left(\|\u^N(s)\|_{\L^4(\mathcal{O})}\right)\right|\d s 
	\nonumber\\ & \to 0 \;\;\; \text{	as  }\;\;\; N\to\infty,
\end{align*}
for any $p>1$.
\end{proof}

Using the result from Lemma \ref{FN-Lemma}, we are now ready to provide a  proof of Theorem \ref{Con-N}. 

\begin{proof}[Proof of Theorem \ref{Con-N}]
	From \eqref{UE1-N}, we infer that the set $\{\u^N\}_{N>0}$ is a bounded in $\mathrm{L}^{\infty}(0,T;\H)\cap\mathrm{L}^{2}(0,T;\V).$ 
	%In view of Banach-Alaoglu theorem,  there exists a sequence $\{\u^{N_j}(\cdot)\}_{j=1}^{\infty}$ and an element $\u^*(\cdot)\in \mathrm{L}^{\infty}(0,T;\H)\cap \mathrm{L}^{2}(0,T;\V)$ such that  $\{\u^{N_j}(\cdot)\}_{j=1}^{\infty}$ converges to $\u^*$ weak-star in $\mathrm{L}^{\infty}(0,T;\H)$, weakly in $\mathrm{L}^{2}(0,T;\V)$ and strongly to $\mathrm{L}^{2}(0,T;\L^2_{\mathrm{loc}}(\mathcal{O}))$ as $N_j\to\infty$.
	
		For any arbitrary element $\boldsymbol{\psi}\in\V$, using H\"older's inequality and Sobolev's embedding, we have from \eqref{D-GMNSE} that
	\begin{align*}
		\left|\left\langle\frac{\d\u^N}{\d t},\boldsymbol{\psi}\right\rangle\right|
		 &\leq \nu\left|(\nabla\u^N,\nabla\boldsymbol{\psi})\right|+\left|F_{N}(\|\u^N\|_{\L^4(\mathcal{O})})   \cdot b(\u^N,\boldsymbol{\psi},\u^N)\right| +\left|\langle\f,\boldsymbol{\psi}\rangle\right| 
		\nonumber\\&\leq C\|\u^N\|_{\V}\|\boldsymbol{\psi}\|_{\V}
		+  \|\u^N\|_{\L^{4}(\mathcal{O})}^{2}\|\boldsymbol{\psi}\|_{\V} +\|\f\|_{\H^{-1}(\mathcal{O})}\|\boldsymbol{\psi}\|_{\V}
		\nonumber\\&\leq C\left[\|\u^N\|_{\V}
		+  \|\u^N\|_{\H}^{\frac12}\|\u^N\|_{\V}^{\frac32} +\|\f\|_{\H^{-1}(\mathcal{O})}\right]\|\boldsymbol{\psi}\|_{\V},
	\end{align*}
which implies 
	\begin{align*}
		\int_{0}^{T}\left\|\frac{\d\u^N(s)}{\d t}\right\|^{\frac{4}{3}}_{\V^{\prime}}\d s
		&\leq C\int_{0}^{T}\|\u^N(s)\|_{\V}^{\frac{4}{3}}\d s
		+  \sup_{s\in [0, T]}\|\u^N(s)\|_{\H}^{\frac23}\int_{0}^{T}\|\u^N(s)\|_{\V}^{2}\d s + T \|\f\|_{\H^{-1}(\mathcal{O})}^{\frac{4}{3}}\nonumber\\
		&\leq CT^{\frac13}K_T
		+ (K_T)^{\frac43} + T \|\f\|_{\H^{-1}(\mathcal{O})}^{\frac{4}{3}}<+\infty.
	\end{align*}
	In view of the Banach-Alaoglu theorem,  there exists a sequence $\{\u^{N_j}(\cdot)\}_{j=1}^{\infty}$ and an element $\u^*(\cdot)\in \mathrm{L}^{\infty}(0,T;\H)\cap \mathrm{L}^{2}(0,T;\V)$ with $\frac{\d\u^*(\cdot)}{\d t}\in\mathrm{L}^{2}(0,T;\V^{\prime})$ such that  
	\begin{align}
		\u^{N_j}\xrightharpoonup{w^*}&\ \u^* &&\text{ in }\ \ \ \ \	\mathrm{L}^{\infty}(0,T;\H),\label{S7-N}\\
		\u^{N_j}\xrightharpoonup{w}&\ \u^*   && \text{ in } \ \ \ \ \ \mathrm{L}^{2}(0,T;\V),\label{S8-N}\\
		\u^{N_j}\xrightharpoonup{w}&\ \u^*   && \text{ in } \ \ \ \ \ \mathrm{L}^{2}(0,T;\L^2_{\mathrm{loc}}(\mathcal{O})),\label{S9-N}\\
		\frac{\d \u^{N_j}}{\d t}\xrightharpoonup{w}&\frac{\d \u^*}{\d t}   && \text{ in }  \ \ \ \ \ \mathrm{L}^{\frac43}(0,T;\V'),\label{S8d-N}
	\end{align}
	 as $N_j\to\infty$.

	From \cite[Ch. III, Lemma 3.2]{Temam} along with convergences in \eqref{S7-N}-\eqref{S9-N}, we infer that for any $\mathcal{O}_1\subset\mathcal{O}$, which is bounded, and  $\psi: \mathcal{O} \to \R^3$ which is a $\mathrm{C}^1$-class function such that $\mathrm{supp} (\psi ( \cdot)) \subset \mathcal{O}_1,$ and 
	$ \sup\limits_{1\leq i, j\leq 3} \sup\limits_{ x\in  \mathcal{O}_1} |D_i \psi^j (x)| = C < \infty,$ 
	\begin{align}\label{convergence_b*}
		\int_{0}^{T}  b(\u^{N_j}(t), \u^{N_j}(t), \psi) \d t \to \int_{0}^{T} b(\u^*(t), \u^*(t), \psi) \d t\ \text{ as }\ n\to\infty.
	\end{align}
	Since $\u^{N_j},\u^* \in \mathrm{L}^{\infty}(0,T;\H)\cap\mathrm{L}^{2}(0,T;\V)$, we can find a  constant $L>0,$ such that
	\begin{align*}
		\|\u^{N_j}\|_{\mathrm{L}^{\infty}(0,T;\H)}^{\frac12} + \|\u^*\|_{\mathrm{L}^{\infty}(0,T;\H)}^{\frac12} + \|\u^{N_j}\|_{\mathrm{L}^{2}(0,T;\V)}^{\frac32} + \|\u^*\|_{\mathrm{L}^{2}(0,T;\V)}^{\frac32}  \leq L.
	\end{align*}
	Let us choose $\epsilon > 0$. Also, let us choose and fix $\w \in \V$. By a standard regularization method, there exists a function $\psi$ satisfying the above assumptions such that $$\|\w- \psi\|_{\V} < \frac{\epsilon}{6L^2T^{\frac14}}.$$ Making use of \eqref{convergence_b*}, we can find $M_{\epsilon} \in \mathbb{N}$ such that  
	$$\left|\int_{0}^{T} b(\u^{N_j}(t), \u^{N_j}(t), \psi) \d t - \int_{0}^{T} b(\u^*(t), \u^*(t), \psi) \d t\right| < \frac{\epsilon}{3},$$ for all $m \geq M_{\epsilon}.$
	Hence,  for $m > M_{\epsilon},$ using H\"older's and interpolation inequalities, we obtain 
	\begin{align*}
		& \left|\int_{0}^{T} b(\u^{N_j}(t), \u^{N_j}(t), \w)  \d t - \int_{0}^{T} b(\u^*(t), \u^*(t),\w) \d t\right|\\
		%	& \left|\int_{0}^{T} b(\u^{N_j}(t), \u^{N_j}(t), \w(t)) \d t - \int_{0}^{T} b(\u^*(t), \u^*(t), \w(t)) \d t\right|\\
		&	\leq \left|\int_{0}^{T} b(\u^{N_j}(t), \u^{N_j}(t), \w-\psi) \d t\right| + \left|\int_{0}^{T} b(\u^*(t), \u^*(t),\w- \psi) \d t\right| \\
		&\quad+ \left|\int_{0}^{T} b(\u^{N_j}(t), \u^{N_j}(t), \psi) \d t - \int_{0}^{T} b(\u^*(t), \u^*(t), \psi) \d t\right|\\
		&	= \left|\int_{0}^{T} b(\u^{N_j}(t), \w-\psi, \u^{N_j}(t)) \d t\right| + \left|\int_{0}^{T} b(\u^*(t), \w-\psi, \u^*(t)) \d t\right| \\
		&\quad+ \left|\int_{0}^{T} b(\u^{N_j}(t), \u^{N_j}(t), \psi) \d t - \int_{0}^{T} b(\u^*(t), \u^*(t), \psi) \d t\right|
		\\
		&	<  \frac{\epsilon}{3} + \int_{0}^{T} \|\u^{N_j}(t)\|^2_{{\L}^{4}(\mathcal{O})} \|\w-\psi\|_{\V} \d t
		+ \int_{0}^{T} \|\u^*(t)\|_{{\L}^{4}(\mathcal{O})}^2  \|\w-\psi\|_{\V}\d t
		\\
		&	\leq  \frac{\epsilon}{3} + 2 \|\w-\psi\|_{\V} \int_{0}^{T} \|\u^{N_j}(t)\|_{\H}^{\frac12}\|\u^{N_j}(t)\|_{\V}^{\frac32}  \d t
		 + 2 \|\w-\psi\|_{\V} \int_{0}^{T} \|\u^*(t)\|_{\H}^{\frac12}  \|\u^*(t)\|_{\V}^{\frac32} \d t
		\\
		&\leq \frac{\epsilon}{3} + 2 \left[\sup_{t\in [0, T]}\|\u^{N_j}(t)\|_{\H}^{\frac12}   \left(\int_{0}^{T}\hspace{-3mm}\|\u^{N_j}(t)\|^{2}_{\V}\d t\right)^{\frac{3}{4}} 
 +   \sup_{t\in [0, T]}\|\u^*(t)\|_{\H}^{\frac12}   \left(\int_{0}^{T}\hspace{-3mm}\|\u^*(t)\|^{2}_{\V}\d t\right)^{\frac{3}{4}} \right]  \|\w- \psi\|_{\V} {T}^{\frac{1}{4}}
		\nonumber\\&<\epsilon,
	\end{align*}
	which implies that 
\begin{align}\label{LN1}
	\int_{0}^{t}  b(\u^{N_j}(s), \u^{N_j}(s), \w) \d s \to \int_{0}^{t} b(\u^*(s), \u^*(s), \w) \d s \ \text{ as }\ N_j\to\infty,
\end{align}
for all $t\in[0,T]$ and $\w\in \V.$

Our next aim is to show that 
\begin{align}\label{LN}
	& \int_{0}^{T}F_{N_j}(\|\u^{N_j}(s)\|_{\L^4(\mathcal{O})})b(\u^{N_j}(s), \u^{N_j}(s), \w) \d s
	\to \int_{0}^{T} b(\u^*(s), \u^*(s), \w) \d s,
\end{align}
as $N_j\to \infty$.

We consider
\begin{align}\label{LN2}
	& \left| \int_{0}^{T}F_{N_j}(\|\u^{N_j}(s)\|_{\L^4(\mathcal{O})})b(\u^{N_j}(s), \u^{N_j}(s), \w)\d s -\int_{0}^{T}
	 b(\u^*(s), \u^*(s), \w) \d s \right|
	 \nonumber\\ & \leq  \left| \int_{0}^{T}\bigg[F_{N_j}(\|\u^{N_j}(s)\|_{\L^4(\mathcal{O})})-1\bigg]b(\u^{N_j}(s), \u^{N_j}(s), \w) \d s \right| 
	 \nonumber\\ & \quad + \left| \int_{0}^{T}b(\u^{N_j}(s), \u^{N_j}(s), \w)\d s -
	\int_{0}^{T} b(\u^*(s), \u^*(s), \w) \d s \right|
	 \nonumber\\ & \leq \underbrace{\left[\int_{0}^{T}\bigg|F_{N_j}(\|\u^{N_j}(s)\|_{\L^4(\mathcal{O})})-1\bigg|^4\d s \right]^{\frac14}  \left[\int_{0}^{T}\bigg|b(\u^{N_j}(s), \u^{N_j}(s), \w)  \bigg|^{\frac43} \d s\right]^{\frac34} }_{:= \B_1(N_j)}
	 \nonumber\\ & \quad + \underbrace{\left| \int_{0}^{T}b(\u^{N_j}(s), \u^{N_j}(s), \w)\d s -
	\int_{0}^{T} b(\u^*(s), \u^*(s), \w) \d s \right|}_{:= \B_2(N_j)}.
\end{align}
Now, using \eqref{b0}, H\"older's inequality, \eqref{lady} and \eqref{UE1-N}, we obtain 
\begin{align}\label{LN3}
	\int_{0}^{T}\bigg|b(\u^{N_j}(s), \u^{N_j}(s), \w)  \bigg|^{\frac43}  \d s &  \leq  \int_{0}^{T}\|\u^{N_j}(s)\|_{\L^4(\mathcal{O})}^{\frac83} \|\w\|_{\V}^{\frac43}  \d s 
	\nonumber\\ & \leq 2^{\frac43} \|\w\|_{\V}^{\frac43} \int_{0}^{T} \|\u^{N_j}(s)\|_{\H}^{\frac23}\|\u^{N_j}(s)\|_{\V}^{2}   \d s 
	\nonumber\\ & \leq 2^{\frac43} \|\w\|_{\V}^{\frac43} [K_T]^{\frac43}.
\end{align}
In view of Lemma \ref{FN-Lemma}, \eqref{LN3} and  \eqref{LN1},  we obtain \eqref{LN} from \eqref{LN2}, as required. Making use of the convergences in \eqref{S8-N}, \eqref{S8d-N} and \eqref{LN}, we conclude that $\u^*$ is a weak solution to the system \eqref{D-NSE}. 

We know that $\u^N$ satisfies the energy equality 
\begin{align}\label{A18}
		&\|\u^N(t)\|_{\H}^2+2\nu\int_0^t\|\u^N(s)\|_{\V}^2\d s   = \|\x^N\|_{\H}^2 +2\int_0^t\langle\f,\u^N(s)\rangle\d s,
\end{align}
for all $t\in[0,T]$. Using the strong convergence $\x^N\to\x$ strongly in $\H$, the weak convergences \eqref{S7-N} and \eqref{S8-N} and the weakly lowersemicontinuity of  the norms, one can take limit infimum on  both sides of \eqref{A18} to obtain that $\u^*$ satisfies the energy inequality \eqref{A5}. 
This completes the proof.
\end{proof}
\end{appendix}

	\medskip\noindent
	{\bf Acknowledgments:}    K. Kinra would like to thank the Department of Atomic Energy, Government of India, for financial assistance and the Tata Institute of Fundamental Research - Centre for Applicable Mathematics (TIFR-CAM) for providing a stimulating scientific environment and resources. M. T. Mohan would like to thank the Department of Science and Technology (DST) Science \& Engineering Research Board (SERB), India for a MATRICS grant (MTR/2021/000066).

			\medskip\noindent
	\textbf{Data availability:} No data was used for the research described in the article.

			\medskip\noindent
			\textbf{Declarations}: During the preparation of this work, the authors have not used AI tools.
	
	\medskip\noindent
	\textbf{Conflict of interest:} The authors declare no conflict of interest.

\end{document}